%% file: rest_root_abridged.tex
\documentclass{amsart}

\usepackage{amsmath,amsfonts,amssymb,epsfig, verbatim}

\input{rest_root_macros}

\title[Restricted Roots for Spherical Varieties]{Restricted Roots and Restricted Form of 
\\
Weyl Dimension Formula  for  Spherical Varieties
\\
(Abridged Version)
}

\author{Simon Gindikin and Roe Goodman}  



\address{
Department of Mathematics\\
Rutgers University\\
110 Frelinghuysen Rd\\ Piscataway NJ 08854-8019\\
 USA
}

\begin{document}

\maketitle

%         \input{abstract}
%%%%%%%%%%%%%%%%%%%%%%%%%%%%%%%%%%%%%%%%%%%%%%%%%%%%%%%%%%%%%%%%%%%%%%%%%%%%%
%  Created       : Mon Jul 9  2012
%  Last Modified : 7/09/12
%
%%%%%%%%%%%%%%%%%%%%%%%%%%%%%%%%%%%%%%%%%%%%%%%%%%%%%%%%%%%%%%%%%%%%%%%%%%%%%

\begin{abstract}
We study in this paper the restricted roots for a class of spherical
homogeneous spaces of semisimple groups which includes simply connected
symmetric spaces. For these spaces we give a detailed description (case by
case) of the set of roots of the group associated with each restricted root of
the space (the {\em nest} of the restricted root). As an application, we
obtain a refinement of the Weyl dimension formula in the case of spherical
representations, expressing the dimension as a product over the set of
indivisible positive restricted roots.
\end{abstract}

\section{Introduction}
	\label{intro.sec}
%	\input{restrootintro}
%% last revised 9/11/12 for short version

In this paper\footnote {This version contains the statements of the
propositions and lemmas (with proofs omitted) and the root nest calculations
used to obtain the main theorems; see \cite{Gindikin-Goodman} for the complete
paper.} we consider restricted roots for a class of affine spherical
homogeneous spaces $X=G/H$, where $G$ is a semisimple group, $H$ is a
reductive subgroup, and a Borel subgroup of $G$ has an open orbit on $X$. Here
all groups considered are complex linear algebraic groups and all topological
notions refer to the Zariski topology.
%%%%%%%%
We assume that $X$ is ``excellent'' (see Section
\ref{spherepair.sec} for the definition; this class of spherical homogeneous
spaces was introduced in \cite{Vinberg-Gindikin}). All simply-connected symmetric spaces and
all rank-one spaces have this property. When $G$ is simple and $H$ is not the
fixed points of an involution of $G$, we obtain from Kr\"amer's tables
\cite{Kramer} a relatively short list of excellent affine spherical
homogeneous spaces. These spaces exhibit several new phenomena. For example,
by contrast with the case of a symmetric space, the restricted roots are not a
root system in the usual sense. We want to understand how these restricted
roots behave in some problems associated with such spaces. We concentrate in
this paper on constructing a restricted version of Weyl's dimension formula,
and we obtain a refined version of the Plancherel formula.

We recall that Weyl's dimension formula is the product over a system of
positive roots of the group $G$. It is natural to expect that for spherical
representations on $X$ the dimension can be expressed as a product over a
system of positive restricted roots of $X$. However, as far as we know, an
explicit dimension formula of this type has not appeared in print for a
general symmetric space. The ${\bf c}$-function of Harish-Chandra for a
non-compact Riemannian symmetric space has such a product formula, and as a
result the Plancherel density for such a space also has a product formula. For
a symmetric space Vretare \cite{Vretare} and Helgason \cite[Ch. III
\S9.4]{Helgason4} showed that Weyl's dimension formula can be obtained as a
regularization of ${\bf c}$-functions. As a consequence we know that for a
symmetric space a version of Weyl's formula for the dimension of a spherical
representation as a product over restricted roots exists, but the computation
of specific factors corresponding to individual restricted roots is a
substantial work.

There is another approach to this problem which we follow in this paper. To
each restricted root of $X$ we associate a set of roots of $G$ that we call
the ``nest'' of the restricted root. Then in Weyl's formula we combine the
factors from the same nest. We know that in the case of a symmetric space it
is convenient to consider together all scalar multiples of a restricted root,
say $\alpha$ and $2\alpha$, each with a multiplicity. We can associate with
this system an ``atomic'' symmetric space of rank one (which has these roots
and multiplicities). In this paper we show that the spherical dimension
function (in the symmetric case) is a product of some explicit combinatorial
functions, corresponding to the ``atomic'' symmetric spaces of rank one. These
functions for rank one are explicit but not simple. This is similar to the
situation for the rank-one factors in the Plancherel formula for a non-compact
Riemannian symmetric space, which are quite complicated in contrast to the
factors occurring in the ${\bf c}$-function \cite{Gindikin3}.
 
We obtain similar product formulas for non-symmetric excellent affine
spherical homogeneous spaces. The focus here is again on the detailed study of
the nests of restricted roots (and, as a result, of atomic spaces of rank
one). This takes up the major part of the paper. We believe the result can be
useful in other problems, such as the horospherical Cauchy transform
\cite{Gindikin1}, \cite{Gindikin2} (cf. \cite{Goodman}), and that our
dimension formula gives a hint as to how a product formula for the ${\bf
c}$-function for an excellent homogeneous spherical space might look.

There are several new interesting facts that emerge from our investigations.
Some of the atomic spaces of rank one are symmetric spaces of rank one, but
there are two nonsymmetric spherical rank-one spaces (one of them having
restricted roots $\alpha$, $2\alpha$, $3\alpha$). Furthermore, there are some
``virtual'' rank-one spaces: they are not realized as spherical spaces of rank
one, but participate as ``atomic'' spaces in certain excellent affine
spherical homogeneous  spaces of higher rank.

Here is a brief description of the organization of the paper. The main results
concerning dimension formulas are stated in Section \ref{dimformthm.sec}. Some
general results concerning excellent affine spherical spaces and associated
parabolic subgroups are established in Sections \ref{spherepair.sec} and
\ref{parsubgp.sec}.  With these structural properties of excellent affine
spherical pairs established, we turn to the detailed consideration of
restricted roots and the dimension formula in Section \ref{restdim.sec}. We
introduce a principal $\fsl_2$ subalgebra that plays a key role in determining
the shifts in the dimension factors. In the following sections we then work
out all the rank-one cases in detail, followed by the higher-rank
non-symmetric excellent affine spherical homogeneous spaces, and conclude with
the higher-rank symmetric spaces.

\vspace{1ex} 

\noindent{\bf Some Notational Conventions.}

\vspace{1ex}

1. $\bZ_{+}$ is the set of nonnegative integers and $\bC^{\times}$ is the
multiplicative group of the field $\bC$ of complex numbers.

\vspace{.5ex}

2. Denote the $n \times n$ matrix $\bx$ with diagonal entries $x_i \in \bC$
and other entries zero by $\diag [x_1, \ldots, x_n]$. Let $\varepsilon_i$ be
the $i$th coordinate function on the diagonal matrices, so that
$\varepsilon_i(\bx) = x_i$. If $\bx = \diag [x_1, \ldots, x_n]$ and $\by =
\diag [y_1, \ldots, y_n]$ then $\langle \bx \mid \by \rangle = \tr(\bx\by) =
x_1y_1 + \cdots + x_ny_n$.

\vspace{.5ex}

3.  For $\bx = [x_1 , \ldots, x_n ]$ with $x_i \in \bC$ let
  $\check{\bx} = [x_n , \ldots, x_1]$.  

\vspace{.5ex}

4. If $V$ is a complex vector space with dual space $V^{*} = \Hom(V, \bC)$,
then $\langle \cdot \,,\, \cdot \rangle$ denotes the tautological duality
pairing of $V^{*}$ with $V$.

\vspace{.5ex}
 
5. Lie algebras of algebraic groups are denoted by the corresponding German
lower case letters. For an algebraic group $L$ let $\bC[L]$ be the
algebra of regular functions on $L$. Let $\frX(L) = \Hom(L, \bC^{\times})$ be
the character group of $L$ (written additively); the value of $\lambda \in
\frX(L)$ on $y\in L$ will be denoted by $y^{\lambda}$. If $V$ is an $L$
module, then $V^L$ denotes the subspace of $L$-fixed vectors.

\vspace{.5ex}

6. Let $G$ be a semisimple simply-connected algebraic group over $\bC$. Fix
a choice of Borel subgroup $B \subset G$ and a choice of maximal torus $T
\subset B$. Let $U$ be the unipotent radical of $B$. Then $B = TU$ and
$(tu)^{\lambda} = t^{\lambda}$ for $t\in T$, $u \in U$, and $\lambda \in
\frX(B)$, so we may identify $\frX(B)$ with $\frX(T)$.

\vspace{.5ex}

7. The set of dominant weights of $B$ is denoted by $\frX_{+}(B)$. Let
$\varpi_1,\ldots, \varpi_{\ell}$ be the fundamental dominant weights, where
$\ell = \rank(G)$. Let $\lambda = k_1\varpi_1 + \cdots +
k_{\ell}\varpi_{\ell}$ with $k_i \in \bZ_{+}$ be a dominant weight. The {\em
support} of $\lambda$ is the set 
$ \Supp \lambda = \{\varpi_i \,:\, k_i > 0 \}$.

\vspace{.5ex}

8. For each $\lambda \in \frX_{+}(B)$ there is an irreducible
finite-dimensional rational $G$-module $E_{\lambda}$ with highest weight
$\lambda$. The action of $g \in G$ on $\bx \in E_{\lambda}$ is denoted by
$g\cdot \bx$. Write $\lambda^{*}$ for the highest weight of the dual
representation $(E_{\lambda})^*$. Fix a highest weight vector $\be_{\lambda}
\in E_{\lambda}$; thus $b\cdot\be_{\lambda} = b^{\lambda}\,\be_{\lambda}$ for
$b\in B$.

\section{Restricted Weyl Dimension Formula}
       \label{dimformthm.sec}
%       \input{dimformthm}
%% revised 9/11/12 for short version

To state our theorems concerning dimension formulas, we introduce the
following functions. For $t = m/2$ with $m$ a nonnegative integer let
$\varphi(x \,;\, t)$ be the monic polynomial of degree $2t+1$ in $x$ whose
zeros are at $t$, $t-1$, $\ldots$\,, $-t+1$, $-t$. Thus $\varphi(x \,;\, 0) =
x$ and 
\begin{equation}
\label{wdimpoly}
  \varphi(x \,;\, t ) =   (x-t)(x-t+1) \cdots (x+t-1)(x+t) 
\end{equation}
when $t > 0$. This polynomial arises naturally as the characteristic
polynomial for the matrix $\begin{bmatrix}\frac{1}{2} & 0 \\ 0 & -\frac{1}{2}
\end{bmatrix} \in \fsl(2,\bC)$  in the irreducible representation of dimension
$2t+1$.
We extend $\rule{0ex}{2ex} \varphi(x \,;\, t)$ to be a meromorphic function
of $x$ and $t$ by setting $\varphi(x \,;\, t ) = \Gamma(x+t+1)/\Gamma(x-t)$.
Define 
\begin{equation}
\label{weyldimfun}
  \Phi(x,y \,;\, t) = \frac{\varphi(x+y \,;\, t)}{\varphi(y \,;\, t)}\,.
\end{equation}
This is a meromorphic function of $x$, $y$, $t$ that is normalized to
satisfy $\Phi(0, y \,;\, t) = 1$. We write $\Phi(x, y) = \Phi(x,y \,;\, 0) =
(x+y)/y$.

The {\em regular} Weyl dimension functions are defined as follows; here $m$
(the {\em multiplicity parameter}) is a positive integer subject to the
additional conditions indicated. 
\begin{equation}
\label{regdimfact1}
  W(x, y \,;\, m) = 
  \begin{cases} 
     \Phi(x,y) 
     &\quad\text{if $m = 1$\,,}
     \\
     \Phi(x,y)^2 
     &\quad\text{if $m = 2$\,,}
     \\
       \Phi(x,y \,;\, 1) 
        &\quad\text{if $m = 3$\,,}
      \\
      \Phi(x,y)\,\Phi\big(x,y \,;\, \frac{1}{2}m - 1\big) 
        &\quad\text{if $m \geq 4$\,,}
    \end{cases}      
\end{equation}
when there is a single multiplicity parameter, and 
\begin{equation}
\label{regdimfact2}
\left\{
\begin{split}
  W(x, y \,;\, m, 1) &= \Phi(x,y) \, 
 \left\{ 
  \Phi\big(x,y \,;\, \textstyle{\frac{1}{4}m - \frac{1}{2}}\big) \right\}^2 
          \quad\mbox{if $m \geq 2$ is even\,,}
 \\
  W(x, y \,;\, m, 3) &=  
                \frac{ \rule[-1ex]{0ex}{2.5ex}
               \Phi\big(x, y \,;\, \frac{1}{4}m - \frac{1}{2}\big)
               \, \Phi\big(x, y \,;\, \frac{1}{4}m + \frac{1}{2}\big)
          }{
       \rule{0ex}{2.5ex} \Phi\big(x, y \,;\, \frac{1}{2} \big)
            } 
              \,  \Phi(2x, 2y \,;\, 1) 
 \\
    &\hspace{32ex}
        \mbox{if $m \geq 2$  is even\,,}
 \\
   W(x, y \,;\, 8, 7) &= 
         \Phi(x,y)\,      
         \Phi\big(2x,2y \,;\, \textstyle{\frac{3}{2}}\big)\, 
        \Phi\big(2x, 2y \,;\, \textstyle{\frac{9}{2}}\big) \,,
\end{split}
\right.
\end{equation}  
\begin{equation}
\label{regdimfact3}
\left\{
\begin{split}
   W(x, y \,;\, 3, 3) &= \Phi(x,y \,;\, 1)\, \Phi(2x, 2y \,;\, 1) \,,
 \\
  W(x, y \,;\, 2, 1, 2) &= 
       \Phi\big(x,y \,;\, \textstyle{\frac{1}{2}}\big)\, 
       \Phi(2x, 2y)\,
       \Phi\big(3x,3y \,;\, \textstyle{\frac{1}{2}}\big) \,,
\end{split}
\right.
\end{equation}  
when there are two or three multiplicity parameters. These functions of $x$,
$y$ occur in the dimension formulas for rank-one affine spherical spaces,
with the first parameter the multiplicity of the indivisible restricted root
$\xi$. The second and third parameters are the multiplicity of $2\xi$ and
$3\xi$ (when these multiplicities are nonzero). The original Weyl dimension
formula is expressed in terms of the function $W(x, y\,;\, 1) = \Phi(x,y)$,
whereas the functions in (\ref{regdimfact3}) occur for non-symmetric spherical
spaces of rank one.

The {\em singular} Weyl dimension functions are defined as follows.
\begin{equation}
\label{singdimfact}
\left\{
\begin{split}
  W_{\rm sing}(x, y \,;\, m) &= 
     \Phi\big(x,y \,;\, \textstyle{\frac{1}{2}m - \frac{1}{2}}\big) \,,
\\
  W_{\rm sing}(x, y \,;\, m, 1) &= 
     \Phi\big(x,y \,;\, \textstyle{\frac{1}{2}m }\big) 
 \quad\mbox{if $m$ is even} \,.
\end{split}
\right.
\end{equation}
These functions only occur in the dimension formulas for some excellent
non-symmetric spherical homogeneous spaces of rank greater than one.

When a multiplicity parameter is  zero, we omit it from the
notation; thus we write
\[
\begin{split}
  W(x,y \,;\,m, 0) &= W(x,y \,;\, m)\,,
\\
 W(x,y \,;\,  m, n, 0) &=  W(x,y \,;\, m, n)
\quad (n = 1, 3, 7)\,,
\\  
  W_{\rm sing}(x,y \,;\,m, 0) &= W_{\rm sing}(x,y \,;\, m)\,. 
\end{split}
\]
With the indicated restrictions on $m$ all these dimension functions are
polynomials in $x$ and rational functions of $y$. They are normalized to take
the value $1$ when $x = 0$.

Assume that $G/K$ is an irreducible simply-connected symmetric space. Fix a
Cartan subspace $\fra \subset \frp$, where $\frg = \frk + \frp$ is the Cartan
decomposition corresponding to the involution. Then $\fra \subset \frt$ where
$\frt$ is the Lie algebra of a maximal torus of $\frg$, and the roots of
$\frt$ on $\frg$ can be restricted to $\fra$. Fix a set of positive restricted
roots $\Sigma^{+} \subset \fra^{*}$ and let $\Sigma^{+}_0$ be the indivisible
positive roots. For $\xi \in \Sigma^{+}_0$ let $m_{\xi}$ and $m_{2\xi}$ be the
associated root multiplicities, and let
\[
  \delta = \frac{1}{2}\sum_{\xi \in \Sigma^{+}_0} 
    (m_{\xi} + 2m_{2\xi}) \xi \,.
\]
Let $\langle \lambda \mid \xi \rangle$ be the bilinear form on $\fra^{*}$
obtained by duality from the restriction to $\fra$ of a positive multiple of
the Killing form of $\frg$ (the appropriate normalization of the form is
described in Section \ref{restdim.sec}). 

\begin{Theorem}
\label{symspacedim.thm}
The finite-dimensional irreducible $K$-spherical representation of $G$ with
highest weight $\lambda$ has dimension
\begin{equation}
\label{symspacedim}
  d(\lambda) = 
 \prod_{\xi \in \Sigma_{0}^{+}} 
 W(\langle \lambda  \mid \xi \rangle,
   \langle  \delta \mid \xi \rangle \,;\, m_{\xi}\,, m_{2\xi}\,)\,.
\end{equation}
\end{Theorem}

For rank-one symmetric spaces we prove Theorem \ref{symspacedim.thm} in
Section \ref{rankone.sec} using classification and the
results from Section \ref{restdim.sec}. Helgason's formula for $d(\lambda)$ in
terms of a regularization of ratios of ${\bf c}$-functions \cite[Ch. III
\S9.4]{Helgason4} suggests that the result should then hold in higher rank.
However, in the rank-one case the multiplicity $m_{\xi}$ determines the value
of $\langle \delta \mid \xi \rangle$; this is not true in rank greater than
one, and it seems necessary to do a case-by-case argument for higher rank to
obtain the explicit factors corresponding to each positive restricted root. We
carry this out in Section \ref{symspaceexam.sec} using
techniques similar to those for rank-one symmetric spaces. By Weyl group
symmetry of the restricted root system, however, we only need to consider the
simple restricted roots in most cases. The necessary information about the
restricted roots is summarized in a {\em marked Satake diagram} (the Satake
diagram as in \cite[Ch. X, Table VI]{Helgason1} with additional labels on
certain vertices) and a table of root data in each case. 

\begin{Remark} 
{\em
If $\xi \in \Sigma_{0}^{+}$ then by Cartan's classification of symmetric
spaces the only possible values for $m_{2\xi}$ are $0$, $1$, $3$, and $7$, and
$m_{3\xi} = 0$. Furthermore, when $m_{2\xi} \neq 0$ then $m_{\xi}$ is even.
Thus all the dimension functions in formula (\ref{symspacedim}) are defined in
(\ref{regdimfact1}) and (\ref{regdimfact2}).
}
\end{Remark}

For an irreducible simply-connected excellent spherical space that is not
symmetric there is an analogue of the subspace $\fra$ that was introduced in
\cite{Brion2}, and there is a corresponding set $\Sigma$ of restricted roots
(although this set is not a root system). As a consequence of our dimension
formulas we can separate the indivisible positive roots $\Sigma_0^+$ into {\em
regular} and {\em singular} roots:
\[
   \Sigma_{0}^{+} =  
    \Sigma_{\rm reg}^{+} \cup \Sigma_{\rm sing}^{+}\,.
\]
By definition, an indivisible positive restricted root is called {\em regular}
if its dimension function occurs in a rank-one affine spherical space; these
functions are given in (\ref{regdimfact1}), (\ref{regdimfact2}), and
(\ref{regdimfact3}). Otherwise, the root is called {\em singular}. The
dimension functions for singular roots are given in (\ref{singdimfact}).
With this terminology established, we can state our second main result.

\begin{Theorem}
\label{spherepacedim.thm}
Assume that $G$ is simple and simply-connected, $H$ is reductive and
connected, and $G/H$ is an excellent spherical space that is not symmetric.
The finite-dimensional irreducible $H$-spherical representation of $G$ with
highest weight $\lambda$ has dimension
\begin{equation}
\label{spherespacedim}
\begin{split}
  d(\lambda) &= 
 \prod_{\xi \in \Sigma_{\rm reg}^{+}} 
 W\big(\langle \lambda  \mid \xi \rangle,
 \langle  \delta \mid \xi \rangle \,;\, m_{\xi}\,, m_{2\xi}\,, m_{3\xi}\,\big)
\\
 &\qquad\times
 \prod_{\xi \in \Sigma_{\rm sing}^{+}} 
 W_{\rm sing}\big(\langle \lambda  \mid \xi \rangle,
   \langle  \delta \mid \xi \rangle \,;\, m_{\xi}\,, m_{2\xi}\,\big)\,.
\end{split}
\end{equation}
(If $\rank G/H = 1$ then $\Sigma_{\rm sing}^{+}$ is empty.)
\end{Theorem}

We prove Theorem \ref{spherepacedim.thm} in Section \ref{highrank.sec} using
methods similar to those for symmetric spaces. Determination of the root nests
and dimension factors requires more calculation in this case because there is
no Weyl group action on the restricted roots. The necessary information about
the restricted roots in each case is summarized in a marked Satake diagram (as
in the symmetric case) and a table of root data.

\begin{Remark}
{\em
For a singular root $\xi$ the multiplicity of $2\xi$ turns out to be either
zero or one, and the multiplicity of $3\xi$ is zero. Thus the dimension
functions in (\ref{spherespacedim}) are all defined in (\ref{regdimfact1}),
(\ref{regdimfact2}), (\ref{regdimfact3}), and (\ref{singdimfact}), with the
convention that 
}
\[
 W\big(\langle \lambda  \mid \xi \rangle,
 \langle  \delta \mid \xi \rangle \,;\, m_{\xi}\,, m_{2\xi}\,, 0\,\big)
 =
 W\big(\langle \lambda  \mid \xi \rangle,
   \langle  \delta \mid \xi \rangle \,;\, m_{\xi}\,, m_{2\xi}\,\big)\,.
\]
\end{Remark}
 
\begin{Remark}
{\em For non-symmetric excellent affine spherical homogeneous spaces of rank
greater than one, the calculations in Sections \ref{highrank.sec} show that
regular and singular restricted roots can have the same multiplicities. Hence
the dimension functions for these spaces are not completely determined just by
the restricted roots and their multiplicities, unlike the case of rank one
spaces or higher rank symmetric spaces. 
}
\end{Remark}

\section{Spherical Pairs}
		\label{spherepair.sec}
%		\input{spherepair}
%% revised 9/11/12 for short version

%%%%%%%%%%%%%%%%%%%%%

Let $G$ be a simply-connected semi-simple complex algebraic group and $H$ an
algebraic subgroup of $G$. The pair $(G, H)$ (and by extension the homogeneous
space $G/H$ and the subgroup $H \subset G$) is called {\em spherical} if $B$
has an open orbit on the variety $G/H$. The existence of such an orbit implies
that $\dim E_{\lambda}^H \leq 1$ for all $\lambda \in \frX_{+}(B)$. If $G/H$
is quasi-affine (such a subgroup $H$ is called {\em observable}), then the
converse is true \cite[Theorem 1]{Vinberg-Kimelfeld}. Since all Borel
subgroups of $G$ are conjugate, the notion of spherical pair does not depend
on the choice of $B$.

Assume that $(G, H)$ is a spherical pair. If $\lambda \in \frX_{+}(B)$ and
$E_{\lambda^{*}}^H \neq 0$, then $\lambda$ will be called an $H$-{\em
spherical} highest weight and $E_{\lambda}$ an $H$-{\em spherical}
representation (thus $\bC[G/H]$ contains $E_{\lambda}$ as a submodule in this
case). Following \cite{Avdeev1} we let $\Gamma(G/H)$ denote the set of
$H$-spherical highest weights for $G$. Then $\Gamma(G/H)$ is a subsemigroup of
$\frX_{+}(B)$. If $H$ is reductive, then for $\lambda \in \frX_{+}(B)$ we have
$E_{\lambda}^H \neq 0$ if and only if $E_{\lambda^*}^H \neq 0$. Hence
$\Gamma(G/H)$ is invariant under the map $\lambda \mapsto \lambda^{*}$ in this
case.

The following class of spherical pairs was introduced in
\cite{Vinberg-Gindikin} (cf. \cite{Avdeev2}). 

\begin{Definition}
\label{excellent.def}
{\em
The spherical pair $(G, H)$ 
is {\em excellent} if $G/H$ is quasi-affine and
$\Gamma(G/H)$ is generated by 
 $\mu_{1}, \ldots, \mu_{r}$ with
 $\Supp \mu_i \cap \Supp \mu_j = \emptyset$ for $i \neq j$\,.
}
\end{Definition}

When $(G, H)$ is an excellent spherical pair, then the support condition
implies that $\{\mu_{1}, \ldots, \mu_{r} \}$ is linearly independent and
$\Gamma(G/H)$ is a free semigroup. For example, from \cite[Ch. II, Prop.
4.23]{Helgason4} one knows that $(G, H)$ is excellent when $G$ is
simply-connected and $H$ is any symmetric subgroup of $G$ (the fixed-point
group of an involutive automorphism of $G$). Here $G/H$ is affine because $H$
is reductive.

\section{Parabolic Subgroups for Excellent Affine  Spherical Pairs}
	  	\label{parsubgp.sec}
%		\input{parsubgp}
%%%%%%%%%%%%%%%%%%%%%%%%%%%%%%%%%%%%%%%%%%%%%%%%%%%%%%%%%%%%%%%%%%%%%%%%%%%%%
%  Created       : Sat Apr 2 16:58:29 2011
%  Last Modified : 9/11/12 
%  for short version of paper
%%%%%%%%%%%%%%%%%%%%%%%%%%%%%%%%%%%%%%%%%%%%%%%%%%%%%%%%%%%%%%%%%%%%%%%%%%%%%

For the rest of the paper we assume that $(G, H)$ is an excellent spherical
pair with $G$ simply connected and simple, $H$ connected and reductive (the
list of such pairs with $H$ not a symmetric subgroup of $G$ is given in
Sections \ref{rankone.sec} and \ref{highrank.sec}). Fix a Borel subgroup
$B$ in $G$. Let $\mu_1, \ldots, \mu_r$ satisfy the conditions of Definition
\ref{excellent.def}. The integer $r$ is the {\em spherical rank} of the pair
$(G, H)$. 

For a vector space $V$ let $\bP(V)$ be the associated projective space, and
denote the canonical map from $V\setminus\{0\}$ to $\bP(V)$ by $\bx \mapsto
[\bx] = \bC^{\times}\cdot\bx$. Define
\[
  P = \{ g \in G \;:\; [ g\cdot {\be}_{\mu_i}] = [{\be}_{\mu_i}]
\mbox{ \ for  $i = 1, \ldots , r$} \}\;. 
\]
Then $P$ is a parabolic subgroup of $G$ since it contains $B$. 

We can describe the structure of $P$ as follows (see, e.g., \cite[\S
30.2]{Humphreys} and \cite{Vinberg-Popov}). Let $\Phi$ be the roots of $T$ on
$\frg$ and let $\Phi^{+}$ be the positive roots determined by the Borel
subgroup $B$. Let $\Delta$ be the simple roots in $\Phi^{+}$. For $\alpha \in
\Phi^{+}$ let $h_{\alpha} \in \frt$ be the coroot to $\alpha$. There is a
unique regular homomorphism $\psi_{\alpha}: \SL(2, \bC) \to G$ whose
differential $d\psi_{\alpha}:\fsl(2,\bC) \to \frg$ satisfies
 \[
   d\psi_{\alpha}\begin{bmatrix} 1 & 0 \\ 0 & -1 \end{bmatrix} = h_{\alpha}\,,
 \quad
   d\psi_{\alpha} \begin{bmatrix} 0 & 1 \\ 0 &0 \end{bmatrix} 
  \in \frg_{\alpha} \,,
 \quad
  d\psi_{\alpha} \begin{bmatrix} 0 & 0 \\ 1 &0 \end{bmatrix} 
  \in \frg_{-\alpha}\,.
 \]  
Write $G^{(\alpha)}$ for the image of $\psi_{\alpha}$; this is a closed
subgroup of $G$. 

Define
\[ 
   \Delta_{0} = 
    \{ \alpha \in \Delta \;:\; \langle \mu_i \,,\, h_{\alpha} \rangle = 0 
    \mbox{ \ for $i = 1, \ldots, r$ } \} \;. 
\]
Thus $\Delta_{0}$ consists of the simple roots $\alpha$ such that $h_{\alpha}$
acts by zero on $\be_{\lambda}$ for all $\lambda \in \Gamma(G/H)$. By the
representation theory of $\SL_2$, one knows that $\alpha \in \Delta_{0}$ if
and only if $G^{(\alpha)}$ fixes $\be_{\lambda}$ for all $\lambda \in
\Gamma(G/H)$.

Viewing the elements of $\Delta_{0}$ as characters of $T$, we define
\[
 C = \Big(\bigcap_{\alpha \in \Delta_{0}} \Ker(\alpha)\Big)^{\circ} \;,
\]
where $K^{\circ}$ denotes the identity component of an algebraic group $K$.
Then $C$ is a subtorus of $T$, and elements of $C$ commute with $G^{(\alpha)}$
for all $\alpha \in \Delta_{0}$. The Lie algebra of $C$ is
\[
  \frc = \  \{ x \in \frt \::\; 
 \langle \alpha \,,\, x \rangle = 0 \mbox{ for all $\alpha \in \Delta_{0}$ } \} \,.
\]

Define

\begin{equation}
\label{levi_factor}
\\
  L = \{ g \in G \;:\;  gc = cg \mbox{ \
   for all $c \in C$} \} \,.
\end{equation}
Then $L$ contains the subgroups $G^{(\alpha)}$ for all $\alpha \in
\Delta_{0}$. Since $C$ is a torus, one knows that $L$ is a connected reductive
group containing $T$, and that $P = LN$ (Levi decomposition), where $N$ is the
unipotent radical of $P$. Furthermore, $C$ is the identity component of the
center of $L$. The Lie algebras of $L$ and $N$ are
\begin{align}
\label{levi}
  \frl =& \ \frc + \sum_{\alpha \in \Delta_{0}} \bC h_{\alpha} 
     + \sum_{\beta \in \Psi} \frg_{\beta} \,,
\\
\label{nilrad}
 \frn =& \ \sum_{\alpha \in \Phi^{+} \setminus \Psi}  \frg_{\alpha} \,,
\end{align}
where $\Psi = \big(\Span \Delta_{0}\big) \cap \Phi$ is the root system
with simple roots $\Delta_{0}$.

Define $M$ to be the subgroup of $L$ that fixes all the highest weight vectors
$\be_{\mu_i}$ for $i = 1, \ldots, r$. Then $M$ is reductive (but not
necessarily connected). Let $M'$ be the commutator subgroup of $M$ 
and let
$   C_0 =  M \cap C $.
The Lie algebras are
\begin{align}
\label{liec0}
 \frc_0 =& \ \{ Y \in \frc \;:\; \langle \mu_i \,,\, Y \rangle = 0
   \mbox{ \  for $i = 1,\ldots, r$}  \} \,,
\\
\label{liem'}
 \frm' =& \   \sum_{\alpha \in \Delta_{0}} \bC h_{\alpha} 
     + \sum_{\beta \in \Psi} \frg_{\beta} \,,
\\
\label{liem}
  \frm =& \  \frc_0 + \frm' \,.
\end{align}

Following \cite{Brion2}, we let $\fra$ be the orthogonal complement to
$\frc_0$ in $\frc$ relative to the Killing form on $\frt$. Since the Killing
form is positive definite on the real span $\frt_{\bR}$ of the coroots, and
since $\frc$ and $\frc_0$ are complexifications of real subspaces of
$\frt_{\bR}$, we have $\frc = \fra \oplus \frc_0$. Thus $\frl = \fra \oplus
\frm$ as a Lie algebra and $\frl' = \frm'$.
Furthermore,  $\frp = \frm \oplus \fra \oplus \frn$ as a vector
space and  
\begin{equation}
\label{nbarman}
  \frg = \frn^{-} \oplus \frm \oplus \fra \oplus \frn 
\quad\mbox{with \quad } 
 \frn^{-} = \sum_{\beta \in \Phi^{+}\setminus \Psi} \frg_{-\beta}\,.
\end{equation} 
 From (\ref{nbarman}) it follows that 
\begin{equation}
\label{zerowtspace}
  \frl = \{ X \in \frg \,:\, [X, \fra ] = 0 \}\;.
\end{equation}

\begin{Lemma}
\label{rankdim.lem}
Let $d = |\bigcup_{i=1}^{r}\Supp \mu_i |$. Then $\dim \frc = d$ and 
$\dim \frc_0 = d - r$. 
Hence $\dim \fra = r$. In particular, if \ 
$\rule{0ex}{2.5ex} |\Supp \mu_i | = 1$ for all $i$,
then $\fra = \frc$.
\end{Lemma}

For proof, see \cite{Gindikin-Goodman}

\section{Restricted Roots and Dimension Factors}
		\label{restdim.sec}
%		\input{restdim}
%%%%%%%%%%%%%%%%%%%%%%%%%%%%%%%%%%%%%%%%%%%%%%%%%%%%%%%%%%%%%%%%%%%%%%%%%%%%%
%  Created       : Tue Apr 12 10:36:16 2011
%  Last Modified : 9/11/12 for short version
%
%%%%%%%%%%%%%%%%%%%%%%%%%%%%%%%%%%%%%%%%%%%%%%%%%%%%%%%%%%%%%%%%%%%%%%%%%%%%%

The spherical subgroup $H$ defines a partition of the root system of $\frg$ as
$\Phi = \Psi \cup (\Phi \setminus \Psi)$, where we recall that $\Psi$ is the
root system of $\frm$ (these are the roots whose restriction to $\fra$ is
zero) and $\Delta_0 = \Psi \cap \Delta$ is a set of simple roots for $\Psi$.
Define the set of {\em restricted roots} $\Sigma$ to be the restrictions of
the roots in $\Phi \setminus \Psi$ to $\fra$.\footnote{
When $H$ is not a symmetric subgroup of $G$, the set $\Sigma$ is usually
not a root system in $\fra_{\bR}^{*}$.}
For $\lambda \in \frt^{*}$ we write $\overline{\lambda}$ for the restriction
of $\lambda$ to $\fra$. 

Let $\Sigma^{+}$ be the set of restrictions to $\fra$ of the
roots in $\Phi^{+} \setminus \Psi$.  For $\xi \in \Sigma^{+}$ define
\begin{equation}
\label{rootnest}
  \Phi^{+}(\xi) = \{ \alpha \in \Phi^{+}\setminus \Psi \,:\, 
    \overline{\alpha} = \xi \}\,.
\end{equation}
We call $\Phi^{+}(\xi)$ the {\em nest} of roots for $\xi$. 
 We define
\[
  \frn_{\xi} = \sum_{\alpha \in \Phi^{+}(\xi)} \frg_{\alpha}
\]
(the $\xi$ eigenspace of $\fra$ in $\frn$). The {\em multiplicity} of $\xi$ is
\[
  m_{\xi} = \dim \frn_{\xi} = \left|\Phi^{+}(\xi)\right|.
\]  
If $\alpha \in \Phi^{+}(\xi)$, $\beta \in \Psi$, and $\alpha + \beta \in
\Phi$, then $\alpha + \beta \in \Phi^{+}(\xi)$. Hence the subspace
$\frn_{\xi}$ is invariant under the adjoint action of $\frm$, and there is a
decomposition
\[ 
    \frn = \bigoplus_{\xi \in \Sigma^{+}} \frn_{\xi} 
\]
as a module relative to the adjoint action of  $\frl = \frm \oplus \fra$.

Let $\langle \, \cdot \mid \cdot \, \rangle$ be a positive multiple of the
Killing form on $\frt$, which we use to identify $\frt$ with $\frt^{*}$ and to
identify $\fra^{*}$ with a subspace of $\frt^{*}$. Then $\Psi \perp \fra^{*}$
and $\fra^{*} = \Span \Gamma(G/H)$.
We normalize this form to make $\langle\, \alpha \mid \alpha \,\rangle = 2$
for $\alpha \in \Delta_0$ when these roots all have the same length. When
$\frm'$ has roots of two lengths, then it follows by the
classification of simple Lie algebras that these roots occur in only one
simple ideal, say $\frq$, of $\frm'$. Thus we can normalize the form so that
when $\frq$ is of type $B$ ({\em resp.} type $C$) then the long roots ({\em
resp.} short roots) in $\Delta_0$ have squared length two.
 Given $\alpha \in \Phi$, we write
\[
   \alpha^\vee = \frac{2}{\langle \, \alpha \mid \alpha \, \rangle}\alpha
\]
for the coroot to $\alpha$.  Define
\[
   \rho_{\frg} = \frac{1}{2} \sum_{\alpha\in \Phi^{+} } \alpha \,,
\qquad
   \rho_{\frm} = \frac{1}{2} \sum_{\beta \in \Psi^{+} } \beta \,,
\qquad
   \delta  = \frac{1}{2} \sum_{\xi \in \Sigma^{+} } 
   m_{\xi} \, \xi\,.
\]
Then $\delta \perp \rho_{\frm}$. Since $\rho_{\frm}$ is the sum of the
fundamental highest weights of $\frm$, we have
\begin{equation}
\label{rhoform}
   \langle \, \rho_{\frm} \mid \alpha^\vee \, \rangle = 1
\quad\mbox{for all $\alpha \in \Delta_0$\,.}   
\end{equation}

\begin{Lemma}
\label{rhodecomp.lem}
There is an orthogonal decomposition
\begin{equation}
\label{rhosum}
  \rho_{\frg} = \delta + \rho_{\frm}\,.
\end{equation}
\end{Lemma}

For proof see \cite{Gindikin-Goodman}.

\vspace{1ex}

Let $\lambda \in \Gamma(G/H)$. For $\xi \in \Sigma^{+}$ define 
\[
   d_{\xi}(\lambda)  = 
   \prod_{\alpha \in \Phi^{+}(\xi)} 
   \frac{\langle \, \lambda + \rho_{\frg} \mid \alpha \, \rangle}
   {\langle \, \rho_{\frg} \mid \alpha \, \rangle}  \,.
\]
Let $d(\lambda)$ be the dimension of the irreducible $G$-module $E_{\lambda}$.
Then the Weyl dimension formula gives 
\begin{equation}
\label{restdim} 
  d(\lambda) =  \prod_{\xi \in \Sigma^{+}} d_{\xi}(\lambda)\,.
\end{equation} 
For $\alpha \in \Phi^{+}(\xi)$ we have 
$\langle\, \lambda \mid \alpha\,\rangle = 
    \langle\,\lambda \mid \xi \,\rangle$.
Hence by (\ref{rhosum})  we can write
\begin{equation}
\label{weyldimfact}
   d_{\xi}(\lambda)  =
   \prod_{\alpha \in \Phi^{+}(\xi)} 
     \frac{ \langle\, \lambda + \delta \mid  \xi \,\rangle
     +  \langle\,  \rho_{\frm} \mid  \alpha \,\rangle }
     { \langle\,  \delta \mid  \xi \,\rangle
     +  \langle\,  \rho_{\frm} \mid  \alpha \,\rangle }  \,.
\end{equation}

We now determine the shifts $\langle\, \rho_{\frm} \mid \alpha \,\rangle$ in
formula (\ref{weyldimfact}) using the representations of $\fsl_2$. Let
$h_{\frm}^0$ be the element of 
$\Span \{ h_{\alpha} \,:\, \alpha \in \Delta_0\}$ such that
\[
  \langle\, h_{\frm}^0 \mid \alpha \,\rangle = 2
\quad\mbox{for all $\alpha \in \Delta_0$\,.}
\]
Then $h_{\frm}^0$ is a regular element in $\frm'$, and there exist elements
$e_{\frm}^0$ and $f_{\frm}^0$ in $\frm'$ such that 
$\rule{0ex}{2.5ex} \{e_{\frm}^0\,,\, f_{\frm}^0 \,,\, h_{\frm}^0
\}$ 
is a principal $\fsl_2$ triple in $\frm'$ (see \cite{Kostant} and \cite[Ch.
VIII, \S11]{Bourbaki4}). Denote the span of these elements by $\frs$ and let
$S \subset M$ be the connected subgroup with Lie algebra $\frs$.

Suppose $g \in G$ and $\Ad(g)\fra = \fra$. If $\xi$ is a restricted root, then
$g \cdot \xi$ is also a restricted root, where $g\cdot \xi \in \fra^*$ is
defined by $\langle g\cdot \xi \mid x \rangle = \langle \xi \mid \Ad(g)^{-1}x
\rangle$ for $x\in \fra$.

\begin{Lemma}
\label{principaltds.lem}
If $g \in G$ and $\Ad(g)$ preserves $\fra$, then the restricted
root spaces $\frn_{\xi}$ and $\frn_{g\cdot\xi}$ are isomorphic as
$\frs$-modules.
\end{Lemma}

For proof see \cite{Gindikin-Goodman}.

\vspace{1ex}

Let $\xi \in \Sigma^{+}$ be a restricted positive root. Define
\[
   k_{\xi} = \max \{ k \,:\, 
    \mbox{ $k$ is an eigenvalue of $\ad  h_{\frm}^0$ on $\frn_{\xi}$ } \}\,. 
\]   
By the representation theory of $\fsl_2$ we know that $k_{\xi}$ is a
non-negative integer and 
\begin{equation}
\label{kxi}
  k_{\xi} = - \min \{ \langle\, h_{\frm}^0 \mid \alpha \,\rangle \,:\, 
    \alpha \in  \Phi^{+}(\xi)\}\,.
\end{equation}
\begin{Definition}
{\em
A root $\alpha \in \Phi^{+}(\xi)$ is a {\em basic root} if $\alpha$ gives the
minimum value in (\ref{kxi}).
}
\end{Definition}

Let $\Phi(x,y \,;\, t)$ be the function defined in (\ref{weyldimfun}); recall
that we write $\Phi(x, y) = \Phi(x, y \,;\, 0)$.

\begin{Proposition}
\label{rhoshift.prop}  
The eigenvalues of $\ad h_{\frm}^0$ on $\frn_{\xi}$ are integers between
$-k_{\xi}$ and $k_{\xi}$. They include $-k_{\xi}$\,,\,$-k_{\xi} +
2$\,,\,$\ldots$\,,\,$k_{\xi} - 2$\,,\,$k_{\xi}$. In particular, $\dim
\frn_{\xi} \geq k_{\xi} + 1$.

Let $b = \frac{1}{2}k_{\xi}$, let $\lambda \in \Gamma(G/H)$, and assume all
roots in $\Delta_0$ have the same length. Then the following hold.
\begin{enumerate}

\vspace{.5ex}
 
\item
% (1)
The shifts $\langle\, \rho_{\frm} \mid \alpha \,\rangle$ in {\rm
(\ref{weyldimfact})} are the eigenvalues (with multiplicities) of
\,$\frac{1}{2}\ad h_{\frm}^0$\, on $\frn_{\xi}$. In particular, the shifts
include  $-b$, $-b+1$, $\ldots$\,, $b-1$, $b$.

\vspace{1ex}

\item
% (2)
Suppose \ $b = 0$\,. Then 
\ $
  d_{\xi}(\lambda) =
  \left[\Phi(\langle\, \lambda \mid \xi \,\rangle \,,\,
  \langle\, \delta \mid \xi \,\rangle ) \right]^{m_{\xi}}
  $\,.

\vspace{1ex}

\item
% (3)
Suppose $b > 0$\,, \ $m_\xi = (2b + 1)p$ \ for some integer $p \geq 1$,
and $\Phi^{+}(\xi)$ has $p$ basic roots. Then \ 
$
  d_{\xi}(\lambda) = \left[
  \Phi(\langle\, \lambda  \mid \xi \,\rangle \,,\,
   \langle\, \delta \mid \xi \,\rangle \,;\, b)
  \right]^p \,.
$
\item
% (4)
Suppose $b > 0$ and  $m_\xi = 2b + 2$. Then
\[
  d_{\xi}(\lambda) = \Phi(\langle\, \lambda  \mid \xi \,\rangle \,,\,
    \langle\,  \delta \mid \xi \,\rangle) \,
   \Phi( \langle\, \lambda \mid \xi \,\rangle \,,\, 
     \langle\,  \delta \mid \xi \,\rangle \,;\, b )  \,.
\]
\end{enumerate}

\end{Proposition}

For proof see \cite{Gindikin-Goodman}.

\vspace{1ex}

Let $W_{G}$ be the Weyl group of $G$.  

\begin{Lemma}
\label{rootnest.lem}
Suppose $\xi = \overline{\alpha}$ and $\eta = \overline{\beta}$ are restricted
roots and that $\alpha$ and $\beta$ are the unique basic roots in the nests
$\Phi^{+}(\xi)$ and $\Phi^{+}(\eta)$. Assume that there exists $w \in W_{G}$
such that $w\alpha = \beta$ and $w\Delta_{0} = \Delta_{0}$. Then
$w\Phi^{+}(\xi) = \Phi^{+}(\eta)$ and 
$
   \langle\, \rho_{\frm} \mid w\mu \,\rangle = 
    \langle\, \rho_{\frm} \mid \mu \,\rangle  
$    
for all  $\mu \in \Phi^{+}(\xi)$.
\end{Lemma}

For proof see \cite{Gindikin-Goodman}.

\vspace{1ex}

When $\Delta_0$ has two root lengths, then the shifts $\langle\, \rho_{\frm}
\mid \alpha \,\rangle$ in the dimension formula cannot be determined just
using $h_{\frm}^{0}$. Define
\[
  \varpi_{\frm}^{0} = \rho_{\frm}  - \frac{1}{2}h_{\frm}^0   \,.
\]
Then $  \varpi_{\frm}^{0} \in \Span \Delta_0$ and the shifts are
\begin{equation}
\label{wtshift}
 \langle\, \rho_{\frm} \mid \alpha \,\rangle 
  =  \frac{1}{2}\langle\, h_{\frm}^0   \mid \alpha \,\rangle 
    + \langle\, \varpi_{\frm}^{0} \mid \alpha \,\rangle\,.
\end{equation}
From (\ref{rhoform}) we have
\[
  \langle\, \varpi_{\frm}^{0} \mid \alpha \,\rangle =
   \frac{1}{2}\langle \, \alpha \mid \alpha \,\rangle - 1\,.
\]
If $\alpha = \alpha^\vee$ for all $\alpha \in \Delta_0$, then
$\varpi_{\frm}^{0} = 0$ and $h_{\frm}^0 = 2\rho_{\frm}$.

When there are two root lengths and $\frm$ contains a simple ideal whose
Dynkin diagram is of type $C$, then from (\ref{wtshift}) and our normalization
of the Killing form we calculate that
\begin{equation}
\label{cpishift}
 \langle\, \varpi_{\frm}^{0} \mid \alpha \,\rangle = 
  \begin{cases} 0 &\mbox{for all short roots $\alpha \in \Delta_0$\,,}
    \\
    1 &\mbox{for the long root $\alpha \in \Delta_0$\,.}
 \end{cases}
\end{equation}
In this case $2\varpi_{\frm}^{0}$ is the fundamental dominant weight of
$\frm'$ associated with the long simple root. Likewise, when $\frm$ has a
simple ideal whose Dynkin diagram is of type $B$, then we calculate that
\begin{equation}
\label{bpishift}
 \langle\, \varpi_{\frm}^{0} \mid \alpha \,\rangle = 
  \begin{cases} 0 &\mbox{for all long roots $\alpha \in \Delta_0$\,,}
     \\
     -1/2 &\mbox{for the short root $\alpha \in \Delta_0$\,.}
  \end{cases}
\end{equation}
In this case $-\varpi_{\frm}^{0}$ is the fundamental dominant weight of
$\frm'$ associated with the short simple root (the highest weight of the spin
representation of $\frm'$).

\begin{Remark}
{\em 
The two situations just described suffice, since the Dynkin diagrams of type
$F_4$ and $G_2$ do not occur as subdiagrams of connected Dynkin diagrams.
}
\end{Remark}

\begin{Proposition}
\label{rhopishift.prop}
If $\xi \in \Sigma^{+}$ and $\dim \frn_{\xi} = 1$, then 
\,$d_{\xi}(\lambda) = 
 \Phi(\langle\, \lambda  \mid \xi \,\rangle \,,\,
   \langle\, \delta \mid \xi \,\rangle)$\ 
for all $\lambda \in \Gamma(G/H)$.
\end{Proposition}

For proof see \cite{Gindikin-Goodman}.

\vspace{1ex}

We create the {\em marked Satake diagram} for the pair $(G, H)$ as
follows.

\begin{itemize}

\item[{\bf (i)}]
In the Dynkin diagram for $\Delta$ indicate the vertices corresponding to
elements of $\Delta_{0}$ by $\bullet$ and indicate the vertices corresponding
to the other simple roots by $\circ$. Join vertices corresponding to roots
with the same restriction to $\fra$ by a double-pointed arrow.   

\item[{\bf (ii)}] When all roots in $\Delta_0$ have the same length and
$\alpha \in \Delta\setminus \Delta_0$ is adjacent to an element of $\Delta_0$,
put the number $\langle\, h_{\frm}^0 \mid  \alpha \,\rangle$ at the vertex for
$\alpha$.

\vspace{1ex}

\item[{\bf (iii)}] 
When the roots in $\Delta_0$ have two lengths and $\alpha \in
\Delta\setminus \Delta_0$ is adjacent to an element of $\Delta_0$, put
the pair of numbers 
$(\langle\, h_{\frm}^0 \mid \alpha \,\rangle \,,\, 
\langle\, \varpi_{\frm} \mid \alpha \,\rangle)$ 
at the vertex for $\alpha$.

\end{itemize}

\noindent
Here {\em adjacent} refers to the corresponding vertices and edges in the
Dynkin diagram for the root system of $\frg$. 

\vspace{1ex}

\begin{Remark}
{\em
Let $\alpha \in \Delta \setminus \Delta_0$. Since $h_{\frm}^0$ is a linear
combination with positive coefficients of the simple coroots of $\frm$, we
have 
$ \langle\, h_{\frm}^0 \mid \alpha \,\rangle < 0$ if $\alpha$ is adjacent to
$\Delta_0$. Furthermore, 
$ \langle\, h_{\frm}^0 \mid \alpha \,\rangle = 
 \langle\, \varpi_{\frm}^0 \mid \alpha \,\rangle = 0$
if $\alpha$ is not adjacent to $\Delta_0$.
Thus the marked Satake diagram and formulas (\ref{cpishift})
and (\ref{bpishift}) determine the values $ \langle\, h_{\frm}^0 \mid \alpha
\,\rangle$ and $ \langle\, \varpi_{\frm}^0 \mid \alpha \,\rangle$ for all
$\alpha \in \Delta$.
}
\end{Remark}

\section{Rank-One Affine Spherical Spaces}
		\label{rankone.sec}
%		\input{rankone}
%%%%%%%%%%%%%%%%%%%%%%%%%%%%%%%%%%%%%%%%%%%%%%%%%%%%%%%%%%%%%%%%%%%%%%%%%%%%%
%  Created       : Tue Apr 5 09:10:48 2011
%  Last Modified : 7/08/12
%%%%%%%%%%%%%%%%%%%%%%%%%%%%%%%%%%%%%%%%%%%%%%%%%%%%%%%%%%%%%%%%%%%%%%%%%%%%%

From \'E. Cartan's classification (see \cite[Ch. X Table VI]{Helgason1}) the
irreducible affine spherical pairs $(G, H)$ of rank one with $G$ semisimple
and $H$ connected and symmetric (the fixed points of an involution of $G$) are
as follows (in types {\bf B\,II} and {\bf D\,II} the spherical representations
are single-valued on the orthogonal groups, so the spin groups are not
needed).

\vspace{.5ex}

\begin{enumerate}

\item[{\bf A\,IV}: ] 
$G =\SL_{\ell+1}(\bC)$ with $\ell\geq 1$ and $H =
\GL_{\ell}(\bC)$ embedded by $h \mapsto h \oplus \det h^{-1} $.

\vspace{.5ex}

\item[{\bf B\,II}: ]
$G =\SO_{2\ell+1}(\bC)$ with $\ell\geq 2$ and $H = \SO_{2\ell}(\bC)$ embedded  by $h
\mapsto h \oplus 1$.

\vspace{.5ex}

\item[{\bf D\,II}: ]
$G =\SO_{2\ell}(\bC)$ with $\ell\geq 2$ and $H = \SO_{2\ell-1}(\bC)$ embedded  by $h
\mapsto h \oplus 1$.

\vspace{.5ex}

\item[{\bf C\,II}: ]
$G =\Sp_{2\ell}(\bC)$ with $\ell\geq 3$ and $H = \Sp_{2}(\bC)\times
\Sp_{2\ell-2}(\bC)$ embedded in block-diagonal form.

\vspace{.5ex}

\item[{\bf F\,II}: ] $G = {\bf F}_4(\bC)$ and $H = \Spin_9(\bC)$ embedded as
described in \cite[\S 4.2]{Baez}.

\end{enumerate}

\vspace{.5ex}

\noindent 
Type {\bf B\,II} with $\ell = 1$ is isomorphic to Type {\bf A\,IV} with $\ell
= 1$, and Type {\bf C\,II} with $\ell = 2$ is isomorphic to Type {\bf B\,II}
with $\ell = 2$, and so these are omitted from this list.

From Kr\"amer's classification \cite{Kramer} there are two
irreducible affine spherical pairs $(G, H)$ of rank one with $G$ simple and
$H$  not a symmetric subgroup of $G$. The groups involved form a descending
chain 
\[
  \Spin_{7}(\bC) \supset {\bf G}_2(\bC) \supset \SL_{3}(\bC) 
\]
and are of dimensions $21$, $14$, $8$ and ranks $3$, $2$, $2$ respectively.
The embeddings of the groups are described in \cite[Ch. 5]{Adams} and \cite[\S
8.10]{Wolf1}. The compact forms of the corresponding homogeneous spaces $G/H$
are constant positive curvature spheres of dimensions $7$ and $6$ (cf.
\cite[\S 12.7]{Wolf2}).

%% \textbf{\em Case 1.} {\bf The pair $(\SL_{\ell+1},\, \GL_{\ell})$. }
%                \input{slglexam}
%%%%%%%%%%%%%%%%%%%%%%%%%%%%%%%%%%%%%%%%%%%%%%%%%%%%%%%%%%%%%%%%%%%%%%%%%%%%%
%
%  Created       : Wed Nov 2 08:35:47 2011
%  Last Modified : 9/11/12 for short version
%
%%%%%%%%%%%%%%%%%%%%%%%%%%%%%%%%%%%%%%%%%%%%%%%%%%%%%%%%%%%%%%%%%%%%%%%%%%%%%

\vspace{2ex}
\noindent
\textbf{\em Case 1.} {\bf The pair $(\SL_{\ell+1},\, \GL_{\ell})$. }
Here $G$ has rank $\ell$. We take the diagonal matrices in
$\frg$ as a Cartan subalgebra and use simple roots $\alpha_i = \varepsilon_i -
\varepsilon_{i+1}$ for $i = 1,\ldots, \ell$. The fundamental $H$-spherical
weight is $\mu_1 = \varpi_1 + \varpi_{\ell}$ (see
\cite[Ch.~12.3.3]{Goodman-Wallach}). Assume for the moment that $\ell \geq 2$.
If $\ell = 2$ then $\Delta_0$ is empty, while if $\ell \geq 3$ then $\Delta_0
= \{\alpha_2, \ldots, \alpha_{\ell-1} \}$; in both cases $\dim \frc = 2$.
Since $|\Supp \mu_1 | = 2$, we have $\dim \frc_0 = 1$ and $\dim \fra = 1$ by
Lemma \ref{rankdim.lem}. Thus 
$\frm \cong \frc_0 \oplus \fsl_{\ell-1}$.

To determine $\frc_0$, we identify $\frt$ with $\frt^{*}$ using the form
$\langle \cdot \mid \cdot \rangle$ and write $\bx \in \frt$ as $c_1\alpha_1 +
\cdots + c_{\ell}\alpha_{\ell}$. Then $\langle \mu_1 \mid \bx \rangle = 0$
gives the relation $c_{\ell} = - c_1$. It is easy to check that the vector
\begin{equation}
\label{aivc0}
\begin{split}
  \by =& (\ell-1)\alpha_1
       + (\ell -3)\alpha_2 + (\ell -5)\alpha_4 + \cdots 
\\
    &\quad  +(5 -\ell)\alpha_{\ell - 2} + (3 - \ell)\alpha_{\ell-1}
       + (1 - \ell)\alpha_{\ell}
\\
   =& (\ell -1)\varepsilon_1 
    - 2\big(\varepsilon_2 + \cdots + \varepsilon_{\ell}\big) 
    + (\ell -1)\varepsilon_{\ell+1}
\end{split}    
\end{equation}
is orthogonal to $\Delta_0$, and hence gives a basis for $\frc_0$.

Since $\frc = \frc_0 \oplus \fra$ and $\dim \frc = 2$, we see from
(\ref{aivc0}) that 
\begin{equation}
\label{slgl.a} 
 \fra = \{ \bx = \diag[\,t \,,\,
  \underbrace{0\,,\,\ldots\,,\,0}_{\ell-1}\,,\,-t\,] 
  \,:\, t \in \bC \}\,.
\end{equation}
Let $\xi_1 \in \fra^*$ take the value $t$ on the element $\bx$ in
(\ref{slgl.a}). Then $\xi_1 = \overline{\alpha_1} = \overline{\alpha_{\ell}}
= \frac{1}{2}(\varepsilon_1 - \varepsilon_{\ell + 1}) $.

The multiplicities of the restricted positive roots are as follows when $\ell
\geq 2$ (details below).

\vspace{1ex}

\begin{center}

\begin{tabular}{|c |c|}
\hline
 restricted root   & multiplicity 
 \\
\hline
 $ \xi_1$  &  $2\ell - 2$
 \\
\hline
  $ 2\xi_1 $  & $1$ 
\\
\hline
\end{tabular}

\end{center}

%%%%%%%%%%%%%%%%%%% 
\begin{figure}[h]
\includegraphics{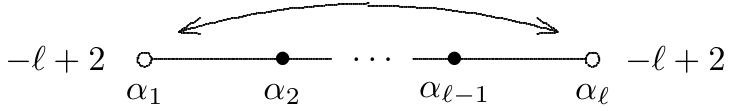} 
\caption{Marked Satake diagram for $\SL_{\ell+1}/\GL_{\ell}$ 
 with $\ell \geq 2$}
\label{diagslgl.fig}
%\vspace{-4ex}
\end{figure}
%%%%%%%%%%%%%%%%%%%%%%%%%%%%%%%%%%%%%%%%%

\vspace{1ex}

\noindent
From the table we see that  $\delta = \ell \xi_1$.
Since all the roots in $\Delta_0$ have the same length, \ 
 $  h_{\frm}^0 = 2\rho_{\frm} = 
 (\ell - 2)\alpha_2 + \cdots 
 + (j-1)(\ell-j)\alpha_j + \cdots + (\ell-2)\alpha_{\ell -1}$\,.
Hence 
\begin{equation}
\label{slglshift}
 \langle \, h_{\frm}^0 \mid \alpha_i \, \rangle =
 \begin{cases} 2 - \ell & \text{if $i = 1$ or $\ell$\,,} \\
               2 & \text{if $i = 2, \ldots , \ell -1 $\,.}
 \end{cases}               
\end{equation}

We now determine the nests of restricted roots, the basic roots, and the
dimension factors $d_{\xi}(\lambda)$ for $\xi \in \Sigma^{+}$ and $\lambda
\in \Gamma(G/H)$.

\vspace{1ex}

\begin{enumerate}

\item[{\bf (i)}] Let $\xi = \xi_1$. Then
\ 
$
 \Phi^{+}(\xi) = 
  \{\beta_1\,,\, \ldots \,,\, \beta_{\ell-1} \}
  \cup
  \{ \gamma_1 \,,, \ldots \,\,  \gamma_{\ell-1} \}, 
$
where 
\[
\begin{split}
 \beta_j &= \varepsilon_1 - \varepsilon_{j+1} 
 = \alpha_1 + \cdots + \alpha_{j}\,,
\\
 \gamma_j &= \varepsilon_{\ell -j + 1} - \varepsilon_{\ell+1} 
 = \alpha_{\ell - j + 1} + \cdots + \alpha_{\ell}\,.
\end{split}
\]
Thus $m_{\xi} = 2(\ell - 1)$ and the basic roots are $\beta_1$ and
$\gamma_1$. From (\ref{slglshift}) we have $k_{\xi} = - \langle \,
h_{\frm}^{0} \,,\, \beta_1 \, \rangle = - \langle \, h_{\frm}^{0} \,,\,
\gamma_1 \, \rangle = \ell - 2$.
Since $m_{\xi} = 2(k_{\xi} + 1)$, Proposition \ref{rhoshift.prop} (3) 
gives 
\[
 d_{\xi}(\lambda) = \left[
 \Phi(\langle\, \lambda + \delta \mid \xi \,\rangle \,,\,
  \langle\, \delta \mid \xi \,\rangle 
  \,;\, \textstyle{ \frac{1}{2}}(\ell - 2))
    \right]^2 \,.
\]

\vspace{1ex}

\item[{\bf (ii)}]
 Let $\xi = 2\xi_1$. Then
\ 
$ \Phi^{+}(\xi) = \{ \alpha_1 + \cdots  + \alpha_{\ell} \}$. Hence 
by Proposition \ref{rhopishift.prop}  we have
\,$d_{\xi}(\lambda)  =  
 \Phi(\langle\, \lambda  \mid \xi \,\rangle \,,\,
  \langle\,  \delta \mid \xi \,\rangle)$\,.
\end{enumerate}

From cases (i), (ii) and (\ref{restdim}) we obtain the dimension formula
\begin{equation}
\label{slgldim}
\begin{split}
 d(\lambda) &= \Phi( \langle\, \lambda  \mid \xi_1 \,\rangle \,,
   \langle\,  \delta \mid \xi_1 \,\rangle )
\left[
   \Phi\big(\langle\, \lambda \mid \xi_1 \,\rangle,
    \langle\, \delta \mid \xi_1 \,\rangle \,;\, 
   \textstyle{ \frac{1}{2} }(\ell - 2)\big)
 \right]^2 
\\
 &=  W\big(\langle\, \lambda  \mid \xi_1 \,\rangle \,,\, 
       \langle\, \delta \mid \xi_1 \,\rangle \,;\,
       m_{\xi_1}\,, m_{2\xi_1} \big) 
 \,.
\end{split}
\end{equation}
Here  $\langle\, \xi_1 \mid \xi_1 \,\rangle = \frac{1}{2}$,  so that 
$\langle\, \delta \mid \xi_1 \,\rangle = \frac{1}{2} \ell$ and 
 $\langle\, \mu_1 \mid \xi_1 \,\rangle = 1$.
Note that when $\ell = 2$ this formula becomes 
\begin{equation*}
\label{slgldim2}
 d(\lambda) = 
\left[
 \Phi\big( \langle\, \lambda + \delta \mid \xi_1 \,\rangle \,,
   \langle\,  \delta \mid \xi_1 \,\rangle \,;\, 0 \big)
 \right]^3  \,.
\end{equation*}
Thus $d(\mu_1) = 2^3$ as expected, since $\mu_1$ is the highest weight of the
adjoint representation of $G$ when $\ell = 2$.

The dimension formula when $\ell = 1$ (so $G = \SL_2(\bC)$ and $H$ is a
maximal torus in $G$) is different. In this case the fundamental $H$-spherical
highest weight is $\mu_1 = 2\varpi_1$, the multiplicities are $m_{\xi_1} = 1$,
$m_{2\xi_1} = 0$, and 
\begin{equation}
\label{slgldim1}
 d\big(\lambda) = \Phi(\langle\, \lambda  \mid \xi_1 \,\rangle \,,\,
 \langle\, \delta \mid \xi_1 \,\rangle \big)
 = W\big(\langle\, \lambda \mid \xi_1 \,\rangle \,,\,
\langle\,  \delta \mid \xi_1 \,\rangle  \,;\, m_{\xi_1} \,\big)
\,.
\end{equation}
In this case $\langle\, \xi_1 \mid \xi_1 \,\rangle = 2$, so that $\langle\,
\delta \mid \xi_1 \,\rangle = 1$ and $d(k\mu_1) = 2k + 1$ as expected.

%% \textbf{\em Case 2.} {\bf The pair $(\SO_{2\ell+1},\, \SO_{2\ell})$. }
%           \input{bdexam}
           %%%%%%%%%% diagram: newbii.eps
%%%%%%%%%%%%%%%%%%%%%%%%%%%%%%%%%%%%%%%%%%%%%%%%%%%%%%%%%%%%%%%%%%%%%%%%%%%%%
%
%  Created       : Wed Nov 2 08:35:47 2011
%  Last Modified : 9/11/12 for short version
%
%%%%%%%%%%%%%%%%%%%%%%%%%%%%%%%%%%%%%%%%%%%%%%%%%%%%%%%%%%%%%%%%%%%%%%%%%%%%%

\vspace{2ex} 
\noindent 
\textbf{\em Case 2.} {\bf The pair $(\SO_{2\ell+1},\,
\SO_{2\ell})$. } Take $G$ in the matrix form of \cite[\S
2.1.2]{Goodman-Wallach} and the diagonal matrices in $\frg$ as a Cartan
subalgebra. 

Consider first the case $\ell \geq 3$. We choose simple roots $\alpha_i =
\varepsilon_i - \varepsilon_{i+1}$ for $i = 1,\ldots, \ell-1$, and
$\alpha_{\ell} = \varepsilon_{\ell}$. The fundamental $H$-spherical weight is
$\mu_1 = \varpi_1 = \varepsilon_1$ (see \cite[Ch.~12.3.3]{Goodman-Wallach}).
Thus $\Delta_0 = \{\alpha_2, \ldots, \alpha_{\ell} \}$ and hence $\dim \frc =
1$. Since $|\Supp \mu_1 | = 1$, we have $ \fra = \frc$ by Lemma
\ref{rankdim.lem}. Thus 
$\frm \cong \fso_{2\ell-1}$
and
\begin{equation}
\label{bd.a} 
 \fra = \{ \bx = \diag[\,t \,,\,
\underbrace{0\,,\,\ldots\,,\,0}_{2\ell-1}\,,\,-t\,] \,:\, t \in \bC \}\,. 
\end{equation}
Let $\xi_1 \in \fra^*$ take the value $t$ on the element $\bx$ in
(\ref{bd.a}).  Then $\xi_1 = \overline{\alpha_1}
= \varepsilon_1$.

The multiplicities of the restricted positive roots for $\ell \geq 2$ are as
follows (details below).

\vspace{1ex}

\begin{center}

\begin{tabular}{|c |c|}
\hline
 restricted root   & multiplicity 
 \\
\hline
 $ \xi_1$  &  $2\ell - 1$
\\
\hline
\end{tabular}

\end{center}

\vspace{1ex}

%%%%%%%%%%%%%%%%%%% 
\begin{figure}[h]
\includegraphics{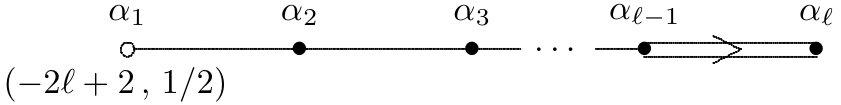} 
\caption{Marked Satake diagram for $\SO_{2\ell+1}/\SO_{2\ell}$ 
 with $\ell \geq 3$}
\label{diagbd.fig}
%\vspace{-4ex}
\end{figure}
%%%%%%%%%%%%%%%%%%%%%%%%%%%%%%%%%%%%%%%%%

\noindent
From the table we obtain
$\delta = \big(\ell - \frac{1}{2}\big) \xi_1$. Using the basis
$\{\varepsilon_1 , \ldots , \varepsilon_{\ell} \}$ for $\frt$ and the
identification of $\frt^{*}$ with $\frt$, we can write
\[
\begin{split}
 2\rho_{\frm} &=  (2\ell - 3)\varepsilon_2 + (2\ell - 5)\varepsilon_3
  + \cdots + 3\varepsilon_{\ell-1} + \varepsilon_{\ell} \,,
\\
 h_{\frm}^{0} &= (2\ell - 2)\varepsilon_2 + (2\ell - 4)\varepsilon_3
  + \cdots + 4\varepsilon_{\ell-1} + 2\varepsilon_{\ell}\,.
\end{split}
\]
Hence 
$\varpi_{\frm}^{0} = \rho_{\frm} - \frac{1}{2}h_{\frm}^{0}
= -\frac{1}{2}\big(\varepsilon_2 +  \cdots + \varepsilon_{\ell}\big)$. 
 From these formulas we see that
\begin{equation}
\label{bdshift1} 
 \langle\, h_{\frm}^0 \mid \alpha_i \,\rangle =
 \begin{cases} -2\ell + 2 & \text{if $i = 1$\,,} \\
               2 & \text{if $i = 2, \ldots, \ell$\,,}
 \end{cases}               
\end{equation}
and
\begin{equation}
\label{bdshift2}
 \langle \, \varpi_{\frm}^0 \mid \alpha_i \,\rangle =
 \begin{cases} 1/2 & \text{if $i = 1$\,,}\\
               0 & \text{if  $i = 2, \ldots, \ell - 1$\,,}\\
               -1/2 & \text{if $i = \ell$\,.}\\
 \end{cases}               
\end{equation}
(The last two cases in (\ref{bdshift2}) were already given in
(\ref{bpishift}).)

The nest of positive restricted roots is 
\[
\begin{split}
 \Phi^{+}(\xi_1) &= 
    \{\varepsilon_1 - \varepsilon_{j} \,:\, 2\leq j \leq \ell \}   
   \cup
   \{\varepsilon_1 + \varepsilon_{j} \,:\, 2\leq j \leq \ell \}   
   \cup
   \{\varepsilon_1 \}   
\\
   &= \{\beta_{j} \,:\, 2\leq j \leq \ell \}   
   \cup
   \{ \gamma_{j} \,:\, 2\leq j \leq \ell  \}   
   \cup
   \{\alpha_1 + \cdots + \alpha_{\ell} \} \,,  
\end{split}
\]
where $\beta_{j} = \alpha_1 + \cdots + \alpha_{j-1}$ and $\gamma_{j} =
\beta_{j} + 2\alpha_{j} + \cdots + 2\alpha_{\ell}$. Thus $\left|
\Phi^{+}(\xi_1) \right| = 2\ell - 1$ as indicated in the table.

From (\ref{bdshift1}) we see that $k_{\xi_1} = 2\ell - 2$ and the only basic
root in the nest is $\beta_{2} = \alpha_{1}$. Hence the eigenvalues of $\ad
h_{\frm}^{0}$ on $\frn_{\xi_1}$ are
\[
  -2\ell + 2, \,  \ldots \,, - 2, \, 0 , 
  \, 2, \, \ldots \,, 2\ell - 2\,,
\]
each with multiplicity one, with the negative eigenvalues coming from
$\{\beta_j\}$ and the positive eigenvalues from $\{\gamma_j\}$.  From 
(\ref{bdshift2}) we have  
$ \langle \, \varpi_{\frm}^0 \mid \beta_j \,\rangle = 1/2$, \ 
$ \langle \, \varpi_{\frm}^0 \mid \gamma_j \,\rangle = -1/2$,
and 
$\langle\, \varpi_{\frm}^0 \mid \alpha_1 + \cdots + \alpha_{\ell} \,\rangle = 0$.
Hence the $\rho_{\frm}$ shifts in the dimension formula are
\begin{equation}
\label{bdrhoshift}
  -\ell + \textstyle{\frac{3}{2}}, \, \ldots \,, - \textstyle{\frac{1}{2}}
  , \, 0 ,   \, \textstyle{\frac{1}{2}}, \, \ldots \,,  \ell - 
   \textstyle{\frac{3}{2}}\,.
\end{equation}
Note that the positive and negative shifts have unit spacing, but the space
around $0$ has size $1/2$ because of the additional shift from
$\varpi_{\frm}^{0}$, which decreases the positive eigenvalues and increases
the negative eigenvalues. Let $\lambda \in \Gamma(G/H)$ be an $H$-spherical
highest weight. It follows from (\ref{bdrhoshift}) and (\ref{restdim}) that
\begin{equation}
\label{bddim}
\begin{split}
 d(\lambda) &= \Phi( \langle\, \lambda  \mid \xi_1 \,\rangle ,
  \langle\,  \delta \mid \xi_1 \,\rangle )
\,  
\Phi\big( \langle\, \lambda  \mid \xi_1 \,\rangle ,
  \langle\,  \delta \mid \xi_1 \,\rangle \,;\, 
  \ell - \textstyle{\frac{3}{2}}\,\big)
\\
 &= W\big( \langle\, \lambda  \mid \xi_1 \,\rangle,
  \langle\, \delta \mid \xi_1 \,\rangle \,;\, 
   m_{\xi_1} \big)\,.
\end{split}
\end{equation}
Here $\langle\, \xi_1 \mid \xi_1 \,\rangle = \langle\, \mu_1 \mid \xi_1
\rangle = 1$, so that $\langle\, \delta \mid \xi_1 \,\rangle = \ell -
\frac{1}{2}$. 

\begin{Remark}
{\em
If $\lambda = k\mu_1$, then (\ref{bddim}) gives $d(\lambda)
= 2\ell + 1$ when $k = 1$ and 
\begin{equation}
\label{spherharmdim}
 d(\lambda) = \frac{
  (k+2\ell - 2)!
  }{
  k! (2\ell -1)!}
  (2k+2\ell-1)
  = \binom{2\ell + k}{k} - \binom{2\ell + k -2}{k-2}
\end{equation}
when $k \geq 2$. This is the well-known formula for the dimension of the space
of spherical harmonics of degree $k$ in $2\ell + 1$ variables.
}
\end{Remark}

The case $\ell = 2$ ($G = \SO_5(\bC)$ and $H = \SO_4(\bC)$) is different.
The fundamental $H$-spherical weight is still $\varpi_1$, but
$\Delta_0 = \{\alpha_2\}$ only has one root length, $\frm \cong \fsl_2$, and
$h_{\frm}^0 = 2\alpha_2$. The root nest for $\xi_1 = \overline{\alpha_1}$ is
\[
  \Phi^{+}(\xi_1) = 
  \{\, \alpha_1 \,,\, \alpha_1 + \alpha_2 \,,\, \alpha_1 + 2\alpha_2 \,\}
\]
with basic root $\alpha_1$, and 
$k_{\xi_1} = - \langle 2\alpha_2 \mid \alpha_1   \rangle = 2$. 
Hence Proposition \ref{rhoshift.prop} (3) gives
\begin{equation}
\label{bddim2}
 d(\lambda) = \Phi\big(\langle\, \lambda \mid \xi_1 \,\rangle \,,\,
    \langle\, \delta \mid \xi_1 \,\rangle \,;\, 1\big)
 = W\big( \langle\, \lambda  \mid \xi_1 \,\rangle,
  \langle\, \delta \mid \xi_1 \,\rangle \,;\, 
   m_{\xi_1} \big) \,.
\end{equation}
%

%% \textbf{\em Case 3.} {\bf The pair $(\SO_{2\ell},\, \SO_{2\ell-1})$. }
%           \label{dbexam.sec}
%          \input{dbexam}
          %%%%%%%%%% diagram:   newdii.eps 
%%%%%%%%%%%%%%%%%%%%%%%%%%%%%%%%%%%%%%%%%%%%%%%%%%%%%%%%%%%%%%%%%%%%%%%%%%%%%
%
%  Created       : Wed Nov 2 08:35:47 2011
%  Last Modified : 9/11/12 for short version
%
%%%%%%%%%%%%%%%%%%%%%%%%%%%%%%%%%%%%%%%%%%%%%%%%%%%%%%%%%%%%%%%%%%%%%%%%%%%%%

\vspace{2ex}
\noindent
\textbf{\em Case 3.} {\bf The pair $(\SO_{2\ell},\, \SO_{2\ell-1})$. } Assume
that $\ell \geq 2$ and take $G$ in the matrix form of \cite[\S
2.1.2]{Goodman-Wallach}, with the diagonal matrices in $\frg$ as a Cartan
subalgebra. Use the simple roots $\alpha_i = \varepsilon_i - \varepsilon_{i+1}$
for $i = 1,\ldots, \ell-1$ and $\alpha_{\ell} = \varepsilon_{\ell-1} +
\varepsilon_{\ell}$. 

Consider first the case $\ell \geq 3$. The fundamental $H$-spherical weight is
$\mu_1 = \varpi_1 = \varepsilon_1$ (see \cite[Ch.~12.3.3]{Goodman-Wallach}).
Thus $\Delta_0 = \{\alpha_2, \ldots, \alpha_{\ell} \}$ and hence $\dim \frc =
1$. Since $|\Supp \mu_1 | = 1$, we have $ \fra = \frc$ by Lemma
\ref{rankdim.lem}; thus 
$\frm = \frm' \cong \fso_{2\ell-2}$.
For this choice of Cartan subalgebra
\begin{equation}
\label{db.a} 
 \fra = \{ \bx = \diag[\,t \,,\,
\underbrace{0\,,\,\ldots\,,\,0}_{2\ell-2}\,,\,-t \,] \,:\, t \in \bC  \} \,. 
\end{equation}
Let $\xi_1 \in \fra^*$ take the value $t$ on the element $\bx$ in
(\ref{db.a}).  Then $\xi_1 = \overline{\alpha_1} = \varepsilon_1$ and 
$\langle \xi_1 \mid \xi_1 \rangle = 1$.

The multiplicities of the restricted positive roots are as follows (details
below).

\vspace{1ex}

\begin{center}

\begin{tabular}{|c |c|}
\hline
 restricted root   & multiplicity 
 \\
\hline
 $ \xi_1$  &  $2\ell - 2$
\\
\hline
\end{tabular}

\end{center}

%

%%%%%%%%%%%%%%%%%%% 
\begin{figure}[h]
%\vspace{-2ex}
%\hspace*{40mm}
\includegraphics{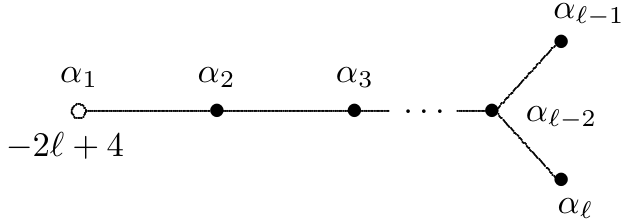} 
\caption{Marked Satake diagram for $\SO_{2\ell}/\SO_{2\ell-1}$ 
        with $\ell \geq 3$}
\label{diagdb.fig}
%\vspace{-4ex}
\end{figure}
%%%%%%%%%%%%%%%%%%%%%%%%%%%%%%%%%%%%%%%%%

\noindent 
From the table we obtain
$\delta = (\ell - 1) \xi_1$. 
Using the basis $\{\varepsilon_1 ,\ldots, \varepsilon_{\ell} \}$ for $\frt$
and the identification of $\frt^{*}$ with $\frt$, we can write
\[
 h_{\frm}^{0} = 2(\ell - 2)\varepsilon_2 + 2(\ell - 3)\varepsilon_3
  + \cdots + 2\varepsilon_{\ell-1} \,.
\]
Hence
\begin{equation}
\label{dbshift} 
 \langle\, h_{\frm}^0 \mid \alpha_i \,\rangle =
 \begin{cases} -2\ell + 4 & \text{if $i = 1$\,,} \\
               2 & \text{if $i = 2, \ldots, \ell$\,,}
 \end{cases}               
\end{equation}

The nest of positive restricted roots is 
\[
\begin{split}
 \Phi^{+}(\xi_1) &= 
    \{\varepsilon_1 - \varepsilon_{j} \,:\, 2\leq j \leq \ell \}   
   \cup
   \{\varepsilon_1 + \varepsilon_{j} \,:\, 2\leq j \leq \ell \}   
\\
   &= \{\beta_{j} \,:\, 2\leq j \leq \ell \}   
   \cup
   \{ \gamma_{j} \,:\, 2\leq j \leq \ell -1  \}   
   \cup\{\alpha_1 + \cdots + \alpha_{\ell-2} + \alpha_{\ell} \} \,,
\end{split}
\]
where $\beta_{j} = \alpha_1 + \cdots + \alpha_{j-1}$ and $\gamma_{j} =
\beta_{j} + 2\alpha_{j} + \cdots + 2\alpha_{\ell-2} + \alpha_{\ell-1} +
\alpha_{\ell}$ (the roots with coefficient $2$ are omitted when $j = \ell -
1$). Thus $\left| \Phi^{+}(\xi_1) \right| = 2\ell - 2$ as indicated in the
table. From (\ref{dbshift}) we see that the only basic root in the nest is
$\beta_{2} = \alpha_{1}$ and that $k_{\xi_1} = 2\ell - 4$. 

Let $\lambda \in \Gamma(G/H)$ be an $H$-spherical highest weight. Since all
the roots in $\Delta_0$ have the same length and \ $\left|\Phi^{+}(\xi)\right|
= k_{\xi_1} + 2$, it follows from Proposition \ref{rhoshift.prop} (4) that
\begin{equation}
\label{dbdim}
\begin{split}
  d(\lambda) &= \Phi(
   \langle\, \lambda  \mid \xi_1 \,\rangle ,
  \langle\, \delta \mid \xi_1 \,\rangle )
  \,
  \Phi(
   \langle\, \lambda  \mid \xi_1 \,\rangle ,
  \langle\, \delta \mid \xi_1 \,\rangle \,;\, \ell - 2 \, \big)
\\
 &= W\big( \langle\, \lambda  \mid \xi_1 \,\rangle,
  \langle\, \delta \mid \xi_1 \,\rangle \,;\, 
   m_{\xi_1}\, \big)\,.
\end{split}
\end{equation}
Here $\langle\,  \delta \mid \xi_1 \,\rangle  = 
 (\ell -1)  \langle\,  \xi_1 \mid \xi_1 \,\rangle  = \ell - 1$.

\vspace{1ex}

Consider the case $\ell = 2$. Now $\frh \cong \fsl_2$ and $\frg \cong \fsl_2
\oplus \fsl_2$ is not simple. In this case the fundamental $H$-spherical
highest weight is $\mu_1 = \varpi_1 + \varpi_2$ and $\Delta_0$ is empty. Hence
$ \frc = \frt = \fra + \frc_0$, where
$\frc_0 = \{\, \diag[\,0\,,\, t \,,\, -t \,,\, 0\,] \,:\, t \in \bC\}$
and
\begin{equation}
\label{db.a2}
  \fra = \{\, \bx = \diag[\,t \,,\, 0\,,\, 0\,,\, -t \,] \,:\, t \in \bC \}
  \,.
\end{equation}
Let $\xi_1 \in \fra^*$ take the value $t$ on the element $\bx$ in
(\ref{db.a2}). Then $\xi_1 = \overline{\alpha_1} = \overline{\alpha_2} =
\varepsilon_1$ is the positive restricted root with multiplicity two. Since
$\frm = \frc_0$, we have $\rho_{\frm} = 0$ and 
$\rho_{\frg} = \delta = \xi_1$. Hence by the Weyl character formula
\[
 d(\lambda) = 
\left[
 \Phi\big(\langle\, \lambda  \mid \xi_1 \,\rangle \,,\,
  \langle\,  \delta \mid \xi_1 \,\rangle \big)
 \right]^2 
  = W\big( \langle\, \lambda  \mid \xi_1 \,\rangle,
  \langle\, \delta \mid \xi_1 \,\rangle \,;\, 
   m_{\xi_1}\, \big)\,.
\]
%

%% \textbf{\em Case 4.}
%        {\bf The pair $(\Sp_{2\ell},\, \Sp_2 \times \Sp_{2\ell-2})$ }
%          \input{ccexam}
          %%%%%%%%% diagram:   newcii1.eps
%%%%%%%%%%%%%%%%%%%%%%%%%%%%%%%%%%%%%%%%%%%%%%%%%%%%%%%%%%%%%%%%%%%%%%%%%%%%%
%  Created       : Wed Nov. 23, 2011
%  Last Modified : 9/11/12
%
%%%%%%%%%%%%%%%%%%%%%%%%%%%%%%%%%%%%%%%%%%%%%%%%%%%%%%%%%%%%%%%%%%%%%%%%%%%%%

\vspace{2ex}
\noindent
 \textbf{\em Case 4.}
        {\bf The pair $(\Sp_{2\ell},\, \Sp_2 \times \Sp_{2\ell-2})$. }
Assume that $\ell \geq 3$ and take $G$ in the matrix form of \cite[\S
2.1.2]{Goodman-Wallach}, with 
the Cartan subalgebra $\frt$ of $\frg$ the matrices $\bx = \diag [\, \by,
-\check{\by}]$, where 
$\by = [\, \varepsilon_1(\by), \ldots, \varepsilon_{\ell}(\by)\,]$
and 
$\check{\by} =  [\, \varepsilon_{\ell}(\by), \ldots, \varepsilon_{1}(\by)\,]$.
The roots of $\frt$ on $\frg$ are $\pm \varepsilon_i \pm \varepsilon_j$ for
$1\leq i < j \leq \ell$ and $\pm 2\varepsilon_i$ for $1 \leq i \leq \ell$.
Take the simple roots as $\alpha_i = \varepsilon_{i} - \varepsilon_{i+1}$ for
$i=1, \ldots, \ell - 1$ and $\alpha_{\ell} = 2\varepsilon_{\ell}$.

The semigroup $\Gamma(G/H)$ is free of rank $1$ with generator
\begin{equation}
\label{ccfundwt}
  \mu_1 = \varpi_2 = \varepsilon_1 + \varepsilon_2 
\end{equation}
(cf. \cite[\S 12.3.3]{Goodman-Wallach}). 
Thus 
$\Delta_{0} = \Delta \setminus \{\alpha_2\}$.
Since $|\Supp \mu_1 | = 1$ we know from Lemma \ref{rankdim.lem} that $\fra =
\frc$, $\frc_0 = 0$, and hence $\frm \cong \fsp_{2}(\bC) \oplus \fsp_{2\ell -
4}(\bC)$. Thus
\begin{equation}
\label{ccdiag}
 \fra = \{ \bx = \diag[\, \by, -\check{\by} \,]
 \ \mbox{with \ $\by = [t, t, 0, \ldots, 0 ]$ and $t \in \bC$ }  \} \,.
\end{equation}
Let $\xi_1 \in \fra^*$ take the value $t$ on the element $\bx$ in
(\ref{ccdiag}). Then 
$\xi_1 = \overline{\alpha_2} = \frac{1}{2}(\varepsilon_1 + \varepsilon_2)$. 
The restricted positive roots and their multiplicities are as follows (details
given below).

\vspace{1ex}

\begin{center}

\begin{tabular}{|c |c|}
\hline
 restricted root   & multiplicity 
 \\
\hline
 $ \xi_1$ &  $4(\ell - 2)$
 \\
\hline
 $2\xi_1$ &   $3$ 
\\
\hline
\end{tabular}

\end{center}

\vspace{1ex}

%

%%%%%%%%%%%%%%%%%%% 
\begin{figure}[h]
\includegraphics{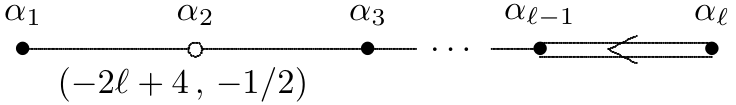} 
\caption{Marked Satake diagram for 
   $\Sp_{2\ell}/(\Sp_{2}{\times} \Sp_{2\ell-2})$ with $\ell \geq 3$}
\label{diagcii1.fig}
\end{figure}
%%%%%%%%%%%%%%%%%%%%%%%%%%%%%%%%%%%%%%%%%

\noindent
From the table we obtain
$\delta = (\ell - 1) \xi_1$.
Here $\Delta_0$ has  two root lengths,
with  $\alpha_{i}^\vee = \alpha_{i}$ for $i < \ell$ and
$\alpha_{\ell}^\vee = \varepsilon_{\ell}$. Furthermore
\[
\begin{split}
 \rho_{\frm} &= \textstyle{\frac{1}{2}}
   \big[\varepsilon_1 - \varepsilon_{2}\big] + 
(\ell-2)\varepsilon_3 + \cdots + 2\varepsilon_{\ell-1} + \varepsilon_{\ell} \,, 
\\
  h_{\frm}^0 &= \big[\varepsilon_1 - \varepsilon_{2}\big]
 + \big[(2\ell-5)\varepsilon_3 + \cdots 
  + 3\varepsilon_{\ell-1} + \varepsilon_{\ell}\big]
   =    h_{\frm_1}^0 +  h_{\frm_2}^0 \,,
\end{split}
\]
corresponding to the decomposition 
$\frm = \frm_1 \oplus \frm_2$ with $\frm_1 = \fsp_2(\bC)$ and $\frm_2 =
\fsp_{2\ell - 4}(\bC)$. We have 
\begin{equation}
\label{ccshift1}
 \langle\, h_{\frm_1}^0 \mid \alpha_i \,\rangle =
 \begin{cases} 
            2 & \text{if $i = 1$} \\
           -1 & \text{if $i = 2$} \\
            0 & \text{if $i > 2$}
 \end{cases}               
\quad\mbox{and}\quad
 \langle\, h_{\frm_2}^0 \mid \alpha_i \,\rangle =
 \begin{cases} 
         0 & \text{if $i = 1$} \\
        -2\ell + 5 & \text{if $i = 2$}\\
         2 & \text{if $3 \leq i \leq \ell$\,.}
 \end{cases}               
\end{equation}
Thus \ 
$ \varpi_{\frm}^0 = \rho_{\frm} - \frac{1}{2}h_{\frm}^0
 =  \frac{1}{2}\big[\varepsilon_3 + \cdots +  \varepsilon_{\ell}\big]$ \ 
 and hence
\begin{equation}
\label{ccshift2}
 \langle \, \varpi_{\frm}^0 \mid \alpha_i \,\rangle =
 \begin{cases} 0 & \text{for $i = 1$ and $3 \leq i \leq  \ell-1$\,,} \\
               -1/2 & \text{if $i = 2$\,,}\\
               1 & \text{if $i = \ell$\,,}
 \end{cases}               
\end{equation}
as in (\ref{cpishift}).
The nests of positive roots  are as follows.

\vspace{1ex}

\begin{enumerate}

%%%%%%%%%%%%%%%%%%% 
\item[{\bf (i)}] Let $\xi = \xi_1$. Then
\[
\begin{split}
 \Phi^{+}(\xi) &= 
  \{\varepsilon_2 \pm \varepsilon_j \,:\, 3 \leq j \leq \ell \}
  \cup
  \{\varepsilon_1 \pm \varepsilon_j \,:\, 3 \leq j \leq \ell \} \\
  &= \{\beta_j \,:\, 3\leq j \leq \ell \}  
   \cup  \{\alpha_1 + \beta_j \,:\, 3\leq j \leq \ell \} \\
  & \quad  
   \cup  \{ \gamma_j \,:\, 3\leq j \leq \ell \} 
   \cup  \{\alpha_1 + \gamma_j   \,:\, 3\leq j \leq \ell \} \,,
\end{split}
\]  
where $\beta_j = \alpha_2 + \cdots + \alpha_{j-1}$ and
$\gamma_j = \beta_j + 2\alpha_j + \cdots 
      + 2\alpha_{\ell-1} + \alpha_{\ell}$ 
(here we take  $\gamma_{\ell} = \beta_{\ell} + \alpha_{\ell}$).
The basic root is $\beta = \alpha_2$ and $\dim \frn_{\xi} = 4(\ell - 2)$.

From (\ref{ccshift1})  we see that
$ \langle\, h_{\frm_1}^0 \mid \beta  \,\rangle = -1$
and
$ \langle\, h_{\frm_2}^0 \mid \beta  \,\rangle = -2\ell+5$.
Hence $\frn_{\xi_1} \cong \bC^{2} \otimes \bC^{2(\ell - 2)}$ as a
representation of $\frm_1 \oplus \frm_2$ (the tensor product of the defining
representations). Thus the eigenvalues of $\ad \frh_{\frm_1}^0$ on
$\frn_{\xi_1}$ are $\pm 1$ with multiplicity $2(\ell -2)$, whereas the
eigenvalues of $\ad \frh_{\frm_2}^0$ on $\frn_{\xi_1}$ are 
$-2\ell + 5$\,,\,$\ldots$\,,\,$2\ell - 5$ with multiplicity 2. 
It follows that 
$\rule{0ex}{2.5ex}\frac{1}{2}\ad \frh_{\frm}^0 = 
  \frac{1}{2}\big(\ad h_{\frm_1}^0 +  \ad h_{\frm_2}^0\big)$
has eigenvalues
\[
  -\ell + 2\,,\,\underbrace{ -\ell+ 3 \,,\, \ldots \,,\, 
   -1 \,,\, 0 \,,\, 1 \,,\, \ldots \,,\, \ell - 3}_{\mbox{multiplicity $2$}}
    \,,\, \ell - 2
\]
on $\frn_{\xi_1}$, with the highest and lowest eigenvalues of multiplicity
one. The eigenvalues $\leq 0$ come from the roots $\beta_j$ and $\alpha_1 +
\beta_j$, while the eigenvalues $\geq 0$ come from the roots $\gamma_j$ and
$\alpha_1 + \gamma_j$. By (\ref{ccshift2}) $\varpi_{\frm}^0$ takes the value
$-1/2$ on the first set of roots and the value $1/2$ on the second set. Since
$\rho_{\frm} = \frac{1}{2}h_{\frm}^0 + \varpi_{\frm}^0$, it follows that the 
values of 
$ \langle\, \rho_{\frm} \mid \alpha \, \rangle$ 
for $\alpha \in \Phi^{+}(\xi_1)$ are
\[
\textstyle{ -\ell + \frac{3}{2}}\,,\,
 \underbrace{\textstyle{ 
  -\ell+ \frac{5}{2} \,,\, \ldots \,,\,  -\frac{3}{2}
  }}_{\mbox{multiplicity $2$}} \,,\,
\textstyle{ -\frac{1}{2} \,,\, \frac{1}{2} \,,\,}
 \underbrace{\textstyle{
 \frac{3}{2} \,,\, \ldots \,,\, \ell - \frac{5}{2}
 }}_{\mbox{multiplicity $2$}}
  \textstyle{  \,,\, \ell - \frac{3}{2}}\,.
\]
Notice that the double eigenvalue $0$ for $\frac{1}{2}\ad h_{\frm}^0$ is
shifted to $\pm \frac{1}{2}$. 

%%%%%%%%%%%%%%%%%%%%%%%%%%%%%%%%%%%%%%%%%%%%%%%%

\vspace{1ex}

\item[{\bf (ii)}] Let $\xi = 2\xi_1$.  Then 
\[
  \Phi^{+}(\xi) = 
 \{2\varepsilon_2\,,\, \varepsilon_1+\varepsilon_2 \,,\, 2\varepsilon_1 \}   
 =  \{ \beta\,,\, \beta + \alpha_1\,,\, \beta + 2\alpha_1 \} \,, 
\]
where 
$\beta = 2\varepsilon_2 = 2\alpha_{2} + \cdots + 2\alpha_{\ell-1} + \alpha_{\ell}$ 
is the basic root. By (\ref{ccshift1}) we have 
$\langle\, h_{\frm_1}^0 \mid \beta \,\rangle = -2$ 
and 
$ \langle\, h_{\frm_2}^0 \mid  \beta \,\rangle = 0$. Thus $\frn_{2\xi_1}$ is
the $3$-dimensional irreducible representation of $\frm_{1}$, with $\frm_{2}$
acting by zero. From (\ref{ccshift2}) we have
$ \langle \, \varpi_{\frm}^0 \mid \alpha \,\rangle = 0$ 
for all $\alpha \in \Phi^{+}(2\xi_1)$. Hence the values of 
$ \langle\, \rho_{\frm} \mid \alpha \, \rangle$ are $-1$\,,\,$0$\,,\,$1$.

%%%%%%%%%%%%%%%%%%%%%%%%%%%%%%%%%%%%%%%%%%%%%%%%

\end{enumerate}

\vspace{1ex}

\noindent
From the calculation of the shifts 
$\langle\, \rho_{\frm} \mid \alpha \, \rangle$
in cases (i) and (ii) and using (\ref{restdim}), we obtain the dimension
formula
\begin{equation}
\label{ccdim}
\begin{split}
 d(\lambda) &=
                \frac{ \rule[-1ex]{0ex}{2.5ex}
               \Phi\big(\langle\, \lambda  \mid \xi_1 \,\rangle ,
               \langle\, \delta \mid \xi_1 \,\rangle  
               \,;\, \ell - \frac{5}{2}\big)
               \, \Phi\big(\langle\, \lambda  \mid \xi_1 \,\rangle ,
                \langle\, \delta \mid \xi_1 \,\rangle  
                \,;\, \ell - \frac{3}{2}\big)
          }{
       \rule{0ex}{2.5ex} \Phi\big(\langle\, \lambda  \mid \xi_1 \,\rangle ,
  \langle\, \delta \mid \xi_1 \,\rangle  \,;\, \frac{1}{2} \big)
            } 
\\
 &\qquad \times
              \Phi(\langle\, \lambda  \mid 2\xi_1 \,\rangle ,
                \langle\, \delta \mid 2\xi_1 \,\rangle \,;\, 1) 
\\
 &= W\big( \langle\, \lambda  \mid \xi_1 \,\rangle,
  \langle\, \delta \mid \xi_1 \,\rangle \,;\, 
   m_{\xi_1}, m_{2\xi_1} \big)\,.
\end{split}
\end{equation}
Here $\langle\,  \delta \mid \xi_1 \,\rangle  = 
 (2\ell - 1)  \langle\,  \xi_1 \mid \xi_1 \,\rangle  
 = \ell - \frac{1}{2}$.

\begin{Remark}
{\em
 Let $\lambda = k\mu_1 = 2k\xi_1$. 
Then $ \langle \lambda + \delta \mid \xi_1 \rangle = k + \ell - \frac{1}{2}$. 
Taking $k = 1$ in (\ref{ccdim}) gives $d(\mu_1) = (2\ell+1)(\ell-1) =
\binom{2\ell}{2} - \binom{2\ell}{0}$. In this case $\mu_1$ is the highest
weight of the traceless ({\em harmonic}) subspace in $\bigwedge^2 \bC^{2\ell}$
(see \cite[Cor. 5.5.17]{Goodman-Wallach}).
}
\end{Remark}

%% \textbf{\em Case 5.} {\bf The pair $({\bf F}_4,\, \Spin_{9})$. }
%           \input{f4b4exam}
           %%%%%%%% diagram: newfii.epx
%%%%%%%%%%%%%%%%%%%%%%%%%%%%%%%%%%%%%%%%%%%%%%%%%%%%%%%%%%%%%%%%%%%%%%%%%%%%%
%
%  Created       : Wed Nov 2 08:35:47 2011
%  Last Modified : 9/11/12 for short version
%
%%%%%%%%%%%%%%%%%%%%%%%%%%%%%%%%%%%%%%%%%%%%%%%%%%%%%%%%%%%%%%%%%%%%%%%%%%%%%

\vspace{1ex}
\noindent
\textbf{\em Case 5.} {\bf The pair $({\bf F}_4,\, \Spin_{9})$. }
For the Cartan subalgebra $\frt \cong
\bC^{4}$ in $\frg$ we follow \cite[Planche VIII]{Bourbaki3} and use simple
roots 
$\alpha_1 = \varepsilon_2 - \varepsilon_{3}$, \
$\alpha_2 = \varepsilon_3 - \varepsilon_{4}$, \ 
$ \alpha_3 = \varepsilon_4$, \ and \ 
$\alpha_4 = (1/2)
(\varepsilon_{1} - \varepsilon_{2}- \varepsilon_{3} - \varepsilon_{4})$.
The fundamental $H$-spherical weight is $\mu_1 = \varpi_4 = \varepsilon_1$
(see \cite{Kramer}). Thus $\Delta_0 = \{\alpha_1, \alpha_2, \alpha_{3} \}$
and hence $\dim \frc = 1$. Since $|\Supp \mu_1 | = 1$, we have $ \fra = \frc$
by Lemma \ref{rankdim.lem}. Thus 
$\frm \cong \fso_{7}(\bC)$.
When $\frt$ is identified with $\bC^{4}$ using the basis dual  
to $\{\varepsilon_1, \varepsilon_2, \varepsilon_3, \varepsilon_4\}$ then 
$\fra$ consists of all elements 
\begin{equation}
\label{f4b4.a}
  \{\bx = [\,  t\,,\, 0 \,,\, 0 \,,\, 0 \,] \,:\, t \in \bC \}\,.
\end{equation}
Let $\xi_1 \in \fra^{*}$ take the value $t$ on the element $\bx$ in
(\ref{f4b4.a}). Then $\xi_1 = \overline{\alpha_4} = \frac{1}{2}\varepsilon_1$.

The multiplicities of the restricted positive roots are as follows (details
given below).

\vspace{1ex}

\begin{center}

\begin{tabular}{|c |c|}
\hline
 restricted root   & multiplicity 
 \\
\hline
 $ \xi_1$  &  $8$
\\
\hline
 $ 2\xi_1$  &  $7$
\\
\hline
\end{tabular}

\end{center}

\vspace{1ex}

%%%%%%%%%%%%%%%%%%% 
\begin{figure}[h]
\includegraphics{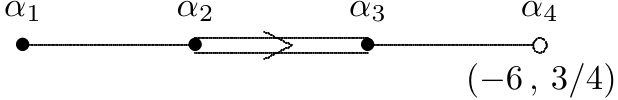} 
\caption{Marked Satake diagram for ${\bf F}_{4}/\Spin_{9}$}
\label{diagf4b4.fig}
%\vspace{-4ex}
\end{figure}
%%%%%%%%%%%%%%%%%%%%%%%%%%%%%%%%%%%%%%%%%

\noindent
From the table we calculate that 
$\delta = 11 \xi_1$.
Using the basis 
$\{\,\varepsilon_1 \,,\, \varepsilon_2 \,,\, \varepsilon_3 \,,\, 
\varepsilon_4 \, \}$ for $\frt$ and the
identification of $\frt^{*}$ with $\frt$, we can write
\[
 2\rho_{\frm} = 5\varepsilon_2 + 3\varepsilon_3 + \varepsilon_{4}\,,
\qquad
 h_{\frm}^{0} = 6\varepsilon_2 + 4\varepsilon_3 + 2\varepsilon_{4}\,.
\]
Hence 
$\varpi_{\frm}^{0} = \rho_{\frm} - \frac{1}{2} h_{\frm}^{0}
= - \frac{1}{2}
 \big(\varepsilon_2 +  \varepsilon_{3} + \varepsilon_{4} \big)$. 
 From these formulas we see that
\begin{equation}
\label{f4b4shift1} 
 \langle\, h_{\frm}^0 \mid \alpha_i \,\rangle =
 \begin{cases} 2 & \text{if $i = 1, 2, 3$\,,} \\
              -6 & \text{if $i = 4$\,,} 
 \end{cases}               
\end{equation}
and
\begin{equation}
\label{f4b4shift2}
 \langle \, \varpi_{\frm}^0 \mid \alpha_i \,\rangle =
 \begin{cases}  0 & \text{if $i = 1, 2$\,,}\\
               -1/2 & \text{if  $i = 3$\,,}\\
                3/4 & \text{if $i = 4$\,.}\\
 \end{cases}               
\end{equation}
Note that the first two cases in (\ref{f4b4shift2}) also follow from
(\ref{bpishift}). 

We now find the nests of positive restricted roots, the basic roots, and the
dimension factors $d_{\xi}(\lambda)$ for $\xi \in \Sigma^{+}$ and $\lambda
\in \Gamma(G/H)$.

\vspace{1ex}

\begin{enumerate}

\item[{\bf (i)}] Let $\xi = \xi_1 = \overline{\alpha_4}$. Then from
\cite[Planche VIII]{Bourbaki3}
\[
\quad\qquad \Phi^{+}(\xi) =   
 \{\, \beta_4 \,,\, \beta_3 \,,\, \beta_2 \,,\,  \beta_1 \,\}
   \cup
  \{ \, \gamma_1 \,,\, \gamma_2 \,,\, \gamma_3 \,,\, \gamma_4 \,\}\,,
\]
where $\beta_j = \alpha_j + \cdots + \alpha_4$\,, $\gamma_1
= \beta_2 + \alpha_3$\,,\, $\gamma_2 = \beta_1 + \alpha_3$\,,\,
  $\gamma_3 = \beta_1 + \alpha_3 + \alpha_2$\,, and 
   $\gamma_4 = \beta_1 + 2\alpha_3 + \alpha_2$\,.
The basic root in the nest is $\beta_4$ and $\dim \frn_{\xi} = 8$. From
(\ref{f4b4shift1})  the eigenvalues of 
$\frac{1}{2}\ad h_{\frm}^{0}$ on $\frn_{\xi}$ are
\[
  -3 \,,\, - 2 \,,\, - 1 \,,\,  0 \,,\, 0 \,,\, 
   1 \,,\, 2 \,,\, 3 
\]
corresponding to the roots
$\beta_4$\,,\,$\ldots$\,,\,$\beta_1$\,,\,$\gamma_1$
\,,\,$\ldots$\,,\,$\gamma_4$
enumerated in increasing length (relative to the simple roots). Thus
$\frn_{\xi}$ decomposes as the one-dimensional plus the seven-dimensional
representation of $\{e_{\frm}^0\,,\, f_{\frm}^0 \,,\, h_{\frm}^0 \}$. From
(\ref{f4b4shift2}) we have 
\[
\begin{split}
  \langle \, \varpi_{\frm}^0 \mid \beta_j \,\rangle &= 
    \begin{cases}
                1/4 & \text{if $i = 1, 2, 3$\,,}\\
                3/4 & \text{if $i = 4$\,,}\\
    \end{cases}               
\\
 \langle \, \varpi_{\frm}^0 \mid \gamma_j \,\rangle &=
    \begin{cases}
                -1/4 & \text{if $i = 1, 2, 3$\,,}\\
                -3/4 & \text{if $i = 4$\,.}\\
    \end{cases}               
\end{split}
\]
 Since
$\rho_{\frm} = \frac{1}{2}h_{\frm}^{0} + \varpi_{\frm}^{0}$, it follows that
the shifts in the formula for $d_{\xi}(\lambda)$ are
\begin{equation}
\label{f4b4rhoshift1}
\textstyle{
 -\frac{9}{4} \,,\, -\frac{7}{4} \,,\,  - \frac{3}{4} \,,\, -\frac{1}{4}\,,\, 
 \frac{1}{4} \,,\, \frac{3}{4} \,,\, \frac{7}{4}  \,,\,  \frac{9}{4} \,.
}
\end{equation}
Thus the shifts are symmetric about $0$ but are not in arithmetic progression.

\vspace{1ex}

\item[{\bf (ii)}]
Let $\xi = 2\xi_1 = \overline{\beta}$, where $\beta = \alpha_2 + 2\alpha_3 +
2\alpha_4$. Then \cite[Planche VIII]{Bourbaki3} gives
\[
\begin{split}
\ \Phi^{+}(\xi) &=   
 \{\, \beta \,,\, \beta + \alpha_1 \,,\, \beta + \alpha_1 + \alpha_2 
 \,,\,  \beta + \alpha_1 + \alpha_2 + \alpha_3 \,\}
\\
 &\   \cup
 \{\,  \beta + \alpha_1 + \alpha_2 + 2\alpha_3 \,,\, 
     \beta + \alpha_1 + 2\alpha_2 + 2\alpha_3  \,,\,
    \beta + 2\alpha_1 + 2\alpha_2 + 2\alpha_3   \,\}\,.
\end{split}
\]
Thus $\beta$ is the basic root in the
nest and  $\dim \frn_{\xi} = 7$. From
(\ref{f4b4shift1})  the eigenvalues of 
$\frac{1}{2}\ad h_{\frm}^{0}$ on $\frn_{\xi}$ are
\[
  -3 \,,\, - 2 \,,\, - 1 \,,\,  0 \,,\,   1 \,,\, 2 \,,\, 3 
\]
corresponding to the roots $\alpha \in \Phi^{+}(\xi)$ enumerated by increasing
length (relative to the simple roots). From (\ref{f4b4shift2}) we calculate
that the corresponding sequence of values of 
$\rule{0ex}{2.5ex}\langle \, \varpi_{\frm}^0 \mid \alpha \,\rangle$ is
\[
\textstyle{
  \frac{1}{2} \,,\,  \frac{1}{2} \,,\,  \frac{1}{2} \,,\,
    0 \,,\, -\frac{1}{2} \,,\, -\frac{1}{2} \,,\, -\frac{1}{2}\,.
}
\]
Since $\rho_{\frm} = \frac{1}{2}h_{\frm}^{0} + \varpi_{\frm}^{0}$, it follows
that the shifts in the formula for $d_{\xi}(\lambda)$ are
\begin{equation}
\label{f4b4rhoshift2}
\textstyle{
 -\frac{5}{2} \,,\, -\frac{3}{2} \,,\,  - \frac{1}{2} \,,\, 0 \,,\, 
 \frac{1}{2} \,,\, \frac{3}{2} \,,\, \frac{5}{2} \,.
}
\end{equation}
Thus the shifts are symmetric about $0$ but are not in arithmetic progression,
and are the same as for $\SO_{9}/\SO_{8}$ in Case 2.
\end{enumerate}

\noindent
From (\ref{f4b4rhoshift1}), (\ref{f4b4rhoshift2}), and (\ref{restdim}) we
obtain the complete dimension formula
\begin{equation}
\label{f4b4dim}
\begin{split}
  d(\lambda) &= c \,     \langle \lambda + \delta \mid 2\xi_1 \rangle 
   \prod_{j = 1, 3, 7, 9}
    \Big\{ \langle \lambda + \delta \mid \xi_1 \rangle^2 - 
      \Big(\frac{j}{4} \Big)^2 \Big\} 
\\
  &\qquad\qquad \times  \prod_{j = 1, 3, 5}
    \Big\{ \langle \lambda + \delta \mid 2\xi_1 \rangle^2 - 
      \Big(\frac{j}{2} \Big)^2 \Big\} \,,
\end{split}
\end{equation}
where $c$ is the normalizing constant to make $d(0) = 1$. Regrouping the
terms, we can write this formula in terms of the normalized  Weyl dimension
function  as
\begin{equation}
\label{f4b4dim2}
\begin{split}
  d(\lambda) &=  
   \Phi( \langle \lambda \mid \xi_1 \rangle ,
  \langle  \delta \mid \xi_1 \rangle )
\,
  \Phi\big(\langle \lambda \mid 2\xi_1 \rangle ,
   \langle  \delta \mid 2\xi_1 \rangle 
   \,;\,  \textstyle{\frac{3}{2}} \,\big)
\,
  \Phi\big(\langle \lambda  \mid 2\xi_1 \rangle ,
  \langle  \delta \mid 2\xi_1 \rangle 
   \,;\,  \textstyle{ \frac{9}{2} } \,\big)
\\
 &= W\big( \langle\, \lambda  \mid \xi_1 \,\rangle,
  \langle\, \delta \mid \xi_1 \,\rangle \,;\, 
   m_{\xi_1}, m_{2\xi_1} \big)\,.
\end{split}   
\end{equation}

\begin{Remark} 
{\em
 If $\lambda = k\mu_1$ with $k$ a nonnegative integer, then
$ \langle \lambda + \delta \mid 2\xi_1 \rangle = 
 k +  \frac{11}{2}$
since $\lambda + \delta = (k + \frac{11}{2})2\xi_1$ and $\langle\, 2\xi_1 \mid 2\xi_1
\,\rangle = 1$. Taking $k = 1$ in (\ref{f4b4dim2}) we obtain $d(\mu_1) = 26$
(the representation of the compact form of $F_4$ on the traceless $3\times 3$
hermitian matrices over the octonians; see \cite[\S 4.2]{Baez}).
}
\end{Remark}

%% \textbf{\em Case 6.} {\bf The pair $(\Spin_{7},\, {\bf G}_2)$. }
%           \input{b3g2exam}
           %%%%%%%% diagram:  newb3g2.eps 
%%%%%%%%%%%%%%%%%%%%%%%%%%%%%%%%%%%%%%%%%%%%%%%%%%%%%%%%%%%%%%%%%%%%%%%%%%%%%
%  Created       : Thu Jul 21 16:28:48 2011
%  Last Modified : 9/11/12 for short version
%
%%%%%%%%%%%%%%%%%%%%%%%%%%%%%%%%%%%%%%%%%%%%%%%%%%%%%%%%%%%%%%%%%%%%%%%%%%%%%

\vspace{1ex}
\noindent 
\textbf{\em Case 6.} {\bf The pair $(\Spin_{7},\, {\bf G}_2)$. } 
We take the matrix realization of $\frg$ as in \cite[\S
2.1.2]{Goodman-Wallach}, with Cartan subalgebra $\frt$ consisting of diagonal
matrices 
$\bx = \diag[\, \by, 0 , -\check{\by}\,]$ with $\by \in \bC^3$.  
The simple roots are $\alpha_1 = \varepsilon_1 - \varepsilon_2$, $\alpha_2 =
\varepsilon_2 - \varepsilon_3$, and $\alpha_3 = \varepsilon_3$.

From \cite{Kramer} we know that $(G, H)$ is a spherical pair and
$\Gamma(G/H)$ has generator 
$\mu_1 = \varpi_3 = 
  \frac{1}{2}(\varepsilon_1 + \varepsilon_2 + \varepsilon_3)$
(the spin representation on $\bC^8$). Hence $\Delta_{0} = \{\alpha_1,
\alpha_2\}$. If $\bx = \diag[\, \by, 0 , -\check{\by}\,]$ and $\langle
\alpha_1, \bx \rangle = \langle \alpha_2, \bx \rangle = 0$, then $\by = [a, a,
a ]$. Thus $\dim \frc = 1$. By Lemma \ref{rankdim.lem} we have $\fra = \frc$
and $\frm \cong \fsl_{3}(\bC)$. We take the restricted root 
$\xi_1 = \overline{\alpha}_3 
 = \frac{1}{3}(\varepsilon_1 + \varepsilon_2 + \varepsilon_3)$ 
as a basis for $\fra^{*}$.

The nests of restricted roots are 
\[
 \Phi^{+}(\xi_1) = 
  \{\alpha_3\,,  \alpha_2 + \alpha_3 \,, \alpha_1 + \alpha_2 + \alpha_3 \} 
\]
with basic root $\beta_1 = \alpha_3$, and 
\[
 \Phi^{+}(2\xi_1) = 
   \{ \alpha_2 + 2\alpha_3\,,
    \alpha_1 + \alpha_2 + 2\alpha_3\,, \alpha_1 + 2\alpha_2 + 2\alpha_3 \} 
\]
with basic root $\beta_2 = \alpha_2 + 2\alpha_3$.
Hence the multiplicities of the restricted positive roots are as follows.

\vspace{1ex}

\begin{center}

\begin{tabular}{|c |c|}
\hline
 restricted root   & multiplicity 
 \\
\hline
 $ \xi_1$  &  $3$
 \\
\hline
  $ 2\xi_1 $  & $3$ 
\\
\hline
\end{tabular}

\end{center}

%%%%%%%%%%%%%%%%%%% 
\begin{figure}[h]
\includegraphics{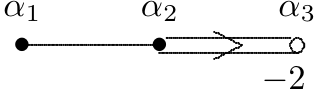} 
\caption{Marked Satake diagram for $\Spin_7/{\bf G}_2$}
\label{diagb3g2.fig}
%\vspace{-4ex}
\end{figure}
%%%%%%%%%%%%%%%%%%%%%%%%%%%%%%%%%%%%%%%%%

\noindent 
From the table we obtain $\delta = \frac{9}{2}\xi_1$.
Since $\{\alpha_1\,, \alpha_2\}$ are the simple roots for $\frm$, 
we have
$h_{\frm}^0 = 2\rho_{\frm} = 2\alpha_1 + 2\alpha_2$.
Hence 
\begin{equation}
\label{b3g2shift}
 \langle \, h_{\frm}^0 \mid \alpha_i \, \rangle =
 \begin{cases} -2 & \text{if $i = 3$\,,} \\
               2 & \text{if $i = 1, 2$\,.}
 \end{cases}               
\end{equation}

We now obtain the dimension formula. From (\ref{b3g2shift}) we have
$\langle\, h_{\frm}^0 \mid \beta_1 \, \rangle = 
\langle\, h_{\frm}^0 \mid \beta_2 \, \rangle = -2$.
Hence $k_{\xi} = 2$ and $m_{\xi} = 3$ for $\xi = \xi_1$ and $\xi = 2\xi_1$.
Thus Proposition \ref{rhoshift.prop} (3) and formula (\ref{restdim}) give
\begin{equation}
\label{b3g2dim} 
\begin{split}
 d(\lambda) &= 
     \Phi(\langle\, \lambda  \mid \xi_1 \,\rangle \,,\, 
     \langle\, \delta \mid \xi_1 \,\rangle \,;\, 1)
 \,
     \Phi(\langle\, \lambda  \mid 2\xi_1 \,\rangle \,,\,
       \langle\,  \delta \mid 2\xi_1 \,\rangle \,;\, 1)
\\
 &= W\big( \langle\, \lambda  \mid \xi_1 \,\rangle \,,\,
  \langle\, \delta \mid \xi_1 \,\rangle \,;\, 
   m_{\xi_1}, m_{2\xi_1} \big)\,.
\end{split}   
\end{equation} 

\begin{Remark}
{\em
 Let $\lambda = k\mu_1$, where $k$ is a nonnegative integer.
Then 
$\langle\, \lambda + \delta \mid \xi_1 \,\rangle = \frac{1}{2}(k + 1)$.
Using this in (\ref{b3g2dim}), we obtain
\[
  d(\lambda) = \frac{k+3}{3}\prod_{j=1}^{5}\frac{k+j}{j} \, .
\]
For $k=1$ the formula gives $d(\mu_1) = 8$ (the spin representation). For $k =
2$ the formula gives $d(2\mu_1) = 35$. 
}
\end{Remark}

\begin{Remark}
{\em
The representation of $\Spin_{7}(\bC)$ on $\bigwedge^3 \bC^7$ has highest
weight $2\mu_1$. A fundamental property of ${\bf G}_2(\bC)$ is that it has a
unique one-dimensional subspace of fixed vectors in $\bigwedge^3 \bC^7$ (see
\cite[\S 4.1]{Baez} and \cite{Agricola}).
}
\end{Remark}

%% \textbf{\em Case 7.}{\bf The pair $({\bf G}_{2},\, \SL_{3})$. }
%           \input{g2a2exam}
           %%%%%%%% diagram:   newg2a2.eps
%%%%%%%%%%%%%%%%%%%%%%%%%%%%%%%%%%%%%%%%%%%%%%%%%%%%%%%%%%%%%%%%%%%%%%%%%%%%%
%  Created       : Thu Jul 21 16:28:48 2011
%  Last Modified : 9/11/12 for short version
%
%%%%%%%%%%%%%%%%%%%%%%%%%%%%%%%%%%%%%%%%%%%%%%%%%%%%%%%%%%%%%%%%%%%%%%%%%%%%%

\vspace{1ex}
\noindent
\textbf{\em Case 7.} {\bf The pair $({\bf G}_{2},\, \SL_{3})$. }
We take ${\bf G}_2$ root system as a subset of the integer vectors in $\bR^3$
with coordinates summing to zero, with simple roots $\alpha_1 = \varepsilon_1 -
\varepsilon_2$ and $\alpha_2 = -2\varepsilon_1 + \varepsilon_2 +
\varepsilon_3$ (see \cite[Planche IX]{Bourbaki3}). The remaining positive
roots are 
$\alpha_1 + \alpha_2$, $2\alpha_1 + \alpha_2$, $3\alpha_1 + \alpha_2$,   
$3\alpha_1 + 2\alpha_2$.   

From \cite{Kramer} we know that $(G, H)$ is a spherical pair and
$\Gamma(G/H)$ has generator 
$\mu_1 = \varpi_1 = 2\alpha_1 + \alpha_2$.
Hence $\Delta_{0} = \{ \alpha_2 \}$ (see \cite[Planche X]{Bourbaki3}). It
follows from Lemma \ref{rankdim.lem} that $\fra = \frc = \{\alpha_2\}^{\perp}$
and $\frm \cong \fsl_{2}(\bC)$.  For the normalized inner product we define
\[
  \langle\, \varepsilon_i \mid \varepsilon_j \,\rangle  =
  \begin{cases} 1/3 & \mbox{if $i = j$\,,}
    \\
      0 &\mbox{if $i \neq j$\,.}
\end{cases}
\]
Then $\langle\, \alpha_2 \mid \alpha_2 \,\rangle = 2$ as required in Section
\ref{restdim.sec}, and the root $2\alpha_1 + \alpha_2= -\varepsilon_2 +
\varepsilon_3$ is a basis for $\fra$.

Let $\xi_1 = \overline{\alpha_1}$. The nests of restricted roots are 
\[
\begin{split}
 \Phi^{+}(\xi_1) &= 
  \{\alpha_1\,,\,  \alpha_1 + \alpha_2 \} \,,
\mbox{\ basic root $\beta_1 = \alpha_1$}\,,
\\
 \Phi^{+}(2\xi_1) &= 
   \{ 2\alpha_1 + \alpha_2 \} \,,
\\
 \Phi^{+}(3\xi_1) &= 
   \{ 3\alpha_1 + \alpha_2\,,\, 3\alpha_1 +  2\alpha_2 \} \,,  
\mbox{\ basic root $\beta_3 = 3\alpha_1 + 2\alpha_2$}\,.
\end{split}
\]
Hence the multiplicities of the restricted positive roots are as follows.

\vspace{1ex}

\begin{center}

\begin{tabular}{|c |c|}
\hline
 restricted root   & multiplicity 
 \\
\hline
 $ \xi_1$  &  $2$
 \\
\hline
  $ 2\xi_1 $  & $1$ 
\\
\hline
  $ 3\xi_1 $  & $2$ 
\\
\hline
\end{tabular}

\end{center}

%%%%%%%%%%%%%%%%%%% 
\begin{figure}[h]
\includegraphics{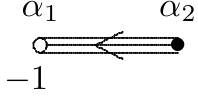} 
\caption{Marked Satake diagram for ${\bf G}_2/\SL_3$}
\label{diagg2a2.fig}
\end{figure}
%%%%%%%%%%%%%%%%%%%%%%%%%%%%%%%%%%%%%%%%%

\noindent 
From the table we obtain $\delta = 5\xi_1$.
% Since $\dim M =
% 3$, we have $\dim MN = 8 = \dim H$.
We have $h_{\frm}^0 = \alpha_2$, so that
\begin{equation}
\label{g2a2shift}
  \langle\, h_{\frm}^0 \mid \alpha_j \,\rangle  =
  \begin{cases} -1 & \mbox{if $j = 1$\,,}
    \\
      2 &\mbox{if $j = 2$\,.}
\end{cases}
\end{equation}

We now obtain the dimension formula. Let $\lambda \in \Gamma(G/H)$. For the
basic roots $\beta_1$ and $\beta_3$ we calculate from (\ref{g2a2shift}) that 
$\langle\, h_{\frm}^0 \mid \beta_j \,\rangle = -1$. Let $\xi$ be $\xi_1$ or
$3\xi_1$. Since $\dim \frn_{\xi} = 2$, Proposition \ref{rhoshift.prop} (3)
gives
\[
 d_{\xi} = \Phi\big(\langle\, \lambda  \mid \xi \,\rangle \,,\,
  \langle\,  \delta \mid \xi \,\rangle 
  \,;\, \textstyle{ \frac{1}{2} }\big)
 \,.
\]
When $\xi = 2\xi_1$, then $\dim \frn_{\xi} = 1$. Hence by Proposition
\ref{rhoshift.prop} (2) and  (\ref{restdim}) it follows that

\begin{equation}
\label{g2a2dim} 
\begin{split}
 d(\lambda) &= 
 \Phi\big(\langle\, \lambda  \mid \xi_1 \,\rangle \,,\,
   \langle\,  \delta \mid \xi_1 \,\rangle 
 \,;\, \textstyle{ \frac{1}{2} } \big)
\,
 \Phi(\langle\, \lambda \mid 2\xi_1 \,\rangle \,,\,
  \langle\, \delta \mid 2\xi_1 \,\rangle )
\\
 &\qquad \times
  \Phi\big(\langle\, \lambda \mid 3\xi_1 \,\rangle \,,\,
   \langle\,  \delta \mid 3\xi_1 \,\rangle 
   \,;\, \textstyle{\frac{1}{2} }\big)
\\
 &= W\big( \langle\, \lambda  \mid \xi_1 \,\rangle\,,
  \langle\, \delta \mid \xi_1 \,\rangle \,;\, 
   m_{\xi_1}\,,\, m_{2\xi_1}\,,\, m_{3\xi_1}\, \big)\,.
\end{split}   
\end{equation} 

\begin{Remark} 
{\em
Let $\lambda = k\mu_1 \in \Gamma(G/H)$, where $k$ is a
nonnegative integer. The normalized inner products are 
$\langle\, \xi_1 \mid \xi_1 \,\rangle = 1/6$, \ 
$\langle\, \delta \mid \xi_1 \,\rangle = 5/6$, \   
and
$\langle\, \mu_1 \mid \xi_1 \,\rangle = 1/3$ (since  $\mu_1 = 2\xi_1$).
Using these values, we can write (\ref{g2a2dim}) in terms of $k$ as
\[
 d(\lambda) =  \frac{2k+5}{5}\prod_{j=1}^{4} \frac{k+j}{j} \, .
\]
In particular, when $k = 1$ we get $d(\mu_1) =  7$ as expected. 
}
\end{Remark}

\section{Higher Rank Non-Symmetric Excellent Affine Spherical Spaces}
            \label{highrank.sec}
%           \input{highrank}
%%%%%%%%%%%%%%%%%%%%%%%%%%%%%%%%%%%%%%%%%%%%%%%%%%%%%%%%%%%%%%%%%%%%%%%%%%%%%
%  Created       : Tue Apr 12 10:36:16 2011
%  Last Modified : 8/02/12
%
%%%%%%%%%%%%%%%%%%%%%%%%%%%%%%%%%%%%%%%%%%%%%%%%%%%%%%%%%%%%%%%%%%%%%%%%%%%%%

Here is the list due to Kr\"amer \cite{Kramer} of excellent irreducible
spherical pairs $(G, H)$ of rank greater than one with $G$ simple and
simply-connected, $H$ reductive, connected, and not a symmetric subgroup of
$G$ (cf. \cite[\S 12.7]{Wolf2} for geometric descriptions). These pairs are
determined by their Lie algebras $(\frg, \frh)$. The enumeration below follows
\cite[Table 1]{Avdeev2}, which also includes all symmetric subgroups and also
includes the pairs that are not excellent. 

\begin{enumerate}

\item[4:] 
$G = \SL_{p+q}(\bC)$ and $H = \SL_{p}(\bC)\times \SL_{q}(\bC)$ with $1\leq p <
q$. Here $H$ is embedded in $G$ by $(x, y) \mapsto (x^{-1})^{t} \oplus y$ for
$x\in \SL_{p}(\bC)$ and $y \in \SL_{q}(\bC)$.

\vspace{1ex}

\item[6:]
$G = \SL_{2n+1}(\bC)$ and $H \cong \Sp_{2n}(\bC)$ with $n \geq 1$. Here $H$ is
embedded in $G$ by $x \mapsto x \oplus 1$ for $x\in \Sp_{2n}(\bC)$.

\vspace{1ex}

\item[9:]
$G = \Spin_{4p+2}(\bC)$ and $H \cong \SL_{2p+1}(\bC)$ with $p \geq 1$. Here
$H$ is embedded in $G$ by lifting the embedding $x \mapsto x \oplus
(x^{-1})^{t}$ of $\SL_{2p+1}(\bC)$ into $\SO_{4p+2}(\bC, \omega)$, where
$\omega$ the symmetric bilinear form on $\bC^{4p+2}$ with matrix $1$ on the
antidiagonal and zero elsewhere. 

\vspace{1ex}

\item[10:]
$G = \Spin_{2n+1}(\bC, \omega)$ and $H$ a covering of $\GL_{n}(\bC)$, with
$\omega$ the symmetric bilinear form on $\bC^{2n+1}$ with matrix $1$ on the
antidiagonal and zero elsewhere. Here $H$ is the connected inverse image under
the spin covering of the embedding $x \mapsto x \oplus 1 \oplus (x^{-1})^{t}$
of $\GL_{n}(\bC)$ into $\SO_{2n+1}(\bC)$.

\vspace{1ex}

\item[13:]
$G = \Spin_{9}(\bC)$ and $H = \Spin_{7}(\bC)$. Here $H$ embedded in $G$ by
lifting to $G$ the homomorphism $x \mapsto \pi_3(x) \oplus 1 $ of $H$ into
$\SO_{9}(\bC)$, where $\pi_3$ is the spin representation of $H$ of degree $8$.

\vspace{1ex}

\item[15:]
$G = \Spin_{8}(\bC)$ and $H = {\bf G}_2(\bC)$. Here $H$ embedded in $G$ by $x
\mapsto \pi_1(x) \oplus 1 $, where $\pi_1$ is the representation of $H$ of
degree $7$.

\vspace{1ex}

\item[19:]
$G = \Sp_{2\ell}(\bC)$ and $H = \bC^{\times} \times \Sp_{2\ell-2}(\bC)$ with
$\ell \geq 2$. Here $H$ is embedded by $(z, h) \mapsto \diag[\, z\,, h\,,
z^{-1}\,]$ for $z \in \bC^{\times}$ and $h \in \Sp_{2\ell-2}(\bC)$.

\vspace{1ex}

\item[26:]
$G = {\bf E}_6(\bC)$ and $H = \Spin_{10}(\bC)$. Here $H$ is embedded into the
degree 27 irreducible representation of $G$ by the map $x \mapsto 1 \oplus
\pi_1(x) \oplus \pi_5(x)$, where $\pi_1$ is the vector representation of
degree 10 and $\pi_5$ is a half-spin representation of degree 16.

\end{enumerate}

\vspace{1ex}

We now proceed to calculate the restricted root systems and the restricted
Weyl dimension functions for these pairs. When $G$ is a classical group we
take its matrix form such that the diagonal matrices in $G$ give a maximal
torus $T$. The usual inner product $\langle\, \cdot \mid \cdot \,\rangle$
making the coordinate functions $\{\varepsilon_i\}$ an orthonormal set
satisfies the normalization conditions of Section \ref{restdim.sec} for these
root systems. The choice of positive roots is indicated in each case.

%% \textbf{\em Case 1.} {\bf The pair $(\SL_{p+q},\, \SL_p \times \SL_q)$.}
%           \input{slpqexam}
           %%%%%%%% diagram: newaiii.eps
%%% created 8 Feb. 2011
%% last modified  9/11/12 for short version

\vspace{2ex}
\noindent
\textbf{\em Case 1.} {\bf The pair $(\SL_{p+q},\, \SL_p \times \SL_q)$.}
Assume $1 \leq p < q$. Let $T$ be diagonal matrices in $G$ and $U$ the
upper-triangular unipotent matrices. Then $B = TU$ is a Borel subgroup. The
simple roots are 
$\alpha_i = \varepsilon_{i} - \varepsilon_{i+1}$ and the fundamental weights
are
\[
 \varpi_i = \varepsilon_1 + \cdots + \varepsilon_{i} 
 - \frac{i}{n}\sum_{j=1}^{n} \varepsilon_{j}
\]  
with $n = p+q$ and  $i = 1, \ldots, \ell$, where $\ell = n-1$.

From \cite{Kramer} we know that $(G, H)$ is a spherical pair and $\Gamma(G/H)$
has $r = p+1$ generators
\begin{align*}
  \mu_1 =& \ \varpi_1+\varpi_{\ell}\,, \;
  \mu_2 = \varpi_2+\varpi_{\ell -1}\,, \,
  \ldots \,,
  \mu_{p-1} = \varpi_{p-1} + \varpi_{q+1} \,,
\\  
  \mu_{p} =& \ \varpi_{p} \,, \;  \mu_{p+1} = \varpi_{q} \;.
\end{align*}
Hence the support condition in Definition \ref{excellent.def} is satisfied and 
\[
 \Delta_{0} = 
  \{\alpha_{p+1}\,, \; \alpha_{p+2}\,, \ldots \,, \alpha_{q-1} \}\,.
\]
Since $|\Supp \mu_i | = 2$ for $i = 1, \ldots, p-1$, we know by Lemma
\ref{rankdim.lem} that $\dim \frc = 2(p-1) + 2 = 2p$ and $\dim \frc_0 = 2p -
(p+1) = p-1$. Thus if $p = 1$ then $\frc_{0} = 0$. Assume now $p \geq 2$ and
identify $\frt$ with $\frt^{*}$ using the form $\langle \cdot \mid \cdot
\rangle$. If $\bx = c_1\alpha_1 + \cdots + c_{\ell}\alpha_{\ell}$, 
 then the equations $\langle \mu_i \mid \bx \rangle = 0$ for $i = 1,
\ldots , r$ give the relations
\[
 c_{\ell + 1 - i} = - c_i \quad\mbox{for $i = 1, \ldots p-1$ and } \ 
 c_{p} = c_{1} = 0\,.
\]    
Hence the linearly independent set
\begin{equation}
\label{slpqc0basis}
 \{\, \alpha_{1} - \alpha_{\ell} \,, \,
  \alpha_{2} - \alpha_{\ell-1} \,, \,
  \cdots \, , \,
  \alpha_{p-1} - \alpha_{q+1} \,\} \,,
\end{equation}
which is orthogonal to $\Delta_0$, is a basis for $\frc_{0}$. 

Consider the orthogonal set of vectors in $\frt$
\[
 \left\{
 \begin{split}
    \bx_1 =&\  \diag[\, 1\,,\, 0\,,\, \ldots \,,\, 0\,,\, -1 \,] \,,
\\
  \vdots
\\
    \bx_p =&\ \diag[\, \underbrace{0\,,\, \ldots \,,\, 0\,,\, 1}_{p}\,,\, 
      \underbrace{0\,,\, \ldots \,,\, 0}_{q-p} \,,\, 
      \underbrace{-1\,,\, 0 \,,\, \ldots \,,\, 0}_{p} \, ] \,,
\\
    \bx_{p+1} =&\ \diag[\, \underbrace{s\,,\, \ldots \,,\, s}_{p}\,,\,
        \underbrace{-t \, , \, \ldots\, ,  -t }_{q-p} ,\, 
       \underbrace{ s\,,\, \ldots \,,\, s}_{p} \, ] \;,
\end{split}
 \right.
\]   
where $s = (q-p)/(q+p)$ and $t = 1 - s$. Here $s$ is chosen to make 
$\langle \, \alpha_{p} \mid \bx_{p+1} \, \rangle = 1$. 
These vectors are orthogonal to $\Delta_0$ and $\frc_0$. Since $\dim \fra =
\dim \frc - \dim \frc_0 = p+1$, it follows that 
$\{\bx_1, \ldots, \bx_{p+1} \}$ 
is a basis for $\fra$. Let $\{\xi_1, \ldots, \xi_{p+1} \}$ be the dual
basis for $\fra^{*}$:
\begin{equation}
\label{slpqbasis}
\begin{split}
  \xi_i &= \frac{1}{2}(\varepsilon_i - \varepsilon_{n+1-i})
  \quad\mbox{for $1 \leq i \leq p$\,,}
\\
  \xi_{p+1} &= \frac{1}{2p}(\varepsilon_1 + \cdots + \varepsilon_p) 
  - \frac{1}{q-p}(\varepsilon_{p+1} + \cdots + \varepsilon_{q}) 
   + \frac{1}{2p}(\varepsilon_{q+1} + \cdots + \varepsilon_n)\,. 
\end{split}
\end{equation}

The restricted root data are as follows (details given below--the entries in
the fourth and sixth columns are calculated using (\ref{slpqdelta}) and
(\ref{slpqshift})). In the left column (r) means {\em regular} and (s) {\em
singular}. 

\vspace{1ex}

\begin{center}

\begin{tabular}{|c|c |c| c | c | c |}
\hline
\small{r/s} & 
 \rule[-1.5ex]{0ex}{4.5ex} \small{restricted root} $\xi$  
  & \small{ mult.}
 &  $\langle \delta \mid \xi \rangle $
 & \parbox{1.5cm}{ $\#$ \small{basic  roots} $\beta$}
 & $\langle h_{\frm}^0 \mid \beta \rangle$
 \\
\hline
\small{(r)} & \rule[-1ex]{0ex}{3ex} $ \xi_i - \xi_j$ \quad 
 \small{($\scriptstyle{1\leq i < j \leq p}$)}  
 &  $2$ &  $j-i$  & $2$ & $0$
 \\
\hline
\small{(r)} & \rule[-1ex]{0ex}{3ex} $ \xi_i + \xi_j$ 
\quad \small{($\scriptstyle{1\leq i < j \leq p}$)  }
 &  $2$ &  $p+q+1-i-j$ & $2$ & $0$
 \\
\hline
\small{(s)} & \rule[-1ex]{0ex}{3ex} $ \xi_i \pm \xi_{p+1}$ \quad 
 \small{($\scriptstyle{1\leq i \leq p}$)}  
 &  $q - p$
 &  \rule[-1ex]{0ex}{3ex} $\frac{1}{2}(p+q + 1 -2i)$
 & $1$ & $-(q-p-1)$
 \\
\hline
\small{(r)} & \rule[-1ex]{0ex}{3ex}  $ 2\xi_i$ 
\quad \small{($\scriptstyle{1 \leq i \leq p}$) } 
 &   $1$  & $p+q+1 -2i$  & $1$ & $0$
\\
\hline
\end{tabular}

\end{center}

%%%%%%%%%%%%%%%%%%% 
\begin{figure}[h]
\includegraphics{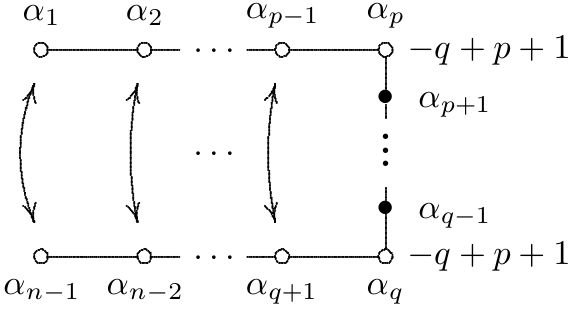} 
\caption{Marked Satake diagram for $\SL_{p+q}/\SL_p\times \SL_q$}
\label{diagaiii.fig}
\end{figure}
%%%%%%%%%%%%%%%%%%%%%%%%%%%%%%%%%%%%%%%%%

\vspace{1ex}

\noindent
From the table we calculate that 
\begin{equation}
\label{slpqdelta}
  \delta = \sum_{i=1}^{p} \big( n - 2i + 1 \big)\xi_i \,.
\end{equation}  
Note that $\langle \, \delta \mid \xi_{p+1} \, \rangle = 0$.

From the Satake diagram we see that $\frm' \cong \fsl_{q-p}(\bC)$ with simple
roots $\Delta_0$. Hence
\begin{equation}
\label{slpqh0}
  h_{\frm}^{0} = \sum_{i= 1}^{q-p} (q-p-2i+1) \varepsilon_{p+i}\,.
\end{equation}
Thus we have
\begin{equation}
\label{slpqshift}
 \langle  h_{\frm}^{0} \,,\, \alpha_i \rangle = 
\begin{cases} 
 -q+p+1 &\mbox{when $i = p$ or $i=q$\,,}
 \\
 0&\mbox{when $i<p$ or $i>q$\,,}
\\
 2&\mbox{when $p+1 \leq i \leq q-1$\,,}
\end{cases}
\end{equation}
which gives the markings on the Satake diagram.

The nests of positive roots, basic roots, and the dimension factors
$d_{\xi}(\lambda)$ for $\xi \in \Sigma^{+}$ and $\lambda \in \Gamma(G/H)$
are as follows (cases {\bf (i)} and {\bf (ii)} are empty when $p = 1$).

\vspace{1ex}

\begin{enumerate}

%%%%%%%%%%%%%%%%%%% 
\item[{\bf (i)}] For $1 \leq i < j \leq p$  let  $\xi = \xi_i - \xi_j$.
From (\ref{slpqbasis}) we see that
$\overline{\alpha_i} = \overline{\alpha_{n - i + 1}} = \xi_i - \xi_{i+1}$ 
for $i = 1, \ldots, p-1$. Hence
\[
\begin{split}
  \Phi^{+}(\xi) &= \{ \, \varepsilon_i - \varepsilon_j  \,,\,
      \varepsilon_{n+1-j} - \varepsilon_{n+1-i} \, \}
\\
  &=  \{ \, \alpha_i + \cdots + \alpha_{j-1} \,,\,
      \alpha_{n+1-j}+\cdots + \alpha_{n-i} \, \} 
\end{split}
\]
with both roots basic.

\vspace{1ex}

%%%%%%%%%%%%%%%%%%%%%%%%%%%%%%%%%%%%%%%%%%%%

\item[{\bf (ii)}] For $1 \leq i < j \leq p$ let  $\xi = \xi_i + \xi_j$.
From (\ref{slpqbasis}) we have
\[
\begin{split}
  \Phi^{+}(\xi)  &= \{ \, \varepsilon_i - \varepsilon_{n+1-j}  \,,\,
      \varepsilon_{j} - \varepsilon_{n+1-i} \, \}
\\  
  &= \{ \, \alpha_i + \cdots + \alpha_{n-j} \,,\,
      \alpha_{j} + \cdots + \alpha_{n-i} \, \} 
\end{split}
\]
with both roots basic.

From (\ref{slpqshift}) the  eigenvalues of $h_{\frm}^{0}$ on $\frn_{\xi}$
are
\[
\begin{split}
  \langle h_{\frm}^{0} \mid \alpha_i + \cdots + \alpha_{n-j} \rangle &=
   \langle h_{\frm}^{0}  \mid \alpha_p \rangle 
   + \langle h_{\frm}^{0} \mid \alpha_{p+1} + \cdots + \alpha_{q-1} \rangle 
   + \langle h_{\frm}^{0} \mid \alpha_q \rangle 
\\
  &= -(q-p-1) + 2(q-p-1) - (q-p-1) = 0\,.
\end{split}
\]
Likewise
$ \langle h_{\frm}^{0} \mid \alpha_j + \cdots + \alpha_{n-i} \rangle = 0$.

%%%%%%%%%%%%%%%%%%%%%%%%%%%%%%%%%%%%%%%%%%%%%%%%
\vspace{1ex}

\item[{\bf (iii)}] For $1 \leq i  \leq p$ let $\xi = \xi_i - \xi_{p+1}$.
From (\ref{slpqbasis}) we see that 
$\overline{\alpha_{q}} = \xi_p - \xi_{p+1}$. 
Hence
\[
\begin{split}
 \Phi^{+}(\xi)  &= \{ \, \varepsilon_{p+j} - \varepsilon_{n+1-i}  \,:\,
      1 \leq j \leq q-p \, \}
\\
  &= \{ \alpha_{p+j} + \cdots + \alpha_{n-i} \,:\,
   \  1 \leq j \leq q-p  \} \,.  
\end{split}
\]
From (\ref{slpqshift})  the lowest eigenvalue
of $h_{\frm}^{0}$ on $\frn_{\xi}$ is $-q + p +1$, and the basic root is 
$ \beta = \alpha_q + \alpha_{q+1} +\cdots + \alpha_{n-i}$\,.
%
%%%%%%%%%%%%%%%%%%%%%%%%%%%%%%%%

\vspace{1ex}

\item[{\bf (iv)}] For $1 \leq i  \leq p$ let $\xi = \xi_i + \xi_{p+1}$.
From (\ref{slpqbasis}) we see that 
 $\overline{\alpha_{p}} = \xi_p + \xi_{p+1}$.
Hence
\[
\begin{split}
 \Phi^{+}(\xi)  &= \{ \, \varepsilon_{i} - \varepsilon_{p+j}  \,:\,
      1 \leq j \leq q-p \, \}
\\      
  &= \{ \alpha_i + \cdots + \alpha_{p+j-1} \,:\,
 \   1 \leq j \leq q - p  \} \,. 
\end{split}
\]
From (\ref{slpqshift}) we see that the lowest eigenvalue of $h_{\frm}^{0}$ on
$\frn_{\xi}$ is $-q + p +1$, and the basic root is 
$  \beta = \alpha_i + \alpha_{i+1} + \cdots + \alpha_{p}$\,.
%

%%%%%%%%%%%%%%%%%%%%%%%%%%%%%%%%
\vspace{1ex}

\item[{\bf (v)}] For $1 \leq i  \leq p$ let $\xi = 2\xi_i $. Then
by (\ref{slpqbasis}) we have
\[
 \Phi^{+}(\xi) = \{\varepsilon_i - \varepsilon_{n+1-i} \}
  = \{ \alpha_i + \cdots + \alpha_{n-i}\} \,.
\]
%
%%%%%%%%%%%%%%%%%%%%%%%%%%%%%%%%

\end{enumerate}

\vspace{1ex}

\noindent
In cases (i), (ii) we have 
$  d_{\xi}(\lambda) 
 = \left[\Phi(\langle \, \lambda  \mid \xi \,\rangle \,,\,
 \langle\,  \delta \mid \xi \,\rangle)\right]^{2}$
by Proposition \ref{rhoshift.prop} (2).
In cases (iii) and (iv), since $m_{\xi} = q-p$, we conclude by
Proposition \ref{rhoshift.prop} (3) that
\[
  d_{\xi}(\lambda) =  \Phi\big(
    \langle\, \lambda \mid \xi \,\rangle \,,\,
    \langle\,  \delta \mid \xi \,\rangle \,;\,
    \textstyle{ \frac{1}{2} }(q-p-1)\big)
   \,.
\]
In case  (v) we have 
$  d_{\xi}(\lambda) = \Phi(\langle \lambda \mid \xi \rangle
   \,,\,   \langle \delta \mid \xi \rangle)$
by Proposition \ref{rhoshift.prop} (2).

\vspace{1ex}

From cases (i)--(v) and (\ref{restdim}) we obtain the full dimension formula.
Let
\[
\begin{split}
  \Xi_{0} &= \{ \xi_i \pm \xi_j \,:\, 1 \leq i < j \leq p \}\,,
\quad
 \Xi_{1} = \{ \xi_i \pm \xi_{p+1} \,:\, 1 \leq i \leq p \}\,, 
\\ 
 \Xi_{2} &= \{ 2\xi_i  \,:\, 1 \leq i \leq p \} \,.
\end{split}
\]
Then \ $\Sigma^{+}_{\rm reg} = \Xi_{0} \cup \Xi_{2}$ 
 \ and \  $\Sigma^{+}_{\rm sing} = \Xi_{1}$. 
The dimension formula is
\begin{equation}
\label{slpqdim}
\begin{split}
 d(\lambda) &= 
 \prod_{\xi \in \Xi^{+}_{0}}
\left\{ \Phi(\langle \lambda \mid \xi \rangle \,,\,
  \langle \delta \mid \xi \rangle)
\right\}^{2}  
  \prod_{\xi \in \Xi^{+}_{1}}
 \Phi\big(\langle\, \lambda \mid \xi \,\rangle \,,\,
   \langle\,  \delta \mid \xi \,\rangle 
   \,;\, \textstyle{\frac{1}{2} }(q-p-1)\big)
\\
  & \qquad \times  
 \prod_{\xi \in \Xi^{+}_{2}}
   \Phi( \langle \lambda \mid \xi \rangle \,,\,
    \langle \delta \mid \xi \rangle )
\\
  &=   
 \prod_{\xi \in \Sigma^{+}_{\rm reg}}
    W\big(\langle \lambda \mid \xi \rangle \,,\, 
  \langle \delta \mid \xi \rangle
  \,;\,  m_{\xi} \, \big)
  \prod_{\xi \in \Sigma^{+}_{\rm sing}} 
    W_{\rm sing}\big(\langle \lambda \mid \xi \rangle \,,\, 
  \langle \delta \mid \xi \rangle \,;\,  m_{\xi} \, \big)\,.
\end{split}
\end{equation}
%

%% \textbf{\em Case 2.} {\bf The pair $(\SL_{2n+1},\, \Sp_{2n})$.}
%           \input{a2ncnexam}
           %%%%%%%% no diagram 
%%%%%%%%%%%%%%%%%%%%%%%%%%%%%%%%%%%%%%%%%%%%%%%%%%%%%%%%%%%%%%%%%%%%%%%%%%%%%
%  Created       : Thu Jul 21 16:28:48 2011
%  Last Modified : 8/09/12
%
%%%%%%%%%%%%%%%%%%%%%%%%%%%%%%%%%%%%%%%%%%%%%%%%%%%%%%%%%%%%%%%%%%%%%%%%%%%%%

\vspace{2ex} 
\noindent 
\textbf{\em Case 2.} {\bf The pair $(\SL_{2n+1},\, \Sp_{2n})$.} 
From \cite{Kramer} we know that $(G, H)$ is a spherical pair and $\Gamma(G/H)
= \frX_{+}(B)$. Thus every irreducible representation of $G$ has a
one-dimensional space of $H$-fixed vectors, $\Delta_{0} = \emptyset$, $\fra =
\frc = \frt$, and the restricted roots are the same as the roots. Thus every
restricted root is regular and has multiplicity one. The dimension function is
given by Weyl's formula (\ref{restdim}).

%% \textbf{\em Case 3.} {\bf The pair $(\Spin_{4p+2},\, \SL_{2p+1})$. }
%           \input{soslexam}
           %%%%%%%% diagram:  newdiii.eps
%%%%%%%%%%%%%%%%%%%%%%%%%%%%%%%%%%%%%%%%%%%%%%%%%%%%%%%%%%%%%%%%%%%%%%%%%%%%%
%  Created       : Tue Jul 19 20:51:34 2011
%  Last Modified : 9/11/12 for short version
% 
%%%%%%%%%%%%%%%%%%%%%%%%%%%%%%%%%%%%%%%%%%%%%%%%%%%%%%%%%%%%%%%%%%%%%%%%%%%%%

\vspace{2ex}
\noindent
\textbf{\em Case 3.} {\bf The pair $(\Spin_{4p+2},\, \SL_{2p+1})$. }
Let $\frg = \fso(\bC^{2\ell}, \omega)$, where $\ell = 2p+1$ is odd and
$\omega$ is the symmetric bilinear form on $\bC^{2\ell}$ with matrix $1$ on
the antidiagonal and $0$ elsewhere. Define the involution $\theta$ as in
\cite[\S 12.3, Type DIII]{Goodman-Wallach}. Then $\frg^{\theta} \cong
\fgl_{\ell}(\bC)$; we take $\frh$ to be the subalgebra corresponding to
$\fsl_{\ell}(\bC)$ under this isomorphism.
Take $\frt$ the diagonal matrices in $\frg$. Then $\bx \in \frt$ is given by
$\bx = \diag [\, \by,\, -\check{\by}\,]$, where 
$\by = [\, \varepsilon_1(\by)\,,\, \ldots \,,\, \varepsilon_{\ell}(\by)\,] $.
The roots of $\frt$ on $\frg$ are $\pm \varepsilon_i \pm \varepsilon_j$ for
$1\leq i < j \leq \ell$. For making calculations in this case it is convenient
to use the ordered basis for $\frt^{*}$
\[
 \varepsilon_1 > - \varepsilon_2 > \varepsilon_3 > 
 \cdots > -\varepsilon_{2p} > \varepsilon_{2p+1}\,, 
\]
as in \cite[\S12.3.2, Type DIII]{Goodman-Wallach}. Let $\Phi^{+}$ be the
positive roots relative to this order (obtained from the usual system of
positive roots by $\varepsilon_i \mapsto -(-1)^i\varepsilon_i$). The simple
roots in $\Phi^{+}$ are then
\begin{equation}
\label{soslsimroot}
\begin{split}
 \alpha_1 &= \varepsilon_{1} + \varepsilon_{2}\,,\; 
 \alpha_2 = -\varepsilon_2 - \varepsilon_3\,,\;
 \ldots\,,\; 
 \alpha_{2p-1} = \varepsilon_{2p-1} + \varepsilon_{2p}\,,
\\
 \alpha_{2p} &=  -\varepsilon_{2p} - \varepsilon_{2p+1}\,,\;
 \alpha_{2p+1} = -\varepsilon_{2p} + \varepsilon_{2p+1}\,.
\end{split}
\end{equation}

Let $G = \Spin_{2\ell}(\bC)$ and $H$ the connected subgroup of $G$ with Lie
algebra $\frh$. From \cite{Kramer} we know that $(G, H)$ is a spherical pair
and $\Gamma(G/H)$ is free of rank $p+1$ with generators
\[
 \mu_1 = \varpi_2\,,\, \mu_2 = \varpi_4\,,\, \ldots\,, 
 \mu_{p} = \varpi_{2p}\,,\, \mu_{p+1} = \varpi_{2p+1}
\] 
($\mu_{p}$ and $\mu_{p+1}$ are the highest weights for the half-spin
representations of $G$). Hence the support condition in Definition
\ref{excellent.def} is satisfied and 
$ \Delta_{0} = 
  \{\alpha_{1}\,,  \alpha_{3}\,, \ldots \,, \alpha_{2p -1} \}$.

Since $|\Supp \mu_i| = 1$ for $i = 1, \ldots, p+1$, Lemma \ref{rankdim.lem}
gives $\fra = \frc$. For the choice (\ref{soslsimroot}) of simple roots we see
that $\fra$ consists of all 
$\bx = \diag[\,\by, -\check{\by}\,] \in \frt$ with 
\begin{equation}
\label{sosldiag}
  \by =  [\, a_1\,, -a_1\,, \ldots, a_p\,, -a_{p}\,, a_{p+1} \,] \;.
\end{equation}
Thus an orthogonal basis for $\fra^{*}$ is given by
\begin{equation}
\label{soslbasis}
  \xi_i = \textstyle{\frac{1}{2}}(\varepsilon_{2i-1} - \varepsilon_{2i})
  \quad\mbox{for $1 \leq i \leq p$\,,}
\qquad  \xi_{p+1} = \varepsilon_{2p+1} \,.
\end{equation}
Note that $\varepsilon_i = \overline{\alpha_{2i}}$ for $1\leq i \leq p$ and
$\varepsilon_{p+1} = -\overline{\alpha_{2p}} = \overline{\alpha_{2p + 1}}$.

The restricted root data are as follows (details given below--the entries in
the fourth and sixth columns are calculated using (\ref{sosldelta}) and
(\ref{soslshift})).

\vspace{1ex}

\begin{center}

\begin{tabular}{|c|c |c| c | c | c |}
\hline
\small{r/s} & \rule[-1.5ex]{0ex}{4.5ex} \small{restricted root} $\xi$  & 
 \small{mult.} 
 & $\langle \delta \mid \xi \rangle $ 
 & \parbox{1.5cm}{ $\#$ \small{basic  roots} \ $\beta$}
 & $\langle h_{\frm}^0 \mid \beta \rangle$
 \\
\hline
 \small{(r)} &  $ \xi_i - \xi_j$ 
  \quad \small{($\scriptstyle{1\leq i < j \leq p}$)}  &  $4$ 
 &  \rule[-1ex]{0ex}{3ex}$\frac{1}{2}(j-i)$
 & $1$  & $-2$
 \\
\hline
 \small{(r)} & $ \xi_i + \xi_j$ 
 \quad \small{($\scriptstyle{1\leq i < j \leq p}$)}  &  $4$ 
 & $4p+3 - 2(i+j)$
 & $1$  & $-2$
 \\
\hline
 \small{(s)} &  $ \xi_i \pm \xi_{p+1}$ 
 \quad \small{($\scriptstyle{1\leq i \leq p}$)}  &  $2$
 &  \rule[-1ex]{0ex}{3ex}$\frac{1}{2}(4p +3 - 4i)$
 & $1$  & $-1$
 \\
\hline
 \small{(r)} &  $ 2\xi_i$ 
 \quad \small{($\scriptstyle{1 \leq i \leq p}$)}  &   $1$ 
 & $4p + 3- 4i$
 & $1$ & $0$
\\
\hline
\end{tabular}

\end{center}

% 
%%%%%%%%%%%%%%%%%%% 
\begin{figure}[h]
\includegraphics{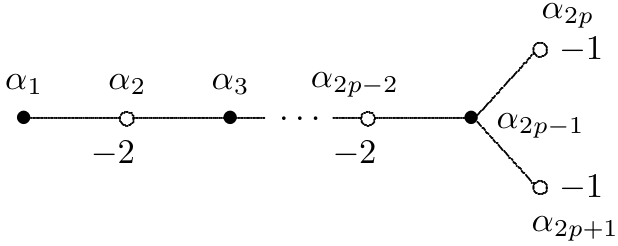} 
\caption{Marked Satake diagram for $\Spin_{4p+2}/\SL_{2p+1}$ with $p \geq 1$}
\label{diagdiii.fig}
\end{figure}
%%%%%%%%%%%%%%%%%%%%%%%%%%%%%%%%%%%%%%%%%
%

\vspace{1ex}

\noindent 
From the table we calculate that the element $\delta$ in this case is
\begin{equation}
\label{sosldelta}
  \delta = \sum_{i=1}^{p} \big(4(p - i) + 3 \big)\xi_i \,.
\end{equation}  
Note that $\langle \, \delta \mid \xi_{p+1} \, \rangle = 0$.

Since $\frc_0 = 0$, we see from the Satake diagram that $\frm \cong
\fsl_{2}(\bC) \times \cdots \times \fsl_{2}(\bC)$ \ ($p$ factors) and
$  
h_{\frm}^{0} = \alpha_{1} + \alpha_{3} + \cdots +  \alpha_{2p-1}
$
when we identify $\frt$ with $\frt^{*}$ using the form $\langle \cdot \mid
\cdot \rangle$. Thus
\begin{equation}
\label{soslshift}
 \langle \, h_{\frm}^{0} \,,\, \alpha_i \, \rangle =
 \begin{cases} 2 & \text{if $i = 2k-1$ for $1 \leq k \leq p$\,,} \\
               -2 & \text{if $i = 2k$ for $1 \leq k \leq p$\,,}\\
               -1 & \text{if $i = 2p$ or $i = 2p+1$\,,}
 \end{cases}               
\end{equation}
which gives the markings in the Satake diagram.

The nests of positive roots, the basic roots, and the dimension factors
$d_{\xi}(\lambda)$ for $\xi \in \Sigma^{+}$ and $\lambda \in \Gamma(G/H)$
are as follows.

\vspace{1ex}

\begin{enumerate}

%%%%%%%%%%%%%%%%%%% 
\item[{\bf (i)}] For $1 \leq i < j \leq p$ let $\xi = \xi_i - \xi_j$.
Then 
\begin{align*}
 \Phi^{+}(\xi) &= \{\varepsilon_{2i-1} - \varepsilon_{2j-1}\,, 
 \varepsilon_{2i-1} + \varepsilon_{2j}\,,
  -\varepsilon_{2i} - \varepsilon_{2j-1}\,,
   -\varepsilon_{2i} + \varepsilon_{2j} \}
\\
 &= \{ \beta \,,\,\alpha_{2i-1} + \beta \,,\,\beta + \alpha_{2j-1} \,,\,
 \alpha_{2i-1} + \beta + \alpha_{2j-1} \} \,,
\end{align*}
where $\beta = \alpha_{2i} + \alpha_{2i+1} + \cdots + \alpha_{2j-2}$. From
(\ref{soslshift}) we see that $\beta$ is the basic root in $\Phi^{+}(\xi)$.

%%%%%%%%%%%%%%%%%%%%%%%%%%%%%%%%%%%%%%%%%%%%

\vspace{1ex}

\item[{\bf (ii)}]  For $1 \leq i < j \leq p$ let $\xi = \xi_i + \xi_j$. 
Then  
\begin{align*}
 \Phi^{+}(\xi) &= \{\varepsilon_{2i-1} + \varepsilon_{2j-1}\,, 
 \varepsilon_{2i-1} - \varepsilon_{2j}\,,
  -\varepsilon_{2i} + \varepsilon_{2j-1}\,,
   -\varepsilon_{2i} - \varepsilon_{2j} \}
\\
 &= \{ \beta \,,\,\alpha_{2i-1} + \beta \,,\,\beta + \alpha_{2j-1} \,,\,
 \alpha_{2i-1} + \beta + \alpha_{2j-1} \} \,,
\end{align*}
where 
$
  \beta = \alpha_{2i} +  \cdots + \alpha_{2j-1}
 + 2\alpha_{2j} + \cdots + 2\alpha_{2p-1} + \alpha_{2p} + \alpha_{2p+1}
$
is the basic root.

\vspace{1ex}
%%%%%%%%%%%%%%%%%%%%%%%%%%%%%%%%%%%%%%%%%%%%%%%%

\item[{\bf (iii)}] For $1 \leq i  \leq p$ let $\xi = \xi_i - \xi_{p+1}$.
Then
\[
 \Phi^{+}(\xi) = \{\varepsilon_{2i-1} - \varepsilon_{2p+1}\,, 
 -\varepsilon_{2i} - \varepsilon_{2p+1}  \}
 = \{ \beta \,,\,\alpha_{2i-1} + \beta  \} \,,
\]
where
$ \beta = \alpha_{2i} +  \cdots + \alpha_{2p}$ is the basic root.

%%%%%%%%%%%%%%%%%%%%%%%%%%%%%%%%
\vspace{1ex}

\item[{\bf (iv)}]
 For $1 \leq i  \leq p$ let  $\xi = \xi_i + \xi_{p+1}$. Then
\[
 \Phi^{+}(\xi) = \{\varepsilon_{2i-1} + \varepsilon_{2p+1}\,,\,
  -\varepsilon_{2i} + \varepsilon_{2p+1} \} 
  = \{ \beta\,,\, \beta + \alpha_{2p-1} \} \,, 
\]
where
$ \beta = \alpha_{2i} +  \cdots + \alpha_{2p-1} + \alpha_{2p+1}$
is the basic root.

%%%%%%%%%%%%%%%%%%%%%%%%%%%%%%%%

\vspace{1ex}

\item[{\bf (v)}] For $1 \leq i  \leq p$ let  $\xi = 2\xi_i $. Then
\[
 \Phi^{+}(\xi) = \{ \varepsilon_{2i-1} - \varepsilon_{2i}\} = \{ \beta \}\,,
\]
where 
$ \beta  = \alpha_{2i-1}
 + 2\alpha_{2i} + \cdots + 2\alpha_{2p-1} + \alpha_{2p} + \alpha_{2p+1}.
$

%%%%%%%%%%%%%%%%%%%%%%%%%%%%%%%%

\end{enumerate}

\vspace{1ex}

\noindent
In cases (i) and (ii) since $k_{\xi} = 2$ and  
$m_{\xi} =  k_{\xi} +2$, Proposition \ref{rhoshift.prop} (4) gives
\[
 d_{\xi}(\lambda) = \Phi(\langle\, \lambda  \mid \xi \,\rangle \,,\,
       \langle\,  \delta \mid \xi \,\rangle) \,
  \Phi(\langle\, \lambda  \mid \xi \,\rangle \,,\, 
       \langle\,  \delta \mid \xi \,\rangle \,;\, 1 )
 \,.
\]
In cases (iii) and (iv) since $k_{\xi} = 1$ and  
$m_{\xi} =  k_{\xi} + 1$, Proposition \ref{rhoshift.prop} (3) gives
\[
 d_{\xi}(\lambda) = \Phi\big(
   \langle\, \lambda \mid \xi \,\rangle \,,\, 
  \varphi(\langle\,  \delta \mid \xi \,\rangle 
   \,;\, \textstyle{\frac{1}{2}})
  \,.
\]
In case (v) since $k_{\xi} =0$, 
Proposition \ref{rhoshift.prop} (2) gives
\[
  d_{\xi}(\lambda) = \Phi(\langle\, \lambda  \mid \xi \,\rangle 
    \,,\,
     \langle\,  \delta \mid \xi \,\rangle 
   )\,.
\]

\vspace{1ex}

From cases (i)--(v) and (\ref{restdim}) we obtain the full dimension formula.
Let
\[
\begin{split}
  \Xi^{+}_{0} &= \{ \xi_i \pm \xi_{j} \,:\, 1 \leq i < j \leq p \}\,,
\quad
  \Xi^{+}_{1} = \{ \xi_i \pm \xi_{p+1} \,:\, 1 \leq i \leq p \}\,,
\\
  \Xi^{+}_{2} &= \{ 2\xi_i  \,:\, 1 \leq i \leq p \}\,.
\end{split}
\]
Then \  $\Sigma_{\rm reg}^{+} = \Xi^{+}_{0} \cup  \Xi^{+}_{2}$ \ and \ 
 $\Sigma_{\rm sing}^{+} = \Xi^{+}_{1}$. 
The dimension formula  is
\begin{equation}
\label{sosldim}
\begin{split}
 d(\lambda) &= 
 \prod_{\xi \in \Xi^{+}_{0}}
 \Phi(\langle\, \lambda  \mid \xi \,\rangle ,   
   \langle\,  \delta \mid \xi \,\rangle) \,
  \Phi(\langle\, \lambda  \mid \xi \,\rangle ,
   \langle\,  \delta \mid \xi \,\rangle \,;\, 1 )
\\
  & \qquad \times  
  \prod_{\xi \in \Xi^{+}_{1}}
 \Phi\big(\langle\, \lambda \mid \xi \,\rangle ,
 \langle\,  \delta \mid \xi \,\rangle 
   \,;\, \textstyle{ \frac{1}{2} } \big)
 \displaystyle{ \prod_{\xi \in \Xi^{+}_{2}} }
     \Phi( \langle\, \lambda \mid \xi \,\rangle ,
     \langle\,  \delta \mid \xi \,\rangle )
\\
  &=   
 \prod_{\xi \in \Sigma^{+}_{\rm reg}}
    W\big(\langle \lambda \mid \xi \rangle , 
  \langle \delta \mid \xi \rangle,
  \,;\,  m_{\xi} \, \big)\,
 \prod_{\xi \in \Sigma^{+}_{\rm sing}} 
    W_{\rm sing}\big(\langle \lambda \mid \xi \rangle , 
  \langle \delta \mid \xi \rangle, \,;\,  m_{\xi} \, \big)\,.
\end{split}
\end{equation}
%

%% \textbf{\em Case 4.} {\bf The pair $(\Spin_{2\ell+1},\, \GL_{\ell})$. }
%           \input{soglexam}
           %%%%%%%% no diagram 
%%%%%%%%%%%%%%%%%%%%%%%%%%%%%%%%%%%%%%%%%%%%%%%%%%%%%%%%%%%%%%%%%%%%%%%%%%%%%
%  Created       : Thu Jul 21 16:28:48 2011
%  Last Modified : 8/09/12
%
%%%%%%%%%%%%%%%%%%%%%%%%%%%%%%%%%%%%%%%%%%%%%%%%%%%%%%%%%%%%%%%%%%%%%%%%%%%%%

\vspace{2ex}
\noindent
 \textbf{\em Case 4.} {\bf The pair $(\Spin_{2\ell+1},\, \GL_{\ell})$. }
By \cite{Kramer} the fundamental $H$-spherical highest weights are
$\varpi_1$, $\varpi_2$, $\ldots$, $\varpi_{n-1}$, $2\varpi_{n}$.
Thus $\Gamma(G/H)$ consists of all dominant
weights with the coefficient of $\varpi_{n}$ even. Hence $\Delta_{0} =
\emptyset$, $\fra = \frt$, $\frm = 0$, and $N = U$.
The restricted roots coincide with the roots, and all root multiplicities are
$1$. Thus every restricted root is regular and has multiplicity one. The
dimension function is given by Weyl's product formula (\ref{restdim}).

%% \textbf{\em Case 5.} {\bf The pair $(\Spin_{9},\, \Spin_7)$. }
%           \input{b4b3exam}
           %%%%%%%% diagram:   newb4b3.eps
%%%%%%%%%%%%%%%%%%%%%%%%%%%%%%%%%%%%%%%%%%%%%%%%%%%%%%%%%%%%%%%%%%%%%%%%%%%%%
%  Created       : Thu Jul 21 16:28:48 2011
%  Last Modified : 9/11/12 for short version
%
%%%%%%%%%%%%%%%%%%%%%%%%%%%%%%%%%%%%%%%%%%%%%%%%%%%%%%%%%%%%%%%%%%%%%%%%%%%%%

\vspace{2ex}
\noindent
\textbf{\em Case 5.} {\bf The pair $(\Spin_{9},\, \Spin_7)$. }
With the matrix form of $\frg$ and $\frt$ chosen
as in \cite[\S 3.1]{Goodman-Wallach}, the Cartan subalgebra consists of
diagonal matrices 
$\bx = \diag[\, \by\,, 0\,,  -\check{\by}\,]$ with $\by \in \bC^4$.  
The simple roots are $\alpha_1 = \varepsilon_1 - \varepsilon_2$, 
$\alpha_2 = \varepsilon_2 - \varepsilon_3$,
$\alpha_3 = \varepsilon_3 - \varepsilon_4$, and 
$\alpha_4 = \varepsilon_4$. 
From \cite{Kramer} one knows that $(G, H)$ is a
spherical pair with fundamental $H$-spherical highest weights
$\mu_1 = \varpi_1$ and $\mu_2 = \varpi_4$.
Hence the support condition in Definition
\ref{excellent.def} is satisfied and $\Delta_{0} = \{\alpha_2, \alpha_3\}$.
We can write
\begin{equation}
\label{b4b3fundwt}
 \mu_1 = \varepsilon_1\,,
\quad 
 \mu_2 = \textstyle{ \frac{1}{2}}\big(\varepsilon_1 + \varepsilon_2 +
      \varepsilon_3 + \varepsilon_4\big) \,.
\end{equation}

Since $|\Supp \mu_i| = 1$ for $i = 1, 2$, Lemma \ref{rankdim.lem} gives $\fra
= \frc$ and $\frm \cong \fsl_{3}(\bC)$.
Using $\alpha_2 = \varepsilon_2 - \varepsilon_3$ and 
$\alpha_3 = \varepsilon_3 - \varepsilon_4$ we thus obtain
\begin{equation}
\label{b4b3diag}
 \fra = \Ker \alpha_1 \cap \Ker \alpha_2 = 
 \{ \diag[\, \by, 0, -\check{\by}\,] \,:\, 
   \by = [\, a, b, b, b \,] \} \,.
\end{equation}
We take
\begin{equation}
\label{b4b3basis}
\xi_{1} = \varepsilon_{1}
\quad\mbox{and}\quad  
\xi_2 = \textstyle{
  \frac{1}{3}}(\varepsilon_{2} + \varepsilon_{3} + \varepsilon_{4})
\end{equation}
as an orthogonal basis for $\fra^{*}$. 

The restricted root data are as follows (details given below; the fourth and
sixth columns are calculated from (\ref{b4b3delta}) and (\ref{b4b3shift})).

\vspace{1ex}

\begin{center}

\begin{tabular}{ |c |c|c| c | c | c |}
\hline
r/s & \rule[-1ex]{0ex}{3ex}  restricted root $\xi$  & mult.
 & $\langle \delta \mid \xi \rangle $ 
 & $\#$ basic roots  $\beta$ & $\langle h_{\frm}^0 \mid \beta \rangle$
\\
\hline
 (r) & $ \xi_1$  &  $1$ & $7/2$ & $1$ & $0$
 \\
\hline
 (r) & $\xi_2$ &  $3$ & $3/2$ &  $1$ & $-2$
 \\
\hline
 (r) &  $ \xi_1 - \xi_2$ & $3$ & $2$ & $1$ & $-2$
\\
\hline 
 (r) & $ \xi_1 + \xi_2$ & $3$ & $5$ & $1$ & $-2$ 
\\
\hline
 (r) &$2\xi_2$  &  $3$ & $3$ & $1$ & $-2$
 \\
\hline
\end{tabular}

\end{center}

\vspace{1ex}

%%%%%%%%%%%%%%%%%%% 
\begin{figure}[h]
\includegraphics{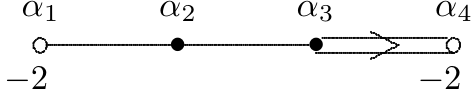} 
\caption{Marked Satake diagram for $\Spin_9/\Spin_7$}
\label{diagb4b3.fig}
\end{figure}
%%%%%%%%%%%%%%%%%%%%%%%%%%%%%%%%%%%%%%%%%

\noindent
From the table we calculate that 
\begin{equation}
\label{b4b3delta}
  \delta = \frac{1}{2} \big(7\xi_1 + 9\xi_2 \big) \,.
\end{equation}  

Since $\frm \cong \fsl_3$ and the positive roots of $\frm$ are $\alpha_2$,
$\alpha_3$, $\alpha_2 + \alpha_3$, we have
$  h_{\frm}^{0} = 2\rho_{\frm} =  2\alpha_{2} + 2\alpha_{3}$\,.
 Hence 
\begin{equation}
\label{b4b3shift}
 \langle \, h_{\frm}^0 \mid \alpha_i \, \rangle =
 \begin{cases} -2 & \text{if $i = 1, 4$\,,} \\
               2 & \text{if $i = 2, 3$\,,}
 \end{cases}               
\end{equation}
which gives the markings in the Satake diagram.

The nests of positive roots, the basic roots, and the dimension factors
$d_{\xi}(\lambda)$ for $\xi \in \Sigma^{+}$ and $\lambda \in \Gamma(G/H)$
are as follows.

\vspace{1ex}

\begin{enumerate}

\item[{\bf (i)}] 
Let  $\xi = \xi_1$. Then
$ \Phi^{+}(\xi) =  \{ \alpha_1 + \alpha_2 + \alpha_3 + \alpha_{4} \}$.
%

%%%%%%%%%%%%%%%%%%%%%%%%%%%%%%%%%%%%%%%%%%%%%

\vspace{1ex}

%%%%%%%%%%%%%%%%%%% 
\item[{\bf (ii)}] Let $\xi = \xi_1 - \xi_2$. Then 
\[
\qquad \Phi^{+}(\xi) =  
   \{\varepsilon_1 - \varepsilon_i \,:\,  2 \leq i \leq  4 \}
 = \{ \beta, \, \beta+\alpha_3, \, 
  \beta+\alpha_3 + \alpha_4   \} \,,
\]  
where $\beta = \alpha_1 + \alpha_{2}$ is the basic root. 
%

%%%%%%%%%%%%%%%%%%%%%%%%%%%%%%%%%%%%%%%%%%%%%%%%%%%%%%%

\vspace{1ex}

\item[{\bf (iii)}] 
Let $\xi = \xi_2$. Then
\[
 \Phi^{+}(\xi) =  \{ \varepsilon_i \,:\,  2 \leq i \leq  4 \}
 = \{ \beta, \, \beta+\alpha_3, \, 
  \beta+\alpha_3 + \alpha_2   \} \,,
\]  
where $\beta = \alpha_{4}$ is the basic root.

%%%%%%%%%%%%%%%%%%%%%%%%%%%%%%%%%%%%%%%%%%%%%%%%%%%%%%%

\vspace{1ex}

\item[{\bf (iv)}]
Let  $\xi = 2\xi_2$. Then
\[
\qquad \Phi^{+}(\xi) =  
  \{ \varepsilon_i + \varepsilon_j \,:\,  2 \leq i < j \leq  4 \}
  = \{ \beta,\, \beta+\alpha_2,\, \beta+\alpha_2 + \alpha_3  \} \,,
\]  
where 
 $\beta = \alpha_3 + 2\alpha_{4}$ 
is the basic root.

%%%%%%%%%%%%%%%%%%%%%%%%%%%%%%%%%%%%%%%%%%%%%%%%%%%%%%%

\vspace{1ex}

\item[{\bf (v)}] 
Let $\xi = \xi_1 + \xi_2$.  Then
\[
 \Phi^{+}(\xi) =  \{ \varepsilon_1 + \varepsilon_j \,:\,  2 \leq  j \leq  4 \}
  = \{ \beta,\, \beta+\alpha_3,\, \beta + \alpha_2 + \alpha_3  \} \,,
\]  
where 
 $\beta = \alpha_1 + \alpha_2 + \alpha_3 + 2\alpha_{4}$ 
is the basic root. 

\end{enumerate}

\vspace{1ex}

\noindent
In case (i) Proposition \ref{rhoshift.prop} (2) gives
$  d_{\xi}(\lambda) = 
    \Phi(\langle\, \lambda  \mid \xi \,\rangle \,,\,
     \langle\,  \delta \mid \xi \,\rangle) $.
In cases (ii)-(iv) we have  $k_{\xi} = 2$ and  
$\dim \frn_{\xi} = k_{\xi} + 1$. Hence Proposition \ref{rhoshift.prop} (3)
gives
\[
 d_{\xi}(\lambda) = \Phi(\langle\, \lambda  \mid \xi \,\rangle \,,\, 
     \langle\,  \delta \mid \xi \,\rangle \,;\, 1 )
\]
in all these cases. Thus from (\ref{restdim}) the dimension formula is
\begin{equation}
\label{b4b3dim}
\begin{split}
 d(\lambda) &= \Phi( \langle\, \lambda\mid \xi_1 \,\rangle
     \langle\,  \delta \mid \xi_1 \,\rangle )
\,
 \Phi(\langle\, \lambda  \mid \xi_2 \,\rangle ,
  \langle\,  \delta \mid \xi_2 \,\rangle \,;\, 1 )
\\
 &\qquad \times 
 \Phi(\langle\, \lambda  \mid 2\xi_2 \,\rangle,
  \langle\,  \delta \mid 2\xi_2 \,\rangle \,;\, 1 )
 \prod_{\xi = \xi_1 \pm \xi_2}
 \Phi(\langle\, \lambda  \mid \xi \,\rangle ,
   \langle\,  \delta \mid \xi \,\rangle \,;\, 1 )
\\
  &=  \prod_{\xi \in \Sigma_0^{+}}
   W\big(\langle\, \lambda  \mid \xi \,\rangle, 
 \langle\,  \delta \mid \xi \,\rangle  \,;\, m_{\xi}\,, m_{2\xi}\,  \big)\,.
 \end{split} 
\end{equation}
Here \ 
$ \Sigma_{0}^{+} = \Sigma_{\rm reg}^{+}  
 = \{\,\xi_1\,,\, \xi_2\,,\, \xi_1 \pm \xi_2\,\}$
and there are no singular roots. However some of the factors in the dimension
formula occur for rank-one symmetric spaces, while some of the factors occur
for the rank-one non-symmetric space $\Spin_7/{\bf G}_2$.

\begin{Remark} 
{\em
 If $\lambda =
k_1\mu_1 + k_2\mu_2$, then formula (\ref{b4b3dim})  can be written as
\begin{equation}
\label{b4b3dim2}
 d(\lambda)= c \, (2k_1 + k_2 + 7) (k_2 + 3) \binom{k_2 + 5}{5}
   \prod_{j=1}^{3} (k_1 + k_2 + j + 3) \,,
\end{equation}
where the normalizing constant $c = 2/7!$. For example, when $\lambda = \mu_1$
formula (\ref{b4b3dim2}) gives $d(\varpi_1) = 9$ (the vector representation on
$\bC^9$). When $\lambda = \mu_2$, the formula gives $d(\varpi_4) = 16$ (the
spin representation).
}
\end{Remark}

%% \textbf{\em Case 6.} {\bf The pair $({\Spin}_8,\, {\bf G}_{2})$. }
%           \input{d4g2exam}
           %%%%%%%% diagram:   newd4g2.eps
%%%%%%%%%%%%%%%%%%%%%%%%%%%%%%%%%%%%%%%%%%%%%%%%%%%%%%%%%%%%%%%%%%%%%%%%%%%%%
%  Created       : Thu Jul 21 16:28:48 2011
%  Last Modified : 9/11/12 for short version
%
%%%%%%%%%%%%%%%%%%%%%%%%%%%%%%%%%%%%%%%%%%%%%%%%%%%%%%%%%%%%%%%%%%%%%%%%%%%%%

\vspace{1ex}
\noindent
\textbf{\em Case 6.} {\bf The pair $({\Spin}_8,\, {\bf G}_{2})$. }
We take the matrix realization of $\frg$ with
Cartan subalgebra $\frt$ consisting of diagonal matrices 
$\bx = \diag[\, \by,  -\check{\by}\,]$ with $\by \in \bC^4$.  
The simple roots are $\alpha_1 = \varepsilon_1 - \varepsilon_2$, \,
$\alpha_2 = \varepsilon_2 - \varepsilon_3$, \,
$\alpha_3 = \varepsilon_3 - \varepsilon_4$, 
\, and \,
$\alpha_4 = \varepsilon_3 + \varepsilon_4$. 

From \cite{Kramer} we know that $(G, H)$ is a spherical pair and
$\Gamma(G/H)$ has generators
\begin{equation}
\label{d4g2fundwt}
\begin{split}
  \mu_1 &= \varpi_1 = \varepsilon_1\,,
\\
 \mu_2 &= \varpi_3 = 
\textstyle{ \frac{1}{2}}(\varepsilon_1 + \varepsilon_2 + \varepsilon_3 
   - \varepsilon_4)\,,
\\
  \mu_3 &= \varpi_4 = 
\textstyle{ \frac{1}{2}}(\varepsilon_1 + \varepsilon_2 + \varepsilon_3 
   + \varepsilon_4)\,.
\end{split}
\end{equation} 
Hence the support condition in Definition \ref{excellent.def} is satisfied
and $\Delta_{0} = \{\alpha_2\}$. Since $|\Supp \mu_i | = 1 $ for $i = 1,2,3$,
Lemma \ref{rankdim.lem} gives $\dim \frc = 3$ and $\fra = \frc$. 

If $\bx = \diag[\, \by, -\check{\by}\,] \in \frt$ and
$\langle \alpha_2, \bx \rangle = 0$, then 
$\by = [\, a, b, b, c \,]$. 
We take
\begin{equation}
\label{d4g2basis}
  \xi_{1} = \varepsilon_{1}\,,
\quad
  \xi_2 = \textstyle{
  \frac{1}{2}}(\varepsilon_{2} + \varepsilon_{3})\,,
\quad\mbox{and}\quad  
 \xi_3 =  \varepsilon_{4}
\end{equation}
as an orthogonal basis for $\fra^{*}$. 
The restricted root data  are as follows (details
given below--the entries in
the fourth and sixth columns are calculated using (\ref{d4g2delta}) and
(\ref{d4g2shift})).

\vspace{1ex}

\begin{center}

\begin{tabular}{|c |c |c| c | c | c |}
\hline
 r/s &  \rule[-1ex]{0ex}{3ex}  restricted root $\xi$  & mult.
  & $\langle \delta \mid \xi \rangle $ 
  & $\#$ basic roots  $\beta$ & $\langle h_{\frm}^0 \mid \beta \rangle$
 \\
\hline
(s) & $ \xi_1 - \xi_2 $\,,\,  $ \xi_2 - \xi_3$\,,\, $ \xi_2 + \xi_3 $   &  $2$
 & $3/2$  & $1$ & $-1$
 \\
\hline
(r) & $ \xi_1 - \xi_3 $\,,\, $ \xi_1 + \xi_3$\,,\, $ 2\xi_2  $   &  $1$
 & 3  & $1$ & $0$
 \\
\hline
(s) &  $\xi_1 + \xi_2$  &  $2$ &  $9/2$  & $2$ & $-1$
 \\
\hline
\end{tabular}

\end{center}

%%%%%%%%%%%%%%%%%%% 
\begin{figure}[h]
\includegraphics{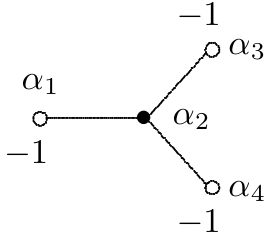} 
\caption{Marked Satake diagram for $\Spin_{8}/{\bf G}_2$}
\label{diagd4g2.fig}
\end{figure}
%%%%%%%%%%%%%%%%%%%%%%%%%%%%%%%%%%%%%%%%%

\noindent 
From the table we obtain
\begin{equation}
\label{d4g2delta}
  \delta = 3\xi_1 + 3\xi_2 \,.
\end{equation}  
Note that $\langle\, \delta \mid \xi_3 \,\rangle = 0$.

From the Satake diagram we see that $\frm \cong \fsl(2, \bC)$ and
$h_{\frm}^{0} = \alpha_2$. Thus
\begin{equation}
\label{d4g2shift}
 \langle \, h_{\frm}^0 \mid \alpha_i \, \rangle =
 \begin{cases} -1 & \text{if $i = 1, 3,  4$ \,,} \\
               2 & \text{if $i = 2$ \,,}
 \end{cases}               
\end{equation}
which gives the markings in the diagram. The nests of positive roots, basic
roots, and the dimension factors $d_{\xi}(\lambda)$ for $\xi \in \Sigma^{+}$
and $\lambda \in \Gamma(G/H)$ are as follows.

\vspace{1ex}

\begin{enumerate}

%%%%%%%%%%%%%%%%%%% 
\item[{\bf (i)}] 
Let $\xi$ be $\xi_1 - \xi_{2}$, $\xi_2 - \xi_3$, or $\xi_2 + \xi_3$. Then 
$ \Phi^{+}(\xi) =  \{\beta \,,\, \beta + \alpha_2\}$,
where the basic root $\beta$ is $\alpha_1$, $\alpha_3$, or $\alpha_4$,
respectively. 
%

%%%%%%%%%%%%%%%%%%%%%%%%%%%%%%%%

\vspace{1ex}

\item[{\bf (ii)}]  Let  $\xi = \xi_1 + \xi_2$. Then
$ \Phi^{+}(\xi) =  \{\beta \,,\, \beta + \alpha_2\}$,
where the basic root $\beta = \alpha_1 + \alpha_2 + \alpha_3 + \alpha_4$. 
%

%%%%%%%%%%%%%%%%%%%%%%%%%%%%%%%%

\vspace{1ex}

\item[{\bf (iii)}] 
Let $\xi = \xi_1 - \xi_3$, \  $\xi_1 + \xi_3$, \ or $2\xi_2$.  Then
$ \Phi^{+}(\xi) =  \{\alpha_i + \alpha_j + \alpha_2\}$
with $(i,j)$ equal to $(1,3)$, \ $(2,4)$, \ or $(1,4)$, respectively. 
%

%%%%%%%%%%%%%%%%%%%%%%%%%%%%%%%%
\end{enumerate}

\vspace{1ex}

\noindent
In cases (i) and (ii) we have  $k_{\xi} = 1$ and  
$m_{\xi} =  k_{\xi} + 1$, so Proposition \ref{rhoshift.prop} (3) gives
\[
 d_{\xi}(\lambda) =  \Phi\big(\langle\, \lambda \mid \xi \,\rangle 
 \,,\, \langle\, \delta \mid \xi \,\rangle 
 \,;\, \textstyle{\frac{1}{2}} \big)
\]
in these cases. In case (iii) Proposition \ref{rhoshift.prop} (2) gives
$
d_{\xi}(\lambda) = \Phi( \langle \, \lambda \mid \xi \, \rangle 
    \,,\, \langle \,  \delta \mid \xi \, \rangle)$.

\vspace{1ex}
Let 
\ $ \Sigma^{+}_{\rm reg} = 
  \{ \xi_1 \pm \xi_3 \,,\, 2\xi_2 \,  \}$ \ and \ 
$ \Sigma^{+}_{\rm sing} = \{ \xi_1 \pm \xi_2 \,,\, \xi_2 \pm \xi_3 \}$.
From (\ref{restdim}) we obtain
\begin{equation}
\label{d4g2dim}
\begin{split}
 d(\lambda) &=  \prod_{\xi \in \Sigma^{+}_{\rm reg}}
 \Phi( \langle \, \lambda \mid \xi \, \rangle ,
  \langle \,  \delta  \mid \xi \, \rangle )
 \prod_{\xi \in \Sigma^{+}_{\rm sing}}
 \Phi\big(\langle\, \lambda \mid \xi \,\rangle ,
 \langle\, \delta \mid \xi \,\rangle \,;\, \textstyle{ \frac{1}{2} } \big)
\\
 &= 
 \prod_{\xi \in \Sigma^{+}_{\rm reg}}
 W(\langle\, \lambda \mid \xi \, \rangle \,,\, 
  \langle\, \delta  \mid \xi\,\rangle   \,;\, m_{\xi}\,) 
% \\
% &\qquad \times
 \prod_{\xi \in \Sigma^{+}_{\rm sing}}
 W_{\rm sing}(\langle\, \lambda \mid \xi \,\rangle \,,\, 
 \langle \delta \mid \xi \,\rangle 
  \,;\, m_{\xi}\,) \,. 
\end{split}
\end{equation}
In this formula the regular factors occur for rank-one symmetric spaces.

\begin{Remark} 
{\em
If $\lambda = k_1\mu_1 + k_2\mu_2 + k_3\mu_3$, then formula (\ref{d4g2dim})
can be written as
\begin{equation}
\label{d4g2dim2}
 d(\lambda) = c_1\, \prod_{i=1}^{3}\binom{k_i + 2}{2}
        \prod_{1\leq i < j \leq 3} (k_i + k_j + 3)
        \prod_{j=1}^{2} (k_1 + k_2 + k_3 + j + 3) \, ,
\end{equation}
where $c_1 = 1/(3^3\cdot 4 \cdot 5) $. This version clearly exhibits the
symmetry in $k_1$, $k_2$, $k_3$ that comes from
the triality outer automorphisms of $G$ associated with the symmetries
of the Dynkin diagram.   Taking $\lambda = \mu_i$ ($i = 1, 2,
3$) in (\ref{d4g2dim2}) gives $d(\lambda) = 8$ (the vector and the two
half-spin representations), whereas taking $\lambda = \mu_2 + \mu_3$ gives
$d(\lambda) = 56 = \binom{8}{3}$ (the representation of $\SO_{8}(\bC)$ on
$\bigwedge^3 \bC^8$).
}
 \end{Remark}

%% \textbf{\em Case 7.} 
%    {\bf The pair $(\Sp_{2\ell},\, \bC^{\times} \times \Sp_{2\ell-2})$. }
%           \input{cglcexam}
           %%%%%%%% diagram:   newcglc.eps
%%%%%%%%%%%%%%%%%%%%%%%%%%%%%%%%%%%%%%%%%%%%%%%%%%%%%%%%%%%%%%%%%%%%%%%%%%%%%
%  Created       : Thu Jul 21 16:28:48 2011
%  Last Modified : 9/11/12 for short version
%
%%%%%%%%%%%%%%%%%%%%%%%%%%%%%%%%%%%%%%%%%%%%%%%%%%%%%%%%%%%%%%%%%%%%%%%%%%%%%

\vspace{1ex}
\noindent
\textbf{\em Case 7.} 
{\bf The pair $(\Sp_{2\ell},\, \bC^{\times} \times \Sp_{2\ell-2})$. }
We take $G$ in the matrix form of \cite[\S
2.1.2]{Goodman-Wallach}, with Cartan subalgebra $\frt$ the diagonal matrices
$\bx = \diag [\, \by, -\check{\by}]$, where 
$\by = [\, \varepsilon_1(\by), \ldots, \varepsilon_{\ell}(\by)\,]$
and 
$\check{\by} =  [\, \varepsilon_{\ell}(\by), \ldots, \varepsilon_{1}(\by)\,]$.
The roots of $\frt$ on $\frg$ are $\pm \varepsilon_i \pm \varepsilon_j$ for
$1\leq i < j \leq \ell$ and $\pm 2\varepsilon_i$ for $1 \leq i \leq \ell$.
Take the simple roots as $\alpha_i = \varepsilon_{i} - \varepsilon_{i+1}$ for
$i=1, \ldots, \ell-1$ and $\alpha_{\ell} = 2\varepsilon_{\ell}$.

From \cite{Kramer} we know that $(G, H)$ is a spherical pair and
$\Gamma(G/H)$ is free of rank $2$ with generators
\begin{equation}
\label{cglcfundwt}
  \mu_1 = 2\varpi_1 = 2\varepsilon_1 \,,
\quad
  \mu_2 = \varpi_2 = \varepsilon_1 + \varepsilon_2 \,.
\end{equation}
Hence the support condition in Definition \ref{excellent.def} is satisfied. If
$\ell = 2$, then $\Delta_{0}$ is empty, $\frc = \frt$, and $\lambda =
k_1\varpi_1 + k_2\varpi_2$ is a spherical highest weight if and only $k_1$ is
even.\footnote{
This is the same as Case 5 since $\Spin_{5} \cong \Sp_4$ and
$\bC^{\times}\times \Sp_2 \cong \GL_2$.}
Thus we may assume that $\ell \geq 3$ in the following.
Then  $\Delta_{0} = \{\alpha_{3}, \ldots, \alpha_{\ell} \}$ and
\begin{equation}
\label{cglcdiag}
 \frc = \{ \bx = \diag[\, \by, -\check{\by} \,]
 \ \mbox{with \ $\by = [\xi_1, \xi_2, 0, \ldots, 0 ]$ }  \} \,.
\end{equation}
Since $|\Supp \mu_i | = 1$ for $i = 1,2$, we know from Lemma \ref{rankdim.lem}
that  $\fra = \frc$ and $\frm = \fsp_{2\ell-4}(\bC)$.
The restricted root data are as follows (details given below--the entries in
the fourth and sixth columns are calculated using (\ref{cglcdelta}) and
(\ref{cglcshift1})).

\vspace{1ex}

\begin{center}

\begin{tabular}{|c |c |c| c| c | c |}
\hline
 r/s & \rule[-1ex]{0ex}{3ex} restricted root $\xi$ & mult.
  & $\langle \delta \mid \xi \rangle $ 
  & $\#$ basic roots  $\beta$ & $\langle h_{\frm}^0 \mid \beta \rangle$
 \\
\hline
 (r) & $ \xi_1 -  \xi_2$  &  $1$ &  $1$ & $1$ & $0$
 \\
\hline
 (r) & $ \xi_1 + \xi_2$  &  $1$ & $2\ell - 1$ & $1$ & $0$
 \\
\hline
 (s) & $ \xi_i$ \quad ($i = 1, 2$)  &  $2\ell - 4$ & $\ell + 1 - i$ 
   & $1$ & $-2\ell + 5$
 \\
\hline
 (s) & $2\xi_i$ \quad ($i = 1, 2$) &   $1$ &  $2(\ell + 1 - i)$ 
   & $1$ & $0$ 
\\
\hline
\end{tabular}

\end{center}

%%%%%%%%%%%%%%%%%%% 
\begin{figure}[h]
\includegraphics{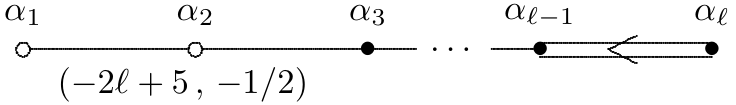} 
\caption{Marked Satake diagram for $\Sp_{2\ell}/(\bC^{\times}\times \Sp_{2\ell-2})$}
\label{diagcglc.fig}
\end{figure}
%%%%%%%%%%%%%%%%%%%%%%%%%%%%%%%%%%%%%%%%%

\noindent 
From the table we obtain
\begin{equation}
\label{cglcdelta}
  \delta = \ell\xi_1 + (\ell - 1)\xi_2 \,.
\end{equation}  

In this case $\Delta_0$ is of type $C_{\ell-2}$ and has two root lengths when
$\ell \geq 4$. We have $\alpha_{\ell}^\vee = \varepsilon_{\ell}$. With the
choice of positive roots for $\frg$ above, we have 
$\rho_{\frm} = 
 (\ell-2)\varepsilon_3 +  (\ell-3)\varepsilon_4 + \cdots 
 + \varepsilon_{\ell}$ 
\ and
$ h_{\frm}^0 =
   (2\ell-5)\varepsilon_3 +   (2\ell-7)\varepsilon_4 + \cdots 
   + \varepsilon_{\ell}$. 
Thus
\[
 \varpi_{\frm}^0 = \rho_{\frm} - \frac{1}{2}h_{\frm}^0
 =  \frac{1}{2}\big(\varepsilon_3 + \ldots + \varepsilon_{\ell}\big)
\]
From these formulas we see that
\begin{equation}
\label{cglcshift1}
 \langle\, h_{\frm}^0 \mid \alpha_i \,\rangle =
 \begin{cases} 0 & \text{if $i = 1$\,,} \\
               -2\ell + 5 & \text{if $i = 2$\,,}
 \end{cases}               
\end{equation}
and
\begin{equation}
\label{cglcshift2}
 \langle \, \varpi_{\frm}^0 \mid \alpha_i \,\rangle =
 \begin{cases} 0 & \text{for $i = 1$ and $3 \leq i \leq  \ell - 1$\,,} \\
               -1/2 & \text{if $i = 2$\,,}\\
               1 & \text{if $i = \ell$\,,}
 \end{cases}               
\end{equation}
as in (\ref{cpishift}). This furnishes the markings in the Satake diagram.

The nests of positive roots and the dimension factors $d_{\xi}(\lambda)$
for $\xi \in \Sigma^{+}$ and $\lambda \in \Gamma(G/H)$ are as follows.

\vspace{1ex}

\begin{enumerate}

%%%%%%%%%%%%%%%%%%% 
\item[{\bf (i)}] Let $\xi = \xi_1 - \xi_2$. Then
$\Phi^{+}(\xi) = \{ \alpha_1 \}$.
%

%%%%%%%%%%%%%%%%%%%%%%%%%%%%%%%%

\vspace{1ex}

\item[{\bf (ii)}]  Let $\xi = \xi_1 + \xi_2$. Then
$ \Phi^{+}(\xi) = \{ \beta \}$, where 
$\beta = \alpha_1 + 2\alpha_{2} + \cdots + 2\alpha_{n-1} + \alpha_{\ell}$.
%

%%%%%%%%%%%%%%%%%%%%%%%%%%%%%%%%

\vspace{1ex}

\item[{\bf (iii)}] Let $\xi = \xi_i$ with  $i = 1,2$. Then
\[
 \Phi^{+}(\xi) =  
  \{ \varepsilon_i \pm \varepsilon_j  \,:\, 3 \leq j \leq  \ell \}
   = \{ \beta_j \,,\,  \gamma_j  \,:\,  3 \leq j \leq  \ell \}\,.
\]
Here
$ \beta_j = \beta + \alpha_{3} +  \cdots + \alpha_{j-1}$
and
$\gamma_j = \beta_{j}
    + 2\alpha_{j} +  \cdots + 2\alpha_{n-1} + \alpha_{\ell}$, 
with the basic root 
$ \beta = \alpha_1 + \alpha_2$ when $i= 1$ and $\beta = \alpha_2$ when $i=2$.
(In these formulas $\beta_3 = \beta$ and $\gamma_{\ell} = \beta_{\ell} +
\alpha_{\ell}$.)
From (\ref{cglcshift1})  we see that
$ \langle\, h_{\frm}^0\,,\, \beta  \,\rangle = -2\ell+5$.
Since $\dim \frn_{\xi} = 2\ell - 4$, it follows from Proposition
\ref{rhopishift.prop} that the eigenvalues of $\ad h_{\frm}^0$ on $\frn_{\xi}$
are precisely
\[
   -2\ell + 5 \,,\, \ldots \,,\, -1 \,,\, 1 \,,\, \ldots \,,\, 2\ell-5 
\]
with multiplicity one. The negative eigenvalues come from the roots $\beta_j =
\varepsilon_i - \varepsilon_j$, while the positive eigenvalues come from the
roots $\gamma_j = \varepsilon_i + \varepsilon_j$. Since 
$\rho_{\frm} = (1/2)h_{\frm}^0 + \varpi_{\frm}^0$ \ and \ 
$ \langle\, \varpi_{\frm}^0\,,\, \beta_j \,\rangle = -1/2$\,, \ 
$ \langle\, \varpi_{\frm}^0\,,\, \gamma_j \,\rangle = 1/2$ \ 
for $3 \leq j \leq \ell$, it follows that the shifts  
$\langle\,\rho_{\frm} \mid \alpha\,\rangle$ 
in  $d_{\xi}(\lambda)$ are
\[
 -\ell + 2 \,,\,  \ldots \,,\,-2\,,\, -1 \,,\, 
   1 \,,\, 2\,\, \ldots \,,\,  \ell - 2
\]
(an arithmetic progression of step one with a gap at zero\footnote{A direct
calculation of these shifts from the formula for $\rho_{\frm}$ does not
explain the gap at zero.}). Hence
\[
 d_{\xi}(\lambda) = \frac{ 
 \Phi(\langle\, \lambda  \mid \xi \,\rangle \,,\,
  \langle\, \delta \mid \xi \,\rangle \,;\, \ell - 2)
 }{
 \Phi(\langle\, \lambda  \mid \xi \,\rangle \,,\,
  \langle\,  \delta \mid \xi \,\rangle )
 }\,,
\]
where the  factor in the denominator creates the
gap at zero.

%%%%%%%%%%%%%%%%%%%%%%%%%%%%%%%%%%%%%%%%%%%%%%%%

\vspace{1ex}
\item[{\bf (iv)}] Let $\xi = 2\xi_i$ with  $i = 1, 2$. Then 
$ \Phi^{+}(\xi) = \{ \beta \}$, where 
$\beta =2\alpha_{i} + \cdots + 2\alpha_{\ell-1} + \alpha_{\ell}$.
%

%%%%%%%%%%%%%%%%%%%%%%%%%%%%%%%%%%%%%%%%%%%%%%%%

\end{enumerate}

\vspace{1ex}

\noindent
In cases (i), (ii), and (iv)  Proposition \ref{rhopishift.prop} gives
$ d_{\xi}(\lambda) =  \Phi(\langle \lambda  \mid \xi \rangle \,,\,
 \langle \delta \mid \xi \rangle)$.
Combining this with the formula from case (iii) and using (\ref{restdim}), we
can obtain the full dimension formula. 
Let
\ $ \Sigma^{+}_{\rm reg} = \{ \xi_1 + \xi_2 \,,\, \xi_1 - \xi_{2}  \}$ 
\ and \
$\Sigma^{+}_{\rm sing} = \{ \xi_1 \,,\,  \xi_2  \}$. \ 
Then
\begin{equation}
\label{cglcdim}
\begin{split}
 d(\lambda) &=  \prod_{\xi \in \Sigma^{+}_{\rm reg} }
 \Phi(\langle\, \lambda  \mid \xi \,\rangle ,
  \langle\,  \delta \mid \xi \,\rangle )
\prod_{\xi \in \Sigma^{+}_{\rm sing}}
 \Phi(\langle\, \lambda \mid \xi \,\rangle ,
  \langle\, \delta \mid \xi \,\rangle \,;\, \ell - 2)
\\
 &= 
 \prod_{\xi \in \Sigma^{+}_{\rm reg}}
 W\big(\langle\, \lambda \mid \xi \, \rangle \,,\, 
  \langle\, \delta  \mid \xi\,\rangle  \,;\, m_\xi \,\big)
\\
 &\qquad \times
 \prod_{\xi \in \Sigma^{+}_{\rm sing}}
 W_{\rm sing}\big(\langle\, \lambda \mid \xi \,\rangle \,,\, 
 \langle \delta \mid \xi \,\rangle 
  \,;\, m_{\xi}\,, m_{2\xi} \,\big) \,. 
\end{split}
\end{equation}
In this formula the regular dimension factors occur for rank-one symmetric
spaces. Note that the denominator for $\xi_i$ in case (iii) is cancelled by
the dimension factor for $2\xi_i$ from case (iv).

\begin{Remark}
{\em
 If $\lambda = k_1\mu_1 + k_2\mu_2$, then formula (\ref{cglcdim}) can be
written as
\begin{equation}
\label{cglcdim2}
 d(\lambda) = c\,
 (2k_1 + 1)(2k_1 + 2k_2+2\ell-1)
      \binom{k_2 + 2\ell - 3}{2\ell-3}\binom{2k_1 + k_2 + 2\ell - 2}{2\ell-3} \, ,
\end{equation}
where $c = 1/[(2\ell-1)(2\ell-2)]$ is the normalizing constant to make $d(0) =
1$. Taking $\lambda = \mu_1$ gives $d(\mu_1) = 2\ell(2\ell+1)/2$. Here $\mu_1$
is the highest weight of the irreducible $G$-module $S^2(\bC^{2\ell})$ (see
\cite[\S10.2.3, Example 2]{Goodman-Wallach}). Taking $\lambda = \mu_2$ gives
$d(\mu_2) = (2\ell+1)(\ell - 1) = \binom{2\ell}{2} - \binom{2\ell}{0}$. 
In this case $\mu_2$ is the highest weight of the traceless 
({\em harmonic}) subspace in 
$\rule{0ex}{2.5ex}\bigwedge^2 \bC^{2\ell}$ (see \cite[Cor.
5.5.17]{Goodman-Wallach}).
}
\end{Remark}

%% \textbf{\em Case 8.} {\bf The pair $({\bf E}_6,\, \Spin_{10})$. }
%           \input{e6d5exam}
           %%%%%%%% diagram:   neweiii.eps

%%%%%%%%%%%%%%%%%%%%%%%%%%%%%%%%%%%%%%%%%%%%%%%%%%%%%%%%%%%%%%%%%%%%%%%%%%%%%
%  Created       : Thu Jul 21 16:28:48 2011
%  Last Modified : 9/11/12 for short version
%
%%%%%%%%%%%%%%%%%%%%%%%%%%%%%%%%%%%%%%%%%%%%%%%%%%%%%%%%%%%%%%%%%%%%%%%%%%%%%

\vspace{1ex}
\noindent
\textbf{\em Case 8.} {\bf The pair $({\bf E}_6,\, \Spin_{10})$. }
We take the root system and the simple roots $\alpha_1\,,\, \ldots \,,\,
\alpha_6$ for $\frg$ as in \cite[Planche V]{Bourbaki3}. We identify the Cartan
subalgebra $\frt$ with the vectors $\bx \in \bC^8$ such that
$ \langle \varepsilon_6 \mid \bx \rangle 
  = \langle \varepsilon_7 \mid \bx \rangle
 = -\langle \varepsilon_8 \mid \bx \rangle$.

From \cite{Kramer} we know that $(G, H)$ is
a spherical pair and $\Gamma(G/H)$ has generators
\begin{equation}
\label{e6d5fundwt}
\begin{split}
  \mu_1 &= \varpi_1 
   = \frac{2}{3}\big(\varepsilon_8 - \varepsilon_7 - \varepsilon_6\big)\,,
\\
 \mu_2 &= \varpi_2 =  \frac{1}{2}
 \big(\varepsilon_1 + \varepsilon_2 + \varepsilon_3 + \varepsilon_4 
 + \varepsilon_5 - \varepsilon_6 - \varepsilon_7 + \varepsilon_8 \big)\,,
\\
  \mu_3 &= \varpi_6 = 
 \frac{1}{3}\big(\varepsilon_8 - \varepsilon_7 - \varepsilon_6\big)
  + \varepsilon_5 
\end{split}
\end{equation} 
corresponding to the endpoints of the Dynkin diagram. Hence the support
condition in Definition \ref{excellent.def} is satisfied and $\Delta_{0} =
\{\alpha_3, \alpha_4, \alpha_5 \}$, where
$\alpha_3 = \varepsilon_2 - \varepsilon_1$\,,\,  
$\alpha_4 = \varepsilon_3 - \varepsilon_2$\,,\, 
$\alpha_5 = \varepsilon_4 - \varepsilon_3$\,.
Since $|\Supp \mu_i | = 1$ for $i = 1,2$, Lemma \ref{rankdim.lem}
gives  $\fra = \frc$.

Let $\bx \in \frt$. Then $\langle \alpha_i \mid \bx \rangle = 0$ for $i = 3,
4, 5$ if and only if
\begin{equation}
\label{e6d5splittor}
  \bx = 
 \gamma_1(\varepsilon_1 + \varepsilon_2 + \varepsilon_3 + \varepsilon_4)
 +  \gamma_2 \varepsilon_5 + 
  \gamma_3(-\varepsilon_6 - \varepsilon_7 + \varepsilon_8) \,,
\end{equation}
with  $\gamma_i \in \bC$.
Thus $\fra $ consists of all $\bx$ in (\ref{e6d5splittor}).

Let $\xi_1 = \alpha_1|_{\fra}$\,, \ $\xi_2 = \alpha_2|_{\fra}$\,, \ $\xi_3 =
\alpha_6|_{\fra}$\,. By \cite[Planche V]{Bourbaki3} these roots are given by
$ \alpha_1 = \frac{1}{2}
\big(\varepsilon_1 - \varepsilon_2 - \varepsilon_3 - \varepsilon_4
 - \varepsilon_5 - \varepsilon_6 - \varepsilon_7 +  \varepsilon_8\big)$, \ 
$ \alpha_2 = \varepsilon_1 + \varepsilon_2$\,, \ and \ 
$ \alpha_6 = \varepsilon_5 - \varepsilon_4$\,.   
Hence from (\ref{e6d5splittor}) we calculate that 
\begin{equation}
\label{e6d5restroot}
\begin{split}
 \xi_1 &= 
 -\frac{1}{4}\big(\varepsilon_1 + \varepsilon_2 + \varepsilon_3 
  + \varepsilon_4\big)
 - \frac{1}{2}\big(\varepsilon_5 + \varepsilon_6 
   + \varepsilon_7 - \varepsilon_8\big) \,,
\\
  \xi_2 &= \varepsilon_5 \,,
\qquad
\xi_3 = \frac{1}{3}\big(\varepsilon_6 + \varepsilon_7 - \varepsilon_8\big) \,.
\end{split}
\end{equation}
The restricted root data are as follows (details given below--the entries in
the fourth and sixth columns are calculated using (\ref{e6d5delta}) and
(\ref{e6d5shift})).

\vspace{1ex}

\begin{center}

\begin{tabular}{|c |c |c| c | c | c |}
\hline
 r/s & \rule[-1ex]{0ex}{3ex} restricted root $\xi$  & mult.
  & $\langle \delta \mid \xi \rangle $ 
  & $\#$ basic roots  $\beta$ & $\langle h_{\frm}^0 \mid \beta \rangle$
 \\
\hline
 (s) & $ \xi_1$\,,\, $ \xi_3$ &  $4$ &  $5/2$ & 1 & $-3$ 
 \\
\hline
 (s) & $ \xi_1 + \xi_2 $\,,\, $ \xi_2 + \xi_3 $ &  $4$ &  $11/2$ 
   & $1$ & $-3$
 \\
\hline
 (r) & $ \xi_2$  &  $6$ & $3$ & $1$ & $-4$
 \\
\hline
 (r) & $ \xi_1 + \xi_2 + \xi_3 $  &  $6$ & $8$ & $1$ & $-4$
 \\
\hline
 (r) & $ \xi_1 + \xi_3 $  &  $1$ & $5$ & $1$ & $0$
 \\
\hline
 (r) & $ \xi_1 + 2\xi_2 + \xi_3 $ &  $1$ & $11$  & $1$ & $0$
 \\
\hline
\end{tabular}

\end{center}

%%%%%%%%%%%%%%%%%%% 
\begin{figure}[h]
\includegraphics{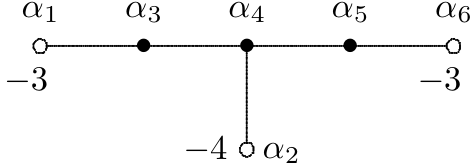} 
\caption{Marked Satake diagram for ${\bf E}_6/\Spin_{10}$}
\label{diageiii.fig}
\end{figure}
%%%%%%%%%%%%%%%%%%%%%%%%%%%%%%%%%%%%%%%%%

\noindent
From the table we obtain
\begin{equation}
\label{e6d5delta}
   \delta = 8\xi_1 + 11\xi_2 + 8\xi_3 \,.
\end{equation}
Since $\frm \cong \fsl(4,\bC)$ with simple roots $\{\alpha_3 \,,\, \alpha_4
\,,\, \alpha_5 \}$, we have
$ h_{\frm}^0 = 3\alpha_3 + 4\alpha_4 + 3\alpha_5$.
Hence 
\begin{equation}
\label{e6d5shift}
 \langle \, h_{\frm}^0 \mid \alpha_i \, \rangle =
 \begin{cases}
         -3 & \text{if $i = 1, \, 6$\,,} \\
         -4 & \text{if $i = 2$\,,}  \\
          2 & \text{if $i = 3, \, 4, \, 5$\,,}
 \end{cases}               
\end{equation}
which gives the markings in the Satake diagram.

The nests of positive roots and the dimension factors $d_{\xi}(\lambda)$ for
$\xi \in \Sigma^{+}$ and $\lambda \in \Gamma(G/H)$ are as follows. We
describe the positive roots in terms of their coefficients relative to the
simple roots, as in \cite[Planche V]{Bourbaki3}.

\vspace{1ex}

\begin{enumerate}

%%%%%%%%%%%%%%%%%%% 
\item[{\bf (i)}] For the restricted root $ \xi_1$ the nest is
\[
\qquad
 \Phi^{+}(\xi_1) = \Big\{  {10000 \atop 0}   \quad  
  {11000 \atop 0} \quad  
  {11100 \atop 0} \quad  
  {11110 \atop 0} \Big\}
\]
and the basic root is $\beta = \alpha_1$.
For the restricted root $\xi_3$ the nest is
\[
\qquad
 \Phi^{+}(\xi_3) = \Big\{ 
  {00001 \atop 0}   \quad  
  {00011 \atop 0} \quad  
  {00111 \atop 0} \quad  
  {01111 \atop 0} 
  \Big\}
\]
and the basic root is $\beta = \alpha_6$.
For the restricted root $\xi_1 + \xi_2$ the nest is
\[
\qquad
 \Phi^{+}(\xi_1 + \xi_2) = \Big\{ 
  {11100 \atop 1}   \quad  
  {11110 \atop 1} \quad  
  {11210 \atop 1} \quad  
  {12210 \atop 1}
 \Big\}
\]
and the basic root is $\beta = \alpha_1 + \alpha_2 + \alpha_3 + \alpha_4$.
For the restricted root $\xi_2 + \xi_3$ the nest is 
\[
\qquad
 \Phi^{+}(\xi_2 + \xi_3) = \Big\{ 
  {00111 \atop 1}   \quad  
  {01111 \atop 1} \quad  
  {01211 \atop 1} \quad  
  {01221 \atop 1}
 \Big\}
\]
and the basic root is $\beta = \alpha_2 + \alpha_4 + \alpha_5 + \alpha_6$. 
When $\xi$ is any of these restricted roots and $\beta$ is the corresponding
basic root, then from (\ref{e6d5shift}) we calculate that $\langle \,
h_{\frm}^0 \mid \beta \,\rangle = -3$. Since $m_{\xi} = k_{\xi} + 1$
in all these cases, Proposition \ref{rhoshift.prop} (3) gives 
\[
 d_{\xi}(\lambda) = \Phi\big(\langle\, \lambda  \mid \xi \,\rangle 
   \,,\,  \langle\, \delta \mid \xi \,\rangle 
   \,;\,  \textstyle{\frac{3}{2}} \big) 
 \,.
\]
%

%%%%%%%%%%%%%%%%%%%%%%%%%%%%%%%%%%%%%%%%%%%%%%%%%%%%%

\vspace{1ex}

\item[{\bf (ii)}] For the restricted root $\xi_2$ the nest is
\[
\qquad \quad
 \Phi^{+}(\xi_2) = \Big\{ 
  {00000 \atop 1}   \quad  
  {00100 \atop 1} \quad  
  {01100 \atop 1} \quad  
  {00110 \atop 1} \quad  
  {01110 \atop 1} \quad  
  {01210 \atop 1}
 \Big\}
\]
and the basic root is $\beta = \alpha_2$.
For the restricted root $\xi_1 + \xi_2 + \xi_3$ the nest is
\[
\qquad \quad 
 \Phi^{+}(\xi_1 + \xi_2 + \xi_3) = \Big\{ 
  {11111 \atop 1}   \quad  
  {11211 \atop 1} \quad  
  {12211 \atop 1} \quad  
  {11221 \atop 1} \quad
  {12221 \atop 1} \quad
  {12321 \atop 1} 
  \Big\}
\]
and the basic root is 
$\beta = \alpha_1 + \alpha_2 + \alpha_3 + \alpha_4 + \alpha_5 + \alpha_6$.
When $\xi$ is either of these restricted roots and $\beta$ is the
corresponding basic root, then from (\ref{e6d5shift}) we calculate that
$\langle \, h_{\frm}^0 \mid \beta \,\rangle = -4$. Since $m_{\xi} =
k_{\xi} + 2$ in both cases, Proposition \ref{rhoshift.prop} (4) gives 
\[
d_{\xi}(\lambda) = \Phi(\langle\, \lambda  \mid \xi \,\rangle \,,\,
   \langle\,  \delta \mid \xi \,\rangle) 
   \,
   \Phi(\langle\, \lambda  \mid \xi \,\rangle  \,,\, 
      \langle\,  \delta \mid \xi \,\rangle  \,;\, 2)
    \,.
\]
%

%%%%%%%%%%%%%%%%%%%%%%%%%%%%%%%%

\vspace{1ex}

\item[{\bf (iii)}]
For $\xi = \xi_1 + \xi_3$  the nest is
$\displaystyle{ \Phi^{+}(\xi) = \Big\{  {11111 \atop 0} \Big\} }$,
while for $\xi = \xi_1 + 2\xi_2 + \xi_3$ the nest is
$\rule{0ex}{3.5ex}
\displaystyle{\Phi^{+}(\xi) = \Big\{  {12321 \atop 2} \Big\} }$.
Hence 
$ d_{\xi}(\lambda) = \Phi(\langle\, \lambda  \mid \xi \,\rangle \,,\,
   \langle\,  \delta \mid \xi \,\rangle)$
by Proposition \ref{rhoshift.prop} (2).

%%%%%%%%%%%%%%%%%%%%%%%%%%%%%%%%

\end{enumerate}

\vspace{1ex}

\noindent Let
\[
 \Xi^{+}_{0} = \{ \xi_2 \,,\, \xi_1 + \xi_2 + \xi_3 \}, \ 
 \Xi^{+}_{1} = 
\{ \xi_1\,,\, \xi_3 \,,\, \xi_1 + \xi_2 \,,\, \xi_2 + \xi_3 \}, \ 
  \Xi^{+}_{2} = \{ \xi_1 + \xi_3 \,,\, \xi_1 + 2\xi_2 + \xi_3 \}.
\]
Then from cases (i)--(iii)  we see that
\  $\Sigma_{\rm reg}^{+} = \Xi^{+}_{0} \cup  \Xi^{+}_{2}$ \ and \ 
 $\Sigma_{\rm sing}^{+} = \Xi^{+}_{1}$. 
Furthermore, the dimension formula  is
\begin{equation}
\label{e6d5dim}
\begin{split}
 d(\lambda) &=
\prod_{\xi \in \Xi^{+}_{0}}
  \Phi(\langle\, \lambda \mid \xi \,\rangle \,,\,
   \langle\,  \delta \mid \xi \,\rangle ) \,
  \Phi(\langle\, \lambda  \mid \xi \,\rangle \,,\,
   \langle\,  \delta \mid \xi \,\rangle  \,;\, 2)
\\
 & \qquad \times  
 \prod_{\xi \in \Xi^{+}_{1}}
   \Phi\big(\langle\, \lambda \mid \xi \,\rangle \,,\,
   \langle\, \delta \mid \xi \,\rangle \,;\, \textstyle{ \frac{3}{2}} \big)
 \displaystyle{  \prod_{\xi \in \Xi^{+}_{2}} }
  \Phi(\langle\, \lambda \mid \xi \,\rangle \,,\,
       \langle\, \delta \mid \xi \,\rangle )
\\
  &=   
 \prod_{\xi \in \Sigma^{+}_{\rm reg}}
    W\big(\langle \lambda \mid \xi \rangle \,,\, 
  \langle \delta \mid \xi \rangle  \,;\,  m_{\xi} \, \big)\,
 \prod_{\xi \in \Sigma^{+}_{\rm sing}} 
    W_{\rm sing}\big(\langle \lambda \mid \xi \rangle \,,\, 
  \langle \delta \mid \xi \rangle \,;\,  m_{\xi} \,  \big)\,.
\end{split}
\end{equation}

\vspace{1ex}

\begin{Remark}
{\em
If $\lambda = k_1\mu_1 + k_2\mu_2 + k_3\mu_3$, then using
(\ref{e6d5fundwt}), (\ref{e6d5restroot}), and the values of $\langle\, \delta
\mid \xi \,\rangle$ given in the table, we can write formula (\ref{e6d5dim})
as
\begin{equation}
\label{e6d5dim2}
\begin{split}
 d(\lambda) &=   c  \,
  (k_2+3)(k_1 + k_3 + 5)(k_1 + k_2 + k_3 + 8)(k_1 + 2k_2 + k_3 + 11)
\\
  &\qquad \times 
   \prod_{j=1}^{4} (k_1 + j)(k_3 + j)(k_1 + k_2 + j + 3)(k_2 + k_3 + j + 3) 
\\
  &\qquad \times 
     \prod_{j=1}^{5} (k_2 + j)(k_1 + k_2 + k_3 + j + 5) \, , 
\end{split}
\end{equation}
where $c$ is the normalizing constant to make $d(0) = 1$. Thus the dimension
formula is symmetric in $k_1$ and $k_3$ (this is evident {\em a priori} from
the outer automorphism of $G$ associated with the Dynkin diagram symmetry).
Taking $\lambda = \mu_1$ or $\mu_3$, we calculate that $d(\lambda) = 27$ (the
two mutually contragredient representations of $G$ on the exceptional simple
Jordan algebra). Taking $\lambda = \mu_2$ (the highest root of $\frg$), we
calculate that $d(\lambda) = 78$ (the adjoint representation of $G$).
}
\end{Remark}

\section{Higher Rank Symmetric Spaces}
           \label{symspaceexam.sec}
%           \input{symspaceexam}
%%%%%%%%%%%%%%%%%%%%%%%%%%%%%%%%%%%%%%%%%%%%%%%%%%%%%%%%%%%%%%%%%%%%%%%%%%%%%
%  Created       : Fri Mar 16 16:39:54 2012
%  Last Modified : 9/11/12 for short version
% 
%%%%%%%%%%%%%%%%%%%%%%%%%%%%%%%%%%%%%%%%%%%%%%%%%%%%%%%%%%%%%%%%%%%%%%%%%%%%%

We now turn to the proof of Theorem \ref{symspacedim.thm} for an irreducible
symmetric pair $(G, K)$ of rank $r \geq 2$ that is the complexification of a
compact symmetric space of type {\bf I} (in the terminology of \cite[Ch. VIII,
\S 5]{Helgason1}). We use the same case-by-case method as for the higher rank
nonsymmetric spherical pairs. However, there is a significant simplification:
due to the Weyl group symmetry of the restricted root system it suffices to
consider the root nests and dimension factors for a set of simple roots.

\begin{Lemma}
\label{regsimple.lem}
Assume that the Dynkin diagram of $\,\frm$ is simply laced.
If the dimension factors satisfy
\begin{equation}
\label{regdimfact}
 d_{\xi}(\lambda) d_{2\xi}(\lambda) =  
   W\big(\langle \lambda \mid \xi\rangle , \langle \delta \mid \xi \rangle
   \;\,;\, m_{\xi}, m_{2\xi}\big)
\quad \mbox{when $\lambda \in \Gamma(G/K)$}
\end{equation}
for each  simple indivisible restricted root $\xi$, then 
\,{\rm (\ref{regdimfact})}\, holds for all $\xi \in \Sigma_{0}^{+}$.
Here $d_{2\xi}(\lambda) = 1$ if $2\xi$ is not a restricted root.
\end{Lemma}

For proof see \cite{Gindikin-Goodman}.

\vspace{1ex}

Recall that for a restricted root $\xi$ we let $k_{\xi}$ be the largest
eigenvalue of $\ad(h_{\frm}^{0})$ on $\frn_{\xi}$.

\begin{Proposition}
\label{regsimple.prop}
Assume that the Dynkin diagram of $\frm$ is simply laced and that every simple
indivisible restricted root $\xi$ satisfies one of the following.

\begin{enumerate}

\item
  $m_{\xi} = 1$\,.

\item
$m_{\xi} = 2$\,, $m_{2\xi} = 0$\,, $k_{\xi} = 0$\,, and there are two
basic roots in $\Phi^{+}(\xi)$\,.

\item 
$m_{\xi} = 3$\,, $m_{2\xi} = 0$\,, and $k_{\xi} = 2$\,.

\item 
 $m_{\xi} = k_{\xi} + 2  \geq 4$ and  $m_{2\xi} = 0$\,.

\item $ m_{\xi} = 2(k_{\xi} + 1)$\,, $m_{2\xi} = 1$\,,  and
there are two basic roots in $\Phi^{+}(\xi)$\,.

\end{enumerate}
Then all restricted root nests are regular and the dimension formula
\,{\rm (\ref{symspacedim})} is valid for all $\lambda \in \Gamma(G/K)$.
\end{Proposition}

\begin{proof}
This follows from Proposition \ref{rhoshift.prop}, formulas
(\ref{regdimfact1}) and (\ref{regdimfact2}), and Lemma \ref{regsimple.lem}
\end{proof}

The symmetric spaces of rank $r \geq 2$ that satisfy condition (1) of
Proposition \ref{regsimple.prop} for all simple restricted roots are those of
types {\bf A\,I}, {\bf D\,I} ($r = \ell$), {\bf E\,I}, {\bf E\,V}, {\bf
E\,VIII}, {\bf F\,I}, and {\bf G} (to check this it suffices to look at the
Satake diagrams).

We proceed to carry out a case-by-case analysis of the remaining irreducible
symmetric spaces, obtain their marked Satake diagrams and root nest data, and
verify that formula (\ref{regdimfact}) holds for all simple restricted roots
with multiplicity greater than one (for multiplicity one, this is automatic).
It turns out that all but two of the cases with rank $r \geq 2$ are covered by
Proposition \ref{regsimple.prop}. The remaining two cases (with $\frm$ not
simply laced) are type {\bf B\,I} and type {\bf C\,II} ($\ell \geq 2r + 1$).
For the simple restricted roots in these cases we use the method of Section
\ref{rankone.sec}, Cases 3 and 4, to prove (\ref{regdimfact}), followed by a
Weyl group argument using Lemma \ref{rootnest.lem} to extend this result to
all the positive restricted roots (details given in Cases 3 and 4 below).

\begin{Remark}
{\em
 Following \cite[Ch. X, Table VI]{Helgason1} the
simple restricted roots are labeled using the enumeration of the simple roots.
Thus $\lambda_i = \overline{\alpha_i}$ when the restriction of $\alpha_i$ to
$\fra$ is nonzero.
}
\end{Remark}

%% \textbf{\em Case 1.} {\bf Type A\,II.}
%            \input{aiiexam}
           %%%%%%%% diagram:    newssaii.eps 

%%%%%%%%%%%%%%%%%%%%%%%%%%%%%%%%%%%%%%%%%%%%%%%%%%%%%%%%%%%%%%%%%%%%%%%%%%%%%
%  Created       : 3/11/12
%  Last Modified : 9/11/12 for short version
%
%%%%%%%%%%%%%%%%%%%%%%%%%%%%%%%%%%%%%%%%%%%%%%%%%%%%%%%%%%%%%%%%%%%%%%%%%%%%%

\vspace{1ex}
\noindent
\textbf{\em Case 1.} {\bf Type A\,II.}
Let $G = \SL_{2r + 2}(\bC)$ and $K = \Sp_{2r+2}(\bC)$ with $r \geq 2$. Then
$G$ has rank $\ell = 2r + 1$ and the fundamental $K$-spherical highest weights
are $\mu_i = \varpi_{2i}$ for $i = 1, \ldots, r$. Hence $\frm \cong \fsl_2
\oplus \cdots \oplus \fsl_2$ \ ($r+1$ copies) and
$h_{\frm}^{0} = 
  \alpha_1 + \alpha_{3} + \cdots + \alpha_{2r-1} + \alpha_{2r+1}$.
Thus $\langle h_{\frm}^{0} \mid \alpha_{2i} \rangle = -2$ for $1\leq i \leq
r$. The simple restricted  root data are as follows.

\vspace{1ex}

\begin{center}

\begin{tabular}{|c |c| c |}
\hline
 restricted root   & multiplicity & \# basic roots 
 \\
\hline
 \rule[-1ex]{0ex}{3ex}  $ \lambda_{2i} $ \quad ($1\leq i \leq r$)  &  $4$  
  & $1$
 \\
\hline
\end{tabular}

\end{center}

\vspace{1ex}

%%%%%%%%%%%%%%%%%%% 
\begin{figure}[h]
\includegraphics{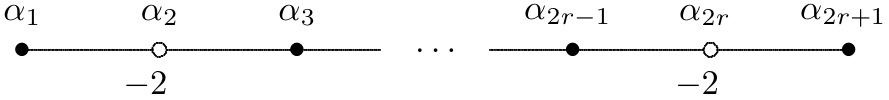} 
\caption{Marked Satake diagram for 
$\SL_{2r + 2}/\Sp_{2r + 2}$  }
\label{diagaii.fig}
\end{figure}
%%%%%%%%%%%%%%%%%%%%%%%%%%%%%%%%%%%%%%%%%

\noindent
The root nests are
\[
  \Phi^{+}\big(\lambda_{2i}\big)  =
  \{\, \alpha_{2i}\,,\, 
       \alpha_{2i-1} + \alpha_{2i}\,,\, 
       \alpha_{2i} + \alpha_{2i+1}\,,\, 
       \alpha_{2i-1} + \alpha_{2i} + \alpha_{2i+1}\, \}
\]
for $i = 1, \ldots, r$, and $\alpha_{2i}$ is the only basic root in the nest.
Since $k_{\lambda} + 2 = m_{\lambda}$ for all restricted simple positive
roots, condition (4) in Proposition \ref{regsimple.prop} is satisfied.

% \textbf{\em Case 2.} {\bf Type A\,III.}
%            \input{aiiiexam}
           %%%%%%%% diagram:   newssaiii.eps

%%%%%%%%%%%%%%%%%%%%%%%%%%%%%%%%%%%%%%%%%%%%%%%%%%%%%%%%%%%%%%%%%%%%%%%%%%%%%
%  Created       : 3/11/12
%  Last Modified : 9/11/12 for short version
%
%%%%%%%%%%%%%%%%%%%%%%%%%%%%%%%%%%%%%%%%%%%%%%%%%%%%%%%%%%%%%%%%%%%%%%%%%%%%%

\vspace{2ex}
\noindent
\textbf{\em Case 2.} {\bf Type A\,III.}
Let $G = \SL_{\ell + 1}(\bC)$ and $K = {\bf S}(\GL_{r}(\bC) \times \GL_{\ell +
1 - r}(\bC))$ with $\ell \geq 2r -1 \geq 3$. The fundamental $K$-spherical
highest weights are $\mu_i = \varpi_{i} + \varpi_{\ell-i+1}$ for $i = 1,
\ldots, r$. The simple restricted root data are as follows.

\vspace{1ex}

\begin{center}

\begin{tabular}{|c |c| c |}
\hline
 \rule[-3.5ex]{0ex}{4ex}\small{ rest. root}   & \small{mult.} 
  & \parbox{8ex}{\,\small{\# basic  roots}}
 \\
\hline
  \rule[-1ex]{0ex}{3ex} $ \lambda_{i}$ \ 
  \small{($\scriptstyle{1 \leq i \leq r-1}$)}  &  $2$ & $2$ 
 \\
\hline
 \rule[-1ex]{0ex}{3ex}  $ \lambda_r $  & $2(\ell - 2r + 1) $ &  $2$ 
\\
\hline
 \rule[-1ex]{0ex}{3ex}  $ 2\lambda_r $  & $1$ &  $1$ 
\\
\hline
\end{tabular}
\quad
\begin{tabular}{|c |c| c | }
\hline
 \rule[-3.5ex]{0ex}{4ex} \small{rest. root}   &  \small{mult.}
  & \parbox{8ex}{\, \small{\# basic  roots} }
 \\
\hline
  \rule[-1ex]{0ex}{3ex} $ \lambda_{i}$ \ 
  \small{($\scriptstyle{1 \leq i \leq r-1}$)}  &  $2$ & $2$
 \\
\hline
  \rule[-1ex]{0ex}{3ex}  $ \lambda_r $  & $1$ & $1$
\\
\hline
\end{tabular}

\end{center}

%%%%%%%%%%%%%%%%%%% 
\begin{figure}[h]
\includegraphics{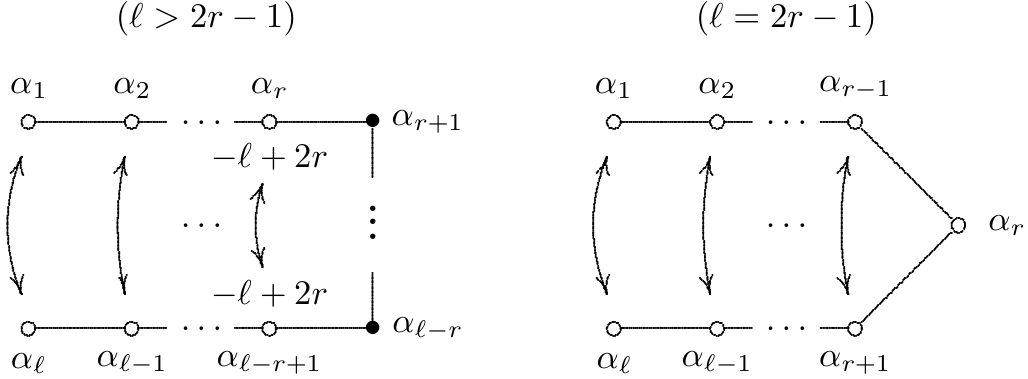} 
\caption{Marked Satake diagrams for 
$\SL_{\ell + 1}/{\bf S}(\GL_{r}\times \GL_{\ell + 1 - r})$}
\label{diagssaiii.fig}
\end{figure}
%%%%%%%%%%%%%%%%%%%%%%%%%%%%%%%%%%%%%%%%%

When $\ell = 2r-1$ then $\Delta_0 = \emptyset$ and $\frm' = 0$. Since $|\Supp
\mu_i | = 2$ for $i = 1, \ldots, r-1$ and $|\Supp \mu_r| = 1$, we know by
Lemma \ref{rankdim.lem} that $\dim \frc = 2(r-1) + 1 = 2r -1$ and $\dim \frc_0
= (2r - 1) - r = r-1$. Identify $\frt$ with $\frt^{*}$ using the form $\langle
\cdot \mid \cdot \rangle$. If $\bx = c_1\alpha_1 + \cdots +
c_{\ell}\alpha_{\ell}$ is in $\frt$, then the equations $\langle \mu_i \mid
\bx \rangle = 0$ for $i = 1, \ldots , r$ become
\[
 c_{\ell + 1 - i} = - c_i \quad\mbox{for $i = 1, \ldots,  r-1$ and } \ 
 c_{r} = 0\,.
\]    
Hence the linearly independent set
\begin{equation}
\label{aiiic0basis1}
 \{\, \alpha_{1} - \alpha_{\ell} \,, \,
  \alpha_{2} - \alpha_{\ell-1} \,, \,
  \cdots \, , \,
  \alpha_{r-1} - \alpha_{r+1} \,\} 
\end{equation}
is a basis for $\frc_{0}$. 
The root nests are
$ \Phi^{+}\big(\lambda_{i}\big) =
  \{\, \alpha_{i} \,,\,  \alpha_{\ell + 1 -i} \, \}$
  for $1\leq i \leq r-1$ and
$ \Phi^{+}\big(\lambda_{r}\big) =  \{\, \alpha_{r} \, \}$.
Thus condition (2) of Proposition \ref{regsimple.prop} is satisfied by
$\lambda_i$ for $1\leq i \leq r-1$ and condition (1) is satisfied by
$\lambda_{r}$.

%%%%%%%%%%%%%%%%%%%%%%%%%%%%%%%%%%%

Now assume $\ell > 2r-1$. Then $\Delta_0 = \{\,\alpha_{r+1}\,,\, \ldots \,,\,
\alpha_{\ell-r}\,\}$ and $|\Supp \mu_i | = 2$ for $i = 1, \ldots, r$. Hence
by Lemma \ref{rankdim.lem} $\dim \frc = 2r$ and $\dim \frc_0 = 2r - r = r$. 
The set (\ref{aiiic0basis1}) is in $\frc_0$, as is the vector
\[
\begin{split}
  \by =& (\ell - r + 1)\alpha_{r} + (\ell - r - 1)\alpha_{r+1} 
                       + (\ell - r - 3)\alpha_{r+2} + \cdots
\\                        
   &\quad + (r + 3 - \ell)\alpha_{\ell - r - 1}
        + (r + 1 - \ell)\alpha_{\ell - r}
        + (r - 1 - \ell)\alpha_{\ell - r+1}\,.
\end{split}
\]
Hence (\ref{aiiic0basis1}) together with $\by$ give a basis for $\frc_0$.
We have
$\frm' \cong \fsl_{\ell - 2r + 1}$ and 
$h_{\frm}^{0} =   
 (\ell - 2r) \alpha_{r+1} +  \cdots +  (\ell - 2r)\alpha_{\ell - r}$.
Thus $\langle h_{\frm}^{0} \mid \alpha_{i} \rangle = -\ell + 2r$ 
for $i = r$ and $i = \ell - r + 1$. Furthermore, the root nests are
\ $ \Phi^{+}\big(\lambda_{i}\big) =
  \{\, \alpha_{i} \,,\,  \alpha_{\ell + 1 -i} \, \}$
for $1\leq i \leq r-1$, and
\[
\begin{split}
 \Phi^{+}\big(\lambda_{r}\big) &=
  \{\, \alpha_{r} + \cdots + \alpha_{r+ i} 
      \,:\, 0 \leq i \leq \ell - 2r \,\}
 \\
   & \qquad \cup
  \{\, \alpha_{\ell - r + 1} + \cdots + \alpha_{\ell - r + 1 - i} 
   \,:\, 0 \leq i \leq \ell - 2r \,\}\,,
\\
 \Phi^{+}\big(2\lambda_{r}\big) &=
  \{\, \alpha_{r} + \cdots + \alpha_{\ell - r+ 1} \, \} \,.
\end{split}
\]
The nest  $ \Phi^{+}\big(\lambda_{i}\big)$ has two basic roots 
$\alpha_i$ and $\alpha_{\ell + 1 - i}$
for $1 \leq i \leq r$.
Thus condition (2) of Proposition \ref{regsimple.prop} is satisfied by
$\lambda_i$ for $1\leq i \leq r-1$ and condition (5) is satisfied by
$\lambda_{r}$.

% \textbf{\em Case 3.}  {\bf Type B\,I.}
%            \input{biexam}
           %%%%%%%% diagram:    newssbi.eps

%%%%%%%%%%%%%%%%%%%%%%%%%%%%%%%%%%%%%%%%%%%%%%%%%%%%%%%%%%%%%%%%%%%%%%%%%%%%%
%  Created       : 3/11/12
%  Last Modified : 9/11/12 for short version
%
%%%%%%%%%%%%%%%%%%%%%%%%%%%%%%%%%%%%%%%%%%%%%%%%%%%%%%%%%%%%%%%%%%%%%%%%%%%%%

\vspace{2ex}
\noindent
\textbf{\em Case 3.}  {\bf Type B\,I.}
Let $G = \SO_{2\ell + 1}(\bC)$ and $K = \SO_{r}(\bC) \times \SO_{2\ell + 1
-r}(\bC)$ with $2 \leq r < \ell$. The fundamental $K$-spherical highest
weights are $\mu_i = 2\varpi_{i}$ for $i = 1, \ldots, r-1$ and $\mu_r =
\varpi_r$. The simple restricted root data are as follows.

\vspace{1ex}

\begin{center}

\begin{tabular}{|c |c| c |}
\hline
 restricted root   & multiplicity & \# basic roots
 \\
\hline
 \rule[-1ex]{0ex}{3ex}  $\lambda_i$ \quad ($1\leq i \leq r-1$)  &  $1$  & $1$
 \\
\hline
 \rule[-1ex]{0ex}{3ex}   $\lambda_r $  & $2(\ell - r) + 1$  & $1$ 
\\
\hline
\end{tabular}

\end{center}

%%%%%%%%%%%%%%%%%%% 
\begin{figure}[h]
\includegraphics{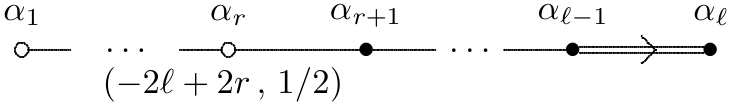} 
\caption{Marked Satake diagram for 
$\SO_{2\ell+1}/\SO_{r} \times \SO_{2\ell+1-r}$ }
\label{diagbi.fig}
\end{figure}
%%%%%%%%%%%%%%%%%%%%%%%%%%%%%%%%%%%%%%%%%

\noindent 
From the Satake diagram and the fact that $|\Supp \mu_i| = 1$ for $i
= 1, \ldots, r$, we know that $\frc_0 = 0$ and 
$  \frm \cong \fso_{2(\ell - r) + 1}$,
\ so we have  \ 
$ \frh_{\frm}^{0} = (2\ell - 2r)\alpha_{2r+1} + \cdots + 2\alpha_{\ell}$
\ and 
\begin{equation}
\label{bishift2}
 \langle \, \varpi_{\frm}^0 \mid \alpha_i \,\rangle =
 \begin{cases} 
                1/2 & \text{if $i = r$\,,}\\
               0 & \text{if  $i = r+1, \ldots, \ell - 1$\,,}\\
               -1/2 & \text{if $i = \ell$\,,}
 \end{cases}               
\end{equation}
as in Section \ref{rankone.sec}, Case 2. The markings on the diagram follow
from this. The only nest with more than one root is 
\[
 \Phi^{+}(\lambda_r) 
   = \{\beta_{j} \,:\, r+1 \leq j \leq \ell \}   
   \cup
   \{ \gamma_{j} \,:\, r+1 \leq j \leq \ell  \}   
   \cup
   \{\alpha_r + \cdots + \alpha_{\ell} \} \,,  
\]
where $\beta_{j} = \alpha_r + \cdots + \alpha_{j-1}$ and $\gamma_{j} =
\beta_{j} + 2\alpha_{j} + \cdots + 2\alpha_{\ell}$. Thus $\left|
\Phi^{+}(\lambda_r) \right| = 2(\ell - r) + 1$ and there is one basic root
$\alpha_r$ as indicated in the table.

The eigenvalues of $\ad h_{\frm}^{0}$ on $\frn_{\lambda_r}$ are
\[
  -2\ell + 2r, \,  \ldots \,, - 2, \, 0 , 
  \, 2, \, \ldots \,, 2\ell - 2r\,,
\]
each with multiplicity one, with the negative eigenvalues coming from
$\{\beta_j\}$ and the positive eigenvalues from $\{\gamma_j\}$, as in Section
\ref{rankone.sec}, Case 2.  From (\ref{bishift2}) we have 
\[
 \langle \, \varpi_{\frm}^0 \mid \beta_j \,\rangle = 1/2, \quad 
 \langle \, \varpi_{\frm}^0 \mid \gamma_j \,\rangle = -1/2, \quad
 \langle\, \varpi_{\frm}^0 \mid \alpha_r + \cdots + \alpha_{\ell} \,\rangle 
 = 0.
\]
Hence the $\rho_{\frm}$ shifts in the dimension factor for $\lambda_r$  are
\begin{equation}
\label{birhoshift}
  -\ell + r + \textstyle{\frac{1}{2}}, \, \ldots \,, - \textstyle{\frac{1}{2}}
  , \, 0 ,   \, \textstyle{\frac{1}{2}}, \, \ldots \,,  \ell - r  -
   \textstyle{\frac{1}{2}}\,.
\end{equation}
Since $ \ell - r - \frac{1}{2} = \frac{1}{2}m_{\lambda_r} - 1$, we conclude
that
$  d_{\lambda_r}(\lambda) = 
  W\big( \langle\, \lambda  \mid \lambda_r \,\rangle,
  \langle\, \delta \mid \lambda_r \,\rangle \,;\, 
   m_{\lambda_r} \big)$.

The restricted root system is of type $B_{r}$ with $\lambda_r$ the short
simple root. Each long restricted root $\xi$ is conjugate under the action of
the Weyl group of $G/K$ to $\alpha_1$. Thus $\xi$ has multiplicity one so by
Proposition \ref{rhopishift.prop} we conclude that (\ref{regdimfact}) holds
for $\xi$.

The positive short restricted roots are 
$\xi_i = \lambda_{i} + \cdots + \lambda_{r}$
for $1 \leq i \leq r$ with basic roots $\alpha_{i} + \cdots +
\alpha_{r}$. If the positive roots of $G$ are defined relative to the standard
ordered basis $\varepsilon_1 < \varepsilon_2 < \cdots < \varepsilon_{\ell}$
for $\frt^{*}$, then
$\alpha_{i} + \cdots + \alpha_{r} = \varepsilon_i - \varepsilon_{r+1}$.  
The Weyl group of $G$ consists of all signed permutations of $\{1, \ldots,
\ell \}$. Let $w$ be the permutation $i
\leftrightarrow r$. Then $w$ sends $\alpha_{i} + \cdots + \alpha_{r} \to
\alpha_{r}$ and fixes the roots in $\Delta_{0}$, so we can apply Lemma
\ref{rootnest.lem} to conclude that  (\ref{regdimfact}) holds for $\xi_i$.
Thus all restricted root nests are regular and the dimension formula
\,{\rm (\ref{symspacedim})} is valid for all $\lambda \in \Gamma(G/K)$.

% \textbf{\em Case 4.} {\bf Type C\,II.}
%            \input{ciiexam}
           %%%%%%%% diagrams:    newcii2.eps,   newcii3.eps

%%%%%%%%%%%%%%%%%%%%%%%%%%%%%%%%%%%%%%%%%%%%%%%%%%%%%%%%%%%%%%%%%%%%%%%%%%%%%
%  Created       : 3/11/12
%  Last Modified : 9/11/12 for short version
%
%%%%%%%%%%%%%%%%%%%%%%%%%%%%%%%%%%%%%%%%%%%%%%%%%%%%%%%%%%%%%%%%%%%%%%%%%%%%%

\vspace{2ex}
\noindent
\textbf{\em Case 4.} {\bf Type C\,II.}
Let $G = \Sp_{2\ell}(\bC)$ and $K =\Sp_{2r}(\bC)\times \Sp_{2\ell - 2r}(\bC)$
with $\ell \geq 2r + 1$ and  $r \geq 2$.
The $K$-spherical highest weights are $\mu_i = \varpi_{2i}$ for $i = 1,
\ldots, r$. The simple restricted root data are as
follows.

\vspace{1ex}

\begin{center}

\begin{tabular}{|c |c| c |}
\hline
 restricted root   & multiplicity & \# basic roots
 \\
\hline
 \rule[-1ex]{0ex}{3ex}  $ \lambda_{2i}$ \quad ($1 \leq i \leq r-1$)  & 
   $4$  & $1$
 \\
\hline
 \rule[-1ex]{0ex}{3ex}   $ \lambda_{2r} $  & $4(\ell - 2r)$  & $1$ 
\\
\hline
 \rule[-1ex]{0ex}{3ex}   $ 2\lambda_{2r} $  & $3$  & $1$ 
\\
\hline
\end{tabular}

\end{center}

%%%%%%%%%%%%%%%%%%% 
\begin{figure}[h]
\includegraphics{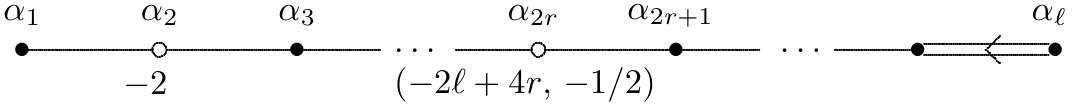} 
\caption{Marked Satake diagram for 
$\Sp_{2\ell}/\Sp_{2r}\times \Sp_{2(\ell - r)}$ \ ($\ell \geq 2r + 1$) }
\label{diagcii2.fig}
\end{figure}
%%%%%%%%%%%%%%%%%%%%%%%%%%%%%%%%%%%%%%%%%

\noindent
From the Satake diagram and the fact that $|\Supp \mu_i| = 1$ for $i
= 1, \ldots, r$, we know that $\frc_0 = 0$ and 
$  \frm \cong
   \underbrace{
 \fsl_{2} \oplus \cdots \oplus \fsl_{2}}_{r {\rm \ copies}} 
 \oplus \fsp_{2(\ell - 2r)}$
\ and 
\begin{equation}
\label{ciishift1}
 \frh_{\frm}^{0} = \alpha_1 + \alpha_3 + \cdots + \alpha_{2r-1}
   + (2\ell -4r -1)\alpha_{2r+1} + \cdots \,.
\end{equation}     
As in Section \ref{rankone.sec}, Case 4, we calculate that
\begin{equation}
\label{ciishift2}
 \langle \, \varpi_{\frm}^0 \mid \alpha_i \,\rangle =
 \begin{cases} 0 & \text{if $1 \leq i \leq 2r -1$ or
                   $2r+ 1 \leq i \leq  \ell-1$\,,} \\
               -1/2 & \text{if $i = 2r$\,,}\\
               1 & \text{if $i = \ell$\,.}
 \end{cases}               
\end{equation}
The markings on the diagram follow from (\ref{ciishift1}) and
(\ref{ciishift2}). 

For the restricted root $\lambda_{2i}$ with $1 \leq i \leq r-1$
the root nest is 
\begin{equation}
\label{ciinest}
  \Phi^{+}\big(\lambda_{2i}\big)  =
  \{\, \alpha_{2i}\,,\, 
       \alpha_{2i-1} + \alpha_{2i}\,,\, 
       \alpha_{2i} + \alpha_{2i+1}\,,\, 
       \alpha_{2i-1} + \alpha_{2i} + \alpha_{2i+1}\, \}
\end{equation}
with basic root $\alpha_{2i}$. Thus condition (4) of Proposition
\ref{regsimple.prop} is satisfied.

For the indivisible restricted root $\xi = \lambda_{2r}$ we have
\[
\begin{split}
 \Phi^{+}(\xi) &= 
   \{\beta_j \,:\, 2r+1 \leq j \leq \ell \}  
   \cup  \{\alpha_{2r-1} + \beta_j \,:\, 2r+1 \leq j \leq \ell \} \\
  & \quad  
   \cup  \{ \gamma_j \,:\, 2r+1 \leq j \leq \ell \} 
   \cup  \{\alpha_{2r-1} + \gamma_j   \,:\, 2r+1 \leq j \leq \ell \} \,,
\end{split}
\]  
where 
$\beta_j = \alpha_{2r} + \cdots + \alpha_{j-1}$ and $\gamma_j = \beta_j
+ 2\alpha_j + \cdots + 2\alpha_{\ell-1} + \alpha_{\ell}$ 
(here we take $\gamma_{\ell} = \beta_{\ell} + \alpha_{\ell}$). The basic root
is $\beta = \alpha_{2r}$ and there are $4(\ell - 2r)$ roots in the nest.
Furthermore,
\[
  \Phi^{+}(2\xi) = 
   \{ \beta'\,,\, \beta' + \alpha_{2r-1}\,,\, \beta' + 2\alpha_{2r-1} \} \,, 
\]
where 
$\beta' =  2\alpha_{2r} + \cdots + 2\alpha_{\ell-1} + \alpha_{\ell}$ 
is the basic root.
Using these root nests, formula (\ref{ciishift1}),  and the same
argument as in Section \ref{rankone.sec}, Case 4, we find that 
the values of $ \langle\, \rho_{\frm} \mid
\alpha \, \rangle$ for $\alpha \in \Phi^{+}(\xi)$ are
\[
\textstyle{ -\ell + 2r - \frac{1}{2}}\,,\,
 \underbrace{\textstyle{ 
  -\ell+ 2r + \frac{1}{2} \,,\, \ldots \,,\,  -\frac{3}{2}
  }}_{\mbox{multiplicity $2$}} \,,\,
\textstyle{ -\frac{1}{2} \,,\, \frac{1}{2} \,,\,}
 \underbrace{\textstyle{
 \frac{3}{2} \,,\, \ldots \,,\, \ell - 2r - \frac{1}{2}
 }}_{\mbox{multiplicity $2$}}
  \textstyle{  \,,\, \ell - 2r + \frac{1}{2}}\,.
\]
Likewise, the values of $ \langle\, \rho_{\frm} \mid \alpha \, \rangle$ for
$\alpha \in \Phi^{+}(2\xi)$ are $-1$, $0$, $1$. Hence for $\lambda \in
\Gamma(G/K)$ we have 
\begin{equation}
\label{cciidim}
\begin{split}
 d_{\xi}(\lambda)d_{2\xi}(\lambda) = 
 W\big( \langle\, \lambda  \mid \xi \,\rangle,
  \langle\, \delta \mid \xi \,\rangle \,;\, 
   m_{\xi}, m_{2\xi} \big)\,.
\end{split}
\end{equation}
Thus all the simple indivisible restricted roots for $G/K$ satisfy
(\ref{regdimfact}). 

Since each root nest for the simple restricted roots has only one basic root,
we can use the same strategy as for type {\bf B\,I} to prove that
(\ref{regdimfact}) holds for all positive indivisible restricted roots. Thus
it suffices to find elements of the Weyl group of $G$ that satisfy the
conditions of Lemma \ref{rootnest.lem}. It will then follow that all
restricted root nests are regular and the dimension formula \,{\rm
(\ref{symspacedim})} is valid for all $\lambda \in \Gamma(G/K)$. We carry out
this argument case-by-case.

Let the positive roots of $G$ be defined relative to an ordered basis
$\varepsilon_1 < \varepsilon_2 < \cdots < \varepsilon_{\ell}$ for $\frt^{*}$.
The Weyl group of $G$ consists of all signed permutations of $\{1, \ldots,
\ell \}$. The restricted root system is of type $BC_r$, with $\lambda_{2r}$
the short indivisible simple root.

\begin{enumerate} 
\item
Let \ 
$\xi_{ij} = \lambda_{2i} +  \cdots + \lambda_{2j}$
\ for $1 \leq i < j \leq r-1$ (long positive
indivisible root). The basic root for $\xi_{ij}$ is  
\[
 \alpha_{2i} + \alpha_{2i+1} + \cdots + \alpha_{2j} = 
   \varepsilon_{2i} - \varepsilon_{2j+1}\,.  
\]
Let $w$ be the signed permutation
\[
 \hspace*{5ex}
  2i + 1 \rightarrow 2j+1, \quad
  2i + 2 \rightarrow 2j+2,  \quad
  2j + 1 \rightarrow \overline{2i+2},  \quad
  2j + 2 \rightarrow \overline{2i+1}
\]
that fixes all other indices (where $p \rightarrow \overline{q}$ means
$\varepsilon_{p} \rightarrow -\varepsilon_{q}$). Then 
\[
 w: \alpha_{2i} \to \varepsilon_{2i} - \varepsilon_{2j+1}\,,
\quad
  w:\alpha_{2i+1} \leftrightarrow \alpha_{2j+1}\,, 
\]
and $w$ fixes the other roots in $\Delta_{0}$.

\vspace{1ex}

\item
Let \ 
$\eta_{ij} = \lambda_{2i} + \cdots + \lambda_{2j-2}
   + 2\lambda_{2j} + \cdots + 2\lambda_{2r}$
\ for $1 \leq i < j \leq r$ (long positive
indivisible root). The basic root for $\eta_{ij}$ is  
\[
 \alpha_{2i} + \cdots + \alpha_{2j-1} + 2\alpha_{2j} + \cdots
  + 2\alpha_{2\ell - 1} + \alpha_{\ell} = 
  \varepsilon_{2i} + \varepsilon_{2j}\,.  
\]
Let $w$ be the signed permutation
\[
  \hspace*{5ex}
 1 \leftrightarrow 2i-1, \quad
 2 \leftrightarrow 2i, \quad
 3 \rightarrow \overline{2j},  \quad
 4 \rightarrow \overline{2j-1},  \quad
 2j-1 \rightarrow 3, \quad
 2j  \rightarrow 4
\]
that fixes all other indices. Then 
\[
 w: \alpha_{2i} \to   \varepsilon_{2i} + \varepsilon_{2j}\,, 
\quad w:\alpha_{1} \leftrightarrow \alpha_{2i - 1}\,, 
\quad w:\alpha_{3} \leftrightarrow \alpha_{2j - 1}\,, 
\]
and $w$ fixes the other roots in $\Delta_{0}$.

\vspace{1ex}

\item
Let \ 
$\xi_{i} = \lambda_{2i} +  \cdots + \lambda_{2r}$
\ for $1 \leq i \leq r$ (short positive
indivisible root). The basic root for $\xi_{i}$ is  
\[
 \alpha_{2i} + \alpha_{2i+1} + \cdots + \alpha_{2r} = 
   \varepsilon_{2i} - \varepsilon_{2r+1}\,.  
\]
Let $w$ be the permutation
\ $ 2i-1   \leftrightarrow 2r-1$,
\ $ 2i  \leftrightarrow 2r$ \ 
that fixes all other indices. Then 
\[
 w: \alpha_{2r} \to \varepsilon_{2i} - \varepsilon_{2r+1}\,,
\quad
  w:\alpha_{2i-1} \leftrightarrow \alpha_{2r - 1}\,, 
\]
and $w$ fixes the other roots in $\Delta_{0}$.

\vspace{1ex}

\item
Let \ 
$2\xi_{i} = 2\lambda_{2i} +  \cdots + 2\lambda_{2r}$
\ for $1 \leq i \leq r$. The basic root for $2\xi_{i}$ is  
\[
 2\alpha_{2i} + 2\alpha_{2i+1} + \cdots + 2\alpha_{\ell-1} + \alpha_{\ell} = 
   2\varepsilon_{2i}\,.  
\]
Let $w$ be the permutation 
\ $ 2i -1   \leftrightarrow 2r-1$, \ 
$2i \leftrightarrow 2r$ \ 
that fixes all other indices. Then 
\[
 w: 2\varepsilon_{2i} \leftrightarrow  2\varepsilon_{2r}\,,
\quad
  w:\alpha_{2i-1} \leftrightarrow \alpha_{2r - 1}\,, 
\]
and $w$ fixes the other roots in $\Delta_{0}$.

\end{enumerate}

%%%%%%%%%%%%%%%%%%%%%%%%%%%%%%%%%%%%%%%%%%%%%%%%%%%%%%

\vspace{1ex}

Now let $G = \Sp_{4r}(\bC)$ and $K =\Sp_{2r}(\bC)\times \Sp_{2r}(\bC)$. The
$K$-spherical highest weights are $\mu_i = \varpi_{2i}$ for $i = 1, \ldots,
r$. The simple restricted root data  are as follows.

\vspace{1ex}

\begin{center}

\begin{tabular}{|c |c| c |}
\hline
 restricted root   & multiplicity & \# basic roots
 \\
\hline
 \rule[-1ex]{0ex}{3ex}  $ \lambda_{2i}$ \quad ($1 \leq i \leq r-1$)  & 
   $4$  & 1
 \\
\hline
 \rule[-1ex]{0ex}{3ex}   $ \lambda_{2r} $  & $3$  & 1 
\\
\hline
\end{tabular}

\end{center}

%\vspace{1ex}

%%%%%%%%%%%%%%%%%%% 
\begin{figure}[h]
\includegraphics{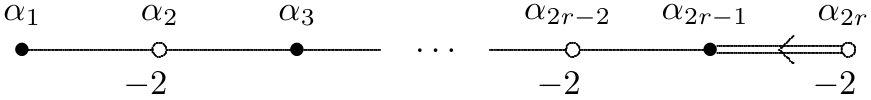} 
\caption{Marked Satake diagram for $\Sp_{4r}/\Sp_{2r}\times \Sp_{2r}$ }
\label{diagcii3.fig}
\end{figure}
%%%%%%%%%%%%%%%%%%%%%%%%%%%%%%%%%%%%%%%%%

\noindent
From the Satake diagram  and the fact that $|\Supp \mu_i| = 1$ for $i
= 1, \ldots, r$, we know that $\frc_0 = 0$ and
$
  \frm \cong
   \underbrace{
 \fsl_{2} \oplus \cdots \oplus \fsl_{2}}_{r {\rm \ copies}} 
$
\ and \ 
$ \frh_{\frm}^{0} = \alpha_1 + \alpha_3 + \cdots + \alpha_{2r-1}$.
This gives the markings on the diagram. 
The root nests are given by (\ref{ciinest}) with basic root $\alpha_{2i}$ for
$1 \leq i \leq r-1$, and 
$
 \Phi^{+}\big(\lambda_{2r}\big) = 
    \{\, \alpha_{2r}\,,\, 
       \alpha_{2r-1} + \alpha_{2r}\,,\, 
       2\alpha_{2r-1} + \alpha_{2r} \, \}
$
with basic root $\alpha_{2r}$. Thus condition (4) of Proposition
\ref{regsimple.prop} is satisfied by $\lambda_{2i}$ for $1\leq i \leq r-1$ and
condition (3) is satisfied by $\lambda_{2r}$.

% \textbf{\em Case 5.} {\bf Type D\,I.}
%            \input{diexam}
           %%%%%%%% diagram:      newdi.eps 

%%%%%%%%%%%%%%%%%%%%%%%%%%%%%%%%%%%%%%%%%%%%%%%%%%%%%%%%%%%%%%%%%%%%%%%%%%%%%
%  Created       : 3/11/12
%  Last Modified : 9/11/12 for short version
%
%%%%%%%%%%%%%%%%%%%%%%%%%%%%%%%%%%%%%%%%%%%%%%%%%%%%%%%%%%%%%%%%%%%%%%%%%%%%%

\vspace{2ex}
\noindent
\textbf{\em Case 5.} {\bf Type D\,I.}
Let $G = \SO_{2\ell}(\bC)$ and $K = \SO_{r}(\bC) \times \SO_{2\ell - r}(\bC)$ 
with $2 \leq r < \ell$.
 The fundamental $K$-spherical highest weights are
$\mu_i = 2\varpi_{i}$ for $i = 1, \ldots, r-1 $, together with 
\[
 \mu_r = \begin{cases} \varpi_{r} & \mbox{if $r \leq \ell - 2$}\,,
            \\
      \varpi_{\ell - 1} +  \varpi_{\ell} & \mbox{if $r = \ell -1$\,.}
    \end{cases}
\]
The simple restricted root data are as follows.

\vspace{1ex}

\begin{center}

\begin{tabular}{|c |c| c |}
\hline
 \rule[-2ex]{0ex}{5ex} \small{restricted root}   & \small{mult.} 
  & \parbox{8ex}{ \small{\# basic roots}}
 \\
\hline
  \rule[-1ex]{0ex}{3ex} $ \lambda_{i}$ \ 
  \small{($\scriptstyle{1 \leq i \leq r-1}$)}  &  $1$ & $1$ 
 \\
\hline
 \rule[-1ex]{0ex}{3ex}  $ \lambda_r $  & $2$ &  $2$ 
\\
\hline
\end{tabular}
\quad
\begin{tabular}{|c |c| c | }
\hline
 \rule[-2ex]{0ex}{5ex} \small{restricted root}   & \small{mult.}
  & \parbox{8ex}{\small{ \# basic roots} }
 \\
\hline
  \rule[-1ex]{0ex}{3ex} $ \lambda_{i}$ \ 
  \small{($\scriptstyle{1 \leq i \leq r-1}$)}  & $1$ & $1$
 \\
\hline
  \rule[-1ex]{0ex}{3ex}  $ \lambda_r $  & $2(\ell - r)$ & $1$
\\
\hline
\end{tabular}

\end{center}

%%%%%%%%%%%%%%%%%%% 
\begin{figure}[h]
\includegraphics{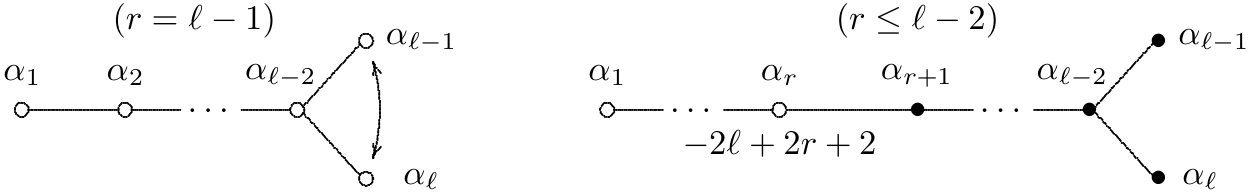} 
\caption{Marked Satake diagrams for 
$\SO_{2\ell}/\SO_{r} \times \SO_{2\ell-r}$ }
\label{diagdi.fig}
\end{figure}
%%%%%%%%%%%%%%%%%%%%%%%%%%%%%%%%%%%%%%%%%

Let $r = \ell - 1$. Then  $\Delta_0 = \emptyset$ and $\dim \frc_0 =
(r+1) - r = 1$ by Lemma \ref{rankdim.lem} since $|\Supp \mu_r| = 2$. Identify
$\frt$ with $\frt^{*}$ using the form $\langle \cdot \mid \cdot \rangle$. If
$\bx = c_1\alpha_1 + \cdots + c_{\ell}\alpha_{\ell}$ is in $\frt$, then the
equations $\langle \mu_i \mid \bx \rangle = 0$ for $i = 1, \ldots , r$ become
\[
 c_{i} = 0 \quad\mbox{for $i = 1, \ldots , r-2$ and } \ 
 c_{\ell - 1} = - c_{\ell}\,.
\]    
Hence $\by = \alpha_{\ell - 1} - \alpha_{\ell}$ is a basis for $\frc_0$.
Condition (1) of Proposition \ref{regsimple.prop}
is satisfied by $\lambda_i$ for $1\leq i \leq r-1$ and condition (2) is
satisfied by $\lambda_{r}$.

Now assume $r \leq \ell - 2$. Then $\dim \frc_0 = 0$ by Lemma
\ref{rankdim.lem} since $|\Supp \mu_i| = 1$ for $i = 1, \ldots, r$. 
When $r < \ell - 2$, then  from the
Satake diagram we conclude that $\frm \cong \fso_{2(\ell - r)}$ and hence 
 $\frh_{\frm}^{0} = 2(\ell - r - 1)\alpha_{r+1} + \cdots $.
Thus $\langle \frh_{\frm}^{0} \mid \alpha_{r} \rangle = -2\ell + 2r + 1$ as
indicated in Figure \ref{diagdi.fig}. As in Section \ref{rankone.sec}, Case 3, we have
\[
 \Phi^{+}(\lambda_r) 
   = \{\beta_{j} \,:\, r+1 \leq j \leq \ell \}   
   \cup
   \{ \gamma_{j} \,:\, r+1 \leq j \leq \ell -1  \}   
   \cup\{\alpha_r + \cdots + \alpha_{\ell-2} + \alpha_{\ell} \} \,,
\]
where $\beta_{j} = \alpha_r + \cdots + \alpha_{j-1}$ and $\gamma_{j} =
\beta_{j} + 2\alpha_{j} + \cdots + 2\alpha_{\ell-2} + \alpha_{\ell-1} +
\alpha_{\ell}$ (the roots with coefficient $2$ are omitted when $j = \ell -
1$). Thus $\left| \Phi^{+}(\xi_1) \right| = 2\ell - 2r$ and the only basic
root in the nest is $\alpha_{r}$.
Thus condition (1) of Proposition \ref{regsimple.prop} is satisfied by
$\lambda_i$ for $1\leq i \leq r-1$ and condition (4) is satisfied by
$\lambda_{r}$.

Finally, let $r = \ell - 2$. Then $\frm \cong \fsl_{2} \oplus \fsl_{2}$ and
 $\frh_{\frm}^{0} = \alpha_{\ell-1} + \alpha_{\ell}$.
Hence $\langle \frh_{\frm}^{0} \mid \alpha_{r} \rangle = -2$ as
indicated in Figure \ref{diagdi.fig} and 
\[
 \Phi^{+}(\lambda_r) 
   = \{\beta\,,\, \beta + \alpha_{\ell - 1}\,,\, \beta + \alpha_{\ell} 
    \,,\, \beta + \alpha_{\ell - 1} + \alpha_{\ell} \} 
\]
where $\beta = \alpha_{r}$ is the basic root. Thus condition (1) of
Proposition \ref{regsimple.prop} is satisfied by $\lambda_i$ for $1\leq i \leq
r-1$ and condition (4) is satisfied by $\lambda_{r}$.

% \textbf{\em Case 6.} {\bf Type D\,III.}
%            \input{diiiexam}
           %%%%%%%% diagram:    newssdiii.eps

%%%%%%%%%%%%%%%%%%%%%%%%%%%%%%%%%%%%%%%%%%%%%%%%%%%%%%%%%%%%%%%%%%%%%%%%%%%%%
%  Created       : 3/11/12
%  Last Modified : 9/11/12 for short version
%
%%%%%%%%%%%%%%%%%%%%%%%%%%%%%%%%%%%%%%%%%%%%%%%%%%%%%%%%%%%%%%%%%%%%%%%%%%%%%

\vspace{2ex}
\noindent
\textbf{\em Case 6.} {\bf Type D\,III.}
Let $G = \SO_{2\ell}(\bC)$ and  $K = \GL_{\ell}(\bC) $
with $\ell \geq 4$. The fundamental $K$-spherical highest weights are
$\mu_i = \varpi_{2i}$ for $i = 1, \ldots, r-1 $ together with 
\[
 \mu_r = \begin{cases}2\varpi_{\ell}
   &\mbox{when $\ell = 2r$\,,}
   \\
   \varpi_{\ell-1} + \varpi_{\ell}
   &\mbox{when $\ell = 2r+1$\,.}
   \end{cases}
\]   
The simple restricted root data are as follows.

\vspace{1ex}

\begin{center}

\begin{tabular}{|c |c| c |}
\hline
 \rule[-2ex]{0ex}{5ex} \small{restricted root}   & \small{mult.} 
  & \parbox{8ex}{ \small{\# basic roots} }
 \\
\hline
  \rule[-1ex]{0ex}{3ex} $ \lambda_{2i}$ \  
  \small{($\scriptstyle{1 \leq i \leq r-1}$)}  &  
   $4$ & $1$ 
 \\
\hline
 \rule[-1ex]{0ex}{3ex}  $ \lambda_{2r} $  & $1$ &  $1$ 
\\
\hline
\end{tabular}
\qquad
\begin{tabular}{|c |c| c | }
\hline
 \rule[-2ex]{0ex}{5ex} \small{restricted root}   & \small{mult.}
  & \parbox{8ex}{ \small{\# basic roots} }
 \\
\hline
  \rule[-1ex]{0ex}{3ex} $ \lambda_{2i}$ \ 
  \small{($\scriptstyle{1 \leq i \leq r-1}$)}  & $4$ & $1$
 \\
\hline
  \rule[-1ex]{0ex}{3ex}  $ \lambda_{2r} $  & $4$ & $2$
\\
\hline
  \rule[-1ex]{0ex}{3ex}  $ 2\lambda_{2r} $  & $1$ & $1$
\\
\hline
\end{tabular}

\end{center}

%%%%%%%%%%%%%%%%%%% 
\begin{figure}[h]
\includegraphics{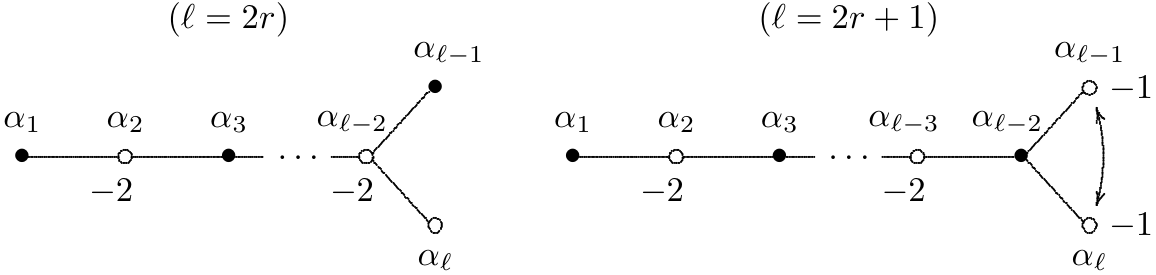} 
\caption{Marked Satake diagrams for 
$\SO_{2\ell}/\GL_{\ell}$ }
\label{diagssdiii.fig}
\end{figure}
%%%%%%%%%%%%%%%%%%%%%%%%%%%%%%%%%%%%%%%%%

Assume $\ell =2r$. Then $\dim \frc_0 = 0$ by Lemma \ref{rankdim.lem} since
$|\Supp \mu_i| = 1$ for $i = 1, \ldots, r$. Hence from the Satake diagram
$\frm \cong \fsl_{2} \oplus \cdots \oplus \fsl_2$ \ ($r$ copies) and
 $\frh_{\frm}^{0} = \alpha_{1} +  \alpha_{3} + \cdots  + \alpha_{\ell - 1}$.
Thus $\langle \frh_{\frm}^{0} \mid \alpha_{2i} \rangle = -2$ for $i = 1,
\ldots, \ell-1$ as indicated in Figure \ref{diagssdiii.fig}. 
The root nests are 
$\Phi^{+}\big(\lambda_{\ell}\big)  = \{ \alpha_{\ell} \}$ together with
\begin{equation}
\label{diiinest}
  \Phi^{+}\big(\lambda_{2i}\big)  =
  \{\, \alpha_{2i}\,,\, 
       \alpha_{2i-1} + \alpha_{2i}\,,\, 
       \alpha_{2i} + \alpha_{2i+1}\,,\, 
       \alpha_{2i-1} + \alpha_{2i} + \alpha_{2i+1}\, \}
\end{equation}
for $i = 1, \ldots, r-1$ with basic root  $\alpha_{2i}$.
Thus condition (4) of Proposition \ref{regsimple.prop} is satisfied by
$\lambda_{2i}$ for $1\leq i \leq r-1$ and condition (1) is satisfied by
$\lambda_{2r}$.

%%%%%%%%%%%%%%%%%%%%%%%%%%%%%%%%

Now assume  $\ell = 2r + 1$. Then $\dim \frc_0 = (r+1) - r = 1$ by Lemma
\ref{rankdim.lem} since $|\Supp \mu_r| = 2$. 
From the Satake diagram we obtain $\frm'
\cong \fsl_{2} \oplus \cdots \oplus \fsl_2$ \ ($r-1$ copies) and
 $\frh_{\frm}^{0} = \alpha_{1} +  \alpha_{3} + \cdots   \alpha_{\ell - 2}$.
Hence $\langle \frh_{\frm}^{0} \mid \alpha_{2i} \rangle = -2$ for $i = 1,
\ldots, \ell-3$, while 
$\langle \frh_{\frm}^{0} \mid \alpha_{i} \rangle = -1$ for $i = \ell-1$ and
$i= \ell$, as indicated in Figure \ref{diagssdiii.fig}. The root nest
$\Phi^{+}\big(\lambda_{2i}\big)$ is given by (\ref{diiinest}) for $i = 1,
\ldots, r-1$, while 
\[
\left\{
\begin{split}
 \Phi^{+}\big(\lambda_{2r}\big) &= 
  \{\, \alpha_{\ell-1}\,,\, \alpha_{\ell} \,,\, 
  \alpha_{\ell} + \alpha_{\ell-2}   \,,\,
  \alpha_{\ell-1} + \alpha_{\ell-2}  \, \} 
\quad\mbox{(basic roots $\alpha_{\ell-1}\,,\, \alpha_{\ell}$)\,,} 
\\
 \Phi^{+}\big(2\lambda_{2r}\big) &= 
  \{\, \alpha_{\ell-2} + \alpha_{\ell-1} + \alpha_{\ell} \,\} \,.
\end{split}
\right.
\]
Thus condition (4) of
Proposition \ref{regsimple.prop} is satisfied by $\lambda_{2i}$ for $1\leq i
\leq r-1$ and condition (5) is satisfied by $\lambda_{\ell}$.

% \textbf{\em Case 7.} {\bf Type E\,II.}
%            \input{eiiexam}
           %%%%%%%% diagram:   newsseii.eps

%%%%%%%%%%%%%%%%%%%%%%%%%%%%%%%%%%%%%%%%%%%%%%%%%%%%%%%%%%%%%%%%%%%%%%%%%%%%%
%  Created       : 3/20/12
%  Last Modified : 9/11/12 for short version
%
%%%%%%%%%%%%%%%%%%%%%%%%%%%%%%%%%%%%%%%%%%%%%%%%%%%%%%%%%%%%%%%%%%%%%%%%%%%%%

\vspace{2ex}
\noindent
\textbf{\em Case 7.} {\bf Type E\,II.}
Let $G$ be the complex exceptional group of type ${\bf E}_6$ and 
$K = \SL_{6} \times \SL_{2}$.
The fundamental $K$-spherical highest weights are
$\mu_1 = \varpi_{1} + \varpi_{6}$, \, $\mu_2 = \varpi_{3} + \varpi_{5}$, 
\, $ \mu_3 = 2\varpi_{4}$, and $ \mu_4 = 2\varpi_{2}$.  
The simple restricted root data are as follows.

\vspace{1ex}

\begin{center}

\begin{tabular}{|c |c| c |}
\hline
 restricted root   & multiplicity & \# basic roots
 \\
\hline
 \rule[-1ex]{0ex}{3ex} $ \lambda_{2}$\,, \ $\lambda_{4}$  &  $1$ & $1$
 \\
\hline
 \rule[-1ex]{0ex}{3ex}  $ \lambda_{1}$\,, $ \lambda_{3}$   &  $2$ & $2$
 \\
\hline
\end{tabular}

\end{center}

%%%%%%%%%%%%%%%%%%% 
\begin{figure}[h]
\includegraphics{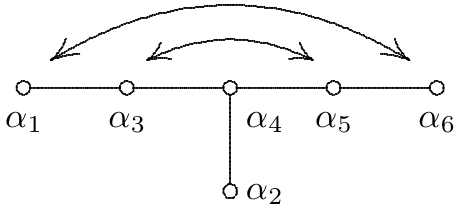} 
\caption{Satake diagram for 
${\bf E}_{6}/\SL_{6}\times \SL_{2}$}
\label{diagsseii.fig}
\end{figure}
%%%%%%%%%%%%%%%%%%%%%%%%%%%%%%%%%%%%%%%%%

\noindent
The root data follows from the Satake diagram. Thus condition (1) of
Proposition \ref{regsimple.prop} is satisfied by $\lambda_2$ and $\lambda_4$,
while condition (2) is satisfied by $\lambda_1$ and $\lambda_3$.

We have $\Delta_0 = \emptyset$ and $\dim \frc_0 = 6 - 4 = 2$ by Lemma
\ref{rankdim.lem} since $|\Supp \mu_i| = 2$ for $i = 1,2$. 

% \textbf{\em Case 8.} {\bf Type E\,III.}
%            \input{eiiiexam}
           %%%%%%%% diagram:   newsseiii.eps

%%%%%%%%%%%%%%%%%%%%%%%%%%%%%%%%%%%%%%%%%%%%%%%%%%%%%%%%%%%%%%%%%%%%%%%%%%%%%
%  Created       : 3/19/12
%  Last Modified : 9/11/12 for short version
%
%%%%%%%%%%%%%%%%%%%%%%%%%%%%%%%%%%%%%%%%%%%%%%%%%%%%%%%%%%%%%%%%%%%%%%%%%%%%%

\vspace{2ex}
\noindent
\textbf{\em Case 8.} {\bf Type E\,III.}
Let $G$ be the complex exceptional group of type ${\bf E}_6$ and
$K$ the connected subgroup with Lie algebra 
$\fso_{10}(\bC) \oplus \fso_{2}(\bC)$.   The fundamental $K$-spherical
highest weights are $\varpi_1 + \varpi_6$ and $\varpi_2$.
The simple restricted root data  are as follows.

\vspace{1ex}

\begin{center}

\begin{tabular}{|c |c| c |}
\hline
 restricted root   & multiplicity & \# basic roots
 \\
\hline
 \rule[-1ex]{0ex}{3ex} $ \lambda_{2}$  &  $6$ & $1$
 \\
\hline
 \rule[-1ex]{0ex}{3ex}  $ \lambda_{1}$  &  $8$ & $2$
 \\
\hline
 \rule[-1ex]{0ex}{3ex}  $ 2\lambda_{1}$  &  $1$ & $1$
 \\
\hline
\end{tabular}

\end{center}

\vspace{1ex}

%%%%%%%%%%%%%%%%%%% 
\begin{figure}[h]
\includegraphics{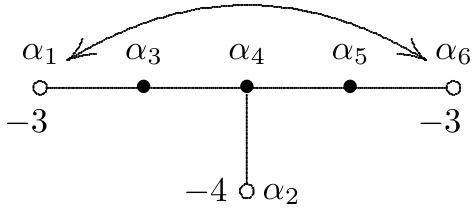} 
\caption{Marked Satake diagram for 
${\bf E}_{6}/\SO_{10}\times \SO_{2}$}
\label{diagsseiii.fig}
\end{figure}
%%%%%%%%%%%%%%%%%%%%%%%%%%%%%%%%%%%%%%%%%

\noindent
From the Satake diagram we see that $\frm' \cong \fsl_{4}$  and thus
$h_{\frm}^{0} = 3\alpha_{2} + 4\alpha_{4} + 3\alpha_{5}$. This gives the
markings in the Satake diagram.
Using the notation of Section \ref{highrank.sec}, Case 8, we obtain a Cartan
subspace $\fra$ for $G/K$ by the equation $\xi_1 = \xi_3$; the root nests for
the simple restricted roots are
\[
  \Phi^{+}(\lambda_1) = \Phi^{+}(\xi_1)\cup \Phi^{+}(\xi_3)\,,
\quad
  \Phi^{+}(\lambda_2) = \Phi^{+}(\xi_2)\,,
\quad
  \Phi^{+}(2\lambda_1)  = \Phi^{+}(\xi_1 + \xi_3)\,.
\]
Since $h_{\frm}^{0}$ is the same for $G/K$ and $G/H$ (where $H =
\SO_{10}(\bC)$), the determination of the number of basic roots follows from
the calculations in Section \ref{highrank.sec}, Case 8. Thus condition (5) of
Proposition \ref{regsimple.prop} is satisfied by $\lambda_{1}$ and condition
(4) is satisfied by $\lambda_2$.

We have $\Delta_0 = \{\alpha_3, \alpha_4, \alpha_5\}$ and $\dim \frc_0 = 3 - 2
= 1$ by Lemma \ref{rankdim.lem} since $|\Supp \mu_1| = 2$.

% \textbf{\em Case 9.} {\bf Type E\,IV.}
%            \input{eivexam}
           %%%%%%%% diagram:    neweiv.eps
%%%%%%%%%%%%%%%%%%%%%%%%%%%%%%%%%%%%%%%%%%%%%%%%%%%%%%%%%%%%%%%%%%%%%%%%%%%%%
%  Created       : 3/11/12
%  Last Modified : 9/11/12 for short version
%
%%%%%%%%%%%%%%%%%%%%%%%%%%%%%%%%%%%%%%%%%%%%%%%%%%%%%%%%%%%%%%%%%%%%%%%%%%%%%

\vspace{2ex}
\noindent
 \textbf{\em Case 9.} {\bf Type E\,IV.}
Let $G$ be the complex exceptional group of type ${\bf E}_6$ and 
$K$ the complex exceptional group of type ${\bf F}_4$.
The fundamental $K$-spherical highest weights are
$\mu_1 = \varpi_{1}$ and
$ \mu_2 = \varpi_{6}$.
The simple restricted root data are as follows.

\vspace{1ex}

\begin{center}

\begin{tabular}{|c |c| c |}
\hline
 restricted root   & multiplicity & \# basic roots
 \\
\hline
 \rule[-1ex]{0ex}{3ex} $ \lambda_{1}$\,, \ $\lambda_{6}$  &  $8$ & $1$
 \\
\hline
\end{tabular}

\end{center}

\vspace{1ex}

%%%%%%%%%%%%%%%%%%% 
\begin{figure}[h]
\includegraphics{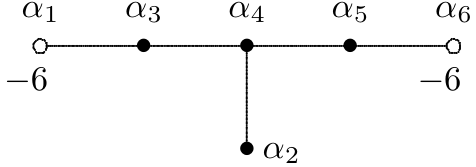} 
\caption{Marked Satake diagram for 
${\bf E}_6/{\bf F}_4$ }
\label{diageiv.fig}
\end{figure}
%%%%%%%%%%%%%%%%%%%%%%%%%%%%%%%%%%%%%%%%%

We have $\frc_0 = 0$ by Lemma
\ref{rankdim.lem} since $|\Supp \mu_i| = 1$ for $i = 1,2$. 
Hence from the Satake diagram we obtain $\frm = \fso_8$ and 
$  h_{\frm}^{0} = 6\alpha_3 + 10\alpha_4 + 6\alpha_2 + 6\alpha_5$,
which gives the indicated markings. Since $m_{\xi} = k_{\xi} + 2$ for both
simple restricted roots $\xi$, condition (4) of Proposition \ref{regsimple.prop} is
satisfied.

% \textbf{\em Case 10.} {\bf Type E\,VI.}
%            \input{eviexam}
           %%%%%%%% diagram:    newevi.eps
%%%%%%%%%%%%%%%%%%%%%%%%%%%%%%%%%%%%%%%%%%%%%%%%%%%%%%%%%%%%%%%%%%%%%%%%%%%%%
%  Created       : 3/11/12
%  Last Modified : 9/11/12 for short version
%
%%%%%%%%%%%%%%%%%%%%%%%%%%%%%%%%%%%%%%%%%%%%%%%%%%%%%%%%%%%%%%%%%%%%%%%%%%%%%

\vspace{2ex}
\noindent
\textbf{\em Case 10.} {\bf Type E\,VI.}
Let $G$ be the complex exceptional group of type ${\bf E}_7$ and let $K$ be
the connected subgroup with Lie algebra $\fso_{12}(\bC) \oplus \fsl_{2}(\bC)$.
The fundamental $K$-spherical highest weights are $\mu_1 = 2\varpi_{1}$, $
\mu_2 = 2\varpi_{3}$, $ \mu_3 = \varpi_{4}$, and $ \mu_4 = \varpi_{6}$.
The simple restricted  root data are as follows.

\vspace{1ex}

\begin{center}

\begin{tabular}{|c |c| c |}
\hline
 restricted root   & multiplicity & \# basic roots
 \\
\hline
 \rule[-1ex]{0ex}{3ex} $ \lambda_{1}$\,, \ $\lambda_{3}$  &  $1$ & $1$
 \\
\hline
 \rule[-1ex]{0ex}{3ex} $ \lambda_{4}$\,,  \ $\lambda_{6}$  &  $4$ &  $1$
 \\
\hline
\end{tabular}

\end{center}

\vspace{1ex}

%%%%%%%%%%%%%%%%%%% 
\begin{figure}[h]
\includegraphics{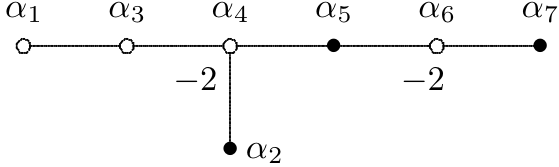} 
\caption{Marked Satake diagram for 
${\bf E}_7/\SO_{12}^{'}\times \SL_2$ }
\label{diagevi.fig}
\end{figure}
%%%%%%%%%%%%%%%%%%%%%%%%%%%%%%%%%%%%%%%%%

We have $\frc_0 = 0$ by Lemma \ref{rankdim.lem} since $|\Supp \mu_i| = 1$ for
$i = 1,\ldots, 4$. Hence from the Satake diagram we obtain $\frm \cong
\fsl_{2} \oplus \fsl_{2} \oplus \fsl_2$ and
 $\frh_{\frm}^{0} = \alpha_{2} +  \alpha_{5} + \alpha_{7}$,
which gives the indicated markings. Condition (4) of Proposition
\ref{regsimple.prop} is satisfied by $\lambda_4$ and $\lambda_6$, while
condition (1) is satisfied by $\lambda_1$ and $\lambda_3$.

% \textbf{\em Case 11.} {\bf Type E\,VII.}
%            \input{eviiexam}
           %%%%%%%% diagram:    newevii.eps

%%%%%%%%%%%%%%%%%%%%%%%%%%%%%%%%%%%%%%%%%%%%%%%%%%%%%%%%%%%%%%%%%%%%%%%%%%%%%
%  Created       : 3/11/12
%  Last Modified : 9/11/12 for short version
%
%%%%%%%%%%%%%%%%%%%%%%%%%%%%%%%%%%%%%%%%%%%%%%%%%%%%%%%%%%%%%%%%%%%%%%%%%%%%%

\vspace{2ex}
\noindent
\textbf{\em Case 11.} {\bf Type E\,VII.}
Let $G$ be the complex exceptional group of type ${\bf E}_7$ and 
$K = {\bf E}_{6} \times \SO_{2}$.
The fundamental $K$-spherical highest weights are
$\mu_1 = \varpi_{1}$, $ \mu_2 = \varpi_{6}$, and $ \mu_3 = 2\varpi_{7}$.
The simple restricted  root data are as follows.

\vspace{1ex}

\begin{center}

\begin{tabular}{|c |c| c |}
\hline
 restricted root   & multiplicity & \# basic roots
 \\
\hline
 \rule[-1ex]{0ex}{3ex} $ \lambda_{1}$\,, \ $\lambda_{6}$  &  $8$ & $1$
 \\
\hline
 \rule[-1ex]{0ex}{3ex} $ \lambda_{7}$   &  $1$ & $1$
 \\
\hline
\end{tabular}

\end{center}

%%%%%%%%%%%%%%%%%%% 
\begin{figure}[h]
\includegraphics{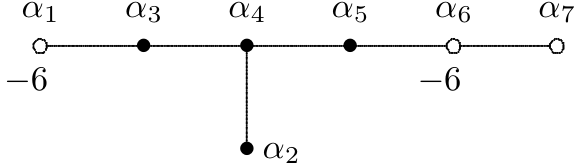} 
\caption{Marked Satake diagram for 
${\bf E}_7/{\bf E}_{6}\times \SO_2$ }
\label{diagevii.fig}
\end{figure}
%%%%%%%%%%%%%%%%%%%%%%%%%%%%%%%%%%%%%%%%%

We have $\frc_0 = 0$ by Lemma \ref{rankdim.lem} since $|\Supp \mu_i| = 1$ for
$i = 1, 2, 3$. Hence
from the Satake diagram we obtain $\frm = \fso_8$ and 
$  h_{\frm}^{0} = 6\alpha_3 + 10\alpha_4 + 6\alpha_2 + 6\alpha_5$,
which gives the indicated markings. Condition (4) of Proposition
\ref{regsimple.prop} is satisfied by $\lambda_1$ and $\lambda_6$, while
condition (1) is satisfied by $\lambda_7$.

% \textbf{\em Case 12.} {\bf Type E\,IX.}
%            \input{eixexam}
           %%%%%%%% diagram:   neweix.eps
%%%%%%%%%%%%%%%%%%%%%%%%%%%%%%%%%%%%%%%%%%%%%%%%%%%%%%%%%%%%%%%%%%%%%%%%%%%%%
%  Created       : 3/11/12
%  Last Modified : 9/11/12 for short version
%
%%%%%%%%%%%%%%%%%%%%%%%%%%%%%%%%%%%%%%%%%%%%%%%%%%%%%%%%%%%%%%%%%%%%%%%%%%%%%

\vspace{2ex}
\noindent
\textbf{\em Case 12.} {\bf Type E\,IX.}
Let $G$ be the complex exceptional group of type ${\bf E}_8$ and 
$K = {\bf E}_{7} \times \SL_{2}$.
The fundamental $K$-spherical highest weights are
$\mu_1 = \varpi_{1}$, $ \mu_2 = \varpi_{6}$, $ \mu_3 = 2\varpi_{7}$,
and  $ \mu_4 = 2\varpi_{8}$.
The simple restricted  root data are as follows.

\vspace{1ex}

\begin{center}

\begin{tabular}{|c |c| c |}
\hline
 restricted root   & multiplicity & \# basic roots
 \\
\hline
 \rule[-1ex]{0ex}{3ex} $ \lambda_{1}$\,, \ $\lambda_{6}$  &  $8$ & $1$
 \\
\hline
 \rule[-1ex]{0ex}{3ex}  $ \lambda_{7}$\,, $ \lambda_{8}$   &  $1$ & $1$
 \\
\hline
\end{tabular}

\end{center}

\vspace{1ex}

%%%%%%%%%%%%%%%%%%% 
\begin{figure}[h]
\includegraphics{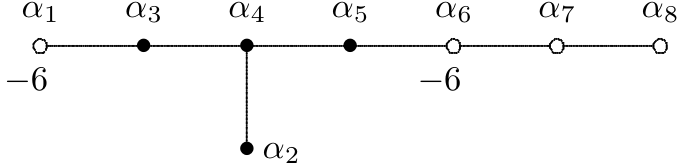} 
\caption{Marked Satake diagram for 
${\bf E}_8/{\bf E}_{7}\times \SL_2$ }
\label{diageix.fig}
\end{figure}
%%%%%%%%%%%%%%%%%%%%%%%%%%%%%%%%%%%%%%%%%

We have $\frc_0 = 0$ by Lemma \ref{rankdim.lem} since $|\Supp \mu_i| = 1$ for
$i = 1, \ldots, 4$. Hence
from the Satake diagram  $\frm = \fso_8$ and
$  h_{\frm}^{0} = 6\alpha_3 + 10\alpha_4 + 6\alpha_2 + 6\alpha_5$,
which gives the indicated markings. Condition (4) of Proposition
\ref{regsimple.prop} is satisfied by $\lambda_1$ and $\lambda_6$, while
condition (1) is satisfied by $\lambda_7$ and $\lambda_8$.

%\input{restrootref}

% revised 9/11/12 for short version

\end{document}

%% file: rest_root_macros.tex
\newcommand{\be}{\ensuremath{\mathbf{e}}}

\newcommand{\bx}{\ensuremath{\mathbf{x}}}
\newcommand{\by}{\ensuremath{\mathbf{y}}}

\newcommand{\fra}{\ensuremath{ \mathfrak a  }}

\newcommand{\frc}{\ensuremath{ \mathfrak c  }}

\newcommand{\frg}{\ensuremath{ \mathfrak g  }}

\newcommand{\frh}{\ensuremath{ \mathfrak h  }}

\newcommand{\frk}{\ensuremath{ \mathfrak k  }}
\newcommand{\frl}{\ensuremath{ \mathfrak l  }}
\newcommand{\frm}{\ensuremath{ \mathfrak m  }}
\newcommand{\frn}{\ensuremath{ \mathfrak n  }}

\newcommand{\frp}{\ensuremath{ \mathfrak p  }}
\newcommand{\frq}{\ensuremath{ \mathfrak q  }}

\newcommand{\frs}{\ensuremath{ \mathfrak s  }}

\newcommand{\frt}{\ensuremath{ \mathfrak t  }}

\newcommand{\frX}{\ensuremath{ \mathfrak X  }}

\newcommand{\fgl}{\ensuremath{ \mathfrak g}{\mathfrak l  }}
\newcommand{\fsl}{\ensuremath{ \mathfrak s}{\mathfrak l  }}
\newcommand{\fso}{\ensuremath{ \mathfrak s}{\mathfrak o  }}

\newcommand{\fsp}{\ensuremath{ \mathfrak s}{\mathfrak p  }}

\newcommand{\bC}{\ensuremath{ \mathbb C  }}

\newcommand{\bP}{\ensuremath{ \mathbb P  }}

\newcommand{\bR}{\ensuremath{ \mathbb R  }}

\newcommand{\bZ}{\ensuremath{ \mathbb Z  }}

\DeclareMathOperator{\ad}{ad}
\DeclareMathOperator{\Ad}{Ad}

\DeclareMathOperator{\diag}{diag}

\DeclareMathOperator{\Supp}{Supp}

\DeclareMathOperator{\GL}{\bf GL}

\DeclareMathOperator{\Hom}{Hom}

\DeclareMathOperator{\Ker}{Ker}

\DeclareMathOperator{\rank}{rank}

\DeclareMathOperator{\Span}{Span}

\DeclareMathOperator{\Spin}{\bf Spin}

\DeclareMathOperator{\SL}{\bf SL}
\DeclareMathOperator{\SO}{\bf SO}
\DeclareMathOperator{\Sp}{\bf Sp}

\DeclareMathOperator{\tr}{tr}

\newtheorem{Theorem}{Theorem}[section]
\newtheorem{thm}{Theorem}[section]

\newtheorem{Proposition}[thm]{Proposition}

\newtheorem{Lemma}[thm]{Lemma}

\newtheorem{Definition}[thm]{Definition}

\newtheorem{Remark}[thm]{Remark}